\documentclass[reqno]{amsart}
\usepackage{amsmath, amsfonts, amsthm, amssymb, setspace, textcomp, bbm, multirow, datetime, mathtools}
\usepackage{geometry}
\usepackage{color}
\newtheorem{theorem}{Theorem}[section]

\newtheorem{corollary}[theorem]{Corollary}
\newtheorem{proposition}[theorem]{Proposition}

\numberwithin{equation}{section}

\newcommand{\aqprod}[3]{\left(#1;#2\right)_{#3}}
\newcommand{\jacprod}[2]{\left[#1;#2\right]_{\infty}}
\newcommand{\TwoTwoMatrix}[4]
{
	\begin{pmatrix}
		#1 & #2\\#3 & #4
	\end{pmatrix}
}
\newcommand{\Ceil}[1]{\left\lceil #1 \right\rceil}
\newcommand{\Floor}[1]{\left\lfloor #1 \right\rfloor}
\newcommand{\Fractional}[1]{\left\{#1\right\} }
\newcommand{\Jac}[2]{\left(\frac{#1}{#2}\right)}

\newcommand{\z}[2]{\zeta_{#1}^{#2}}
\newcommand{\tmu}{\widetilde{\mu}}
\newcommand{\kModular}[3]{\mathfrak{k}_{\left(#1,#2\right)}\left(#3\right)}
\newcommand{\mA}[1]{ \,{^{#1}}\hspace{-2pt}A}

\newcommand{\SLTwo}{\mbox{SL}_2(\mathbb{Z})}
\newcommand{\Z}{\mathbb{Z}}
\newcommand{\G}{\Gamma}

\definecolor{darkBlue}{RGB}{0,0,165}

\author{CHRIS JENNINGS-SHAFFER}
\address{Department of Mathematics, Oregon State University\\
Corvallis, Oregon 97331, USA
\endgraf jennichr@math.oregonstate.edu}

\author{HOLLY SWISHER}
\address{Department of Mathematics, Oregon State University\\
Corvallis, Oregon 97331, USA
\endgraf swisherh@math.oregonstate.edu}

\keywords{Number theory, partitions, overpartitions, ranks, rank differences, Maass forms, modular forms, mock modular forms}

\subjclass[2010]{Primary 11P82, 11F37, Secondary 11P81, 33D99}

\title{Mock modularity of the $M_d$-rank of overpartitions}

\allowdisplaybreaks
\begin{document}


\allowdisplaybreaks

\begin{abstract}
We investigate the modular properties of a new partition rank, the $M_d$-rank
of overpartitions. In fact this is an infinite family of ranks, indexed by
the positive integer $d$, that gives both the Dyson rank of overpartitions and the 
overpartition $M_2$-rank as special cases. The $M_d$-rank of overpartitions is
the holomorphic part of a certain harmonic Maass form of weight $\frac{1}{2}$.
We give the exact transformation of this harmonic Maass form along with a few
identities for the $M_d$-rank.
\end{abstract}

\maketitle

\section{Introduction}
\allowdisplaybreaks

The theory of mock modular forms began with the introduction of mock theta
functions in Ramanujan's last letter to Hardy. 
Currently we understand mock theta functions as special cases of mock modular forms, which we in turn
recognize as the holomorphic parts of harmonic Maass forms. Harmonic Maass forms
are most easily understood as classical half-integer weight modular forms with
relaxed analytic conditions. With the monumental thesis of Zwegers \cite{Zwegers}, 
we can now often take functions we expect to be mock modular, write them in terms of
basic building block functions, and complete them to harmonic Maass forms.
We can then work with these harmonic Maass forms in terms of their building blocks with ease, comparable to working with classical modular forms in terms of Dedekind's $\eta$-function and modular units.

Inherent to studying mock theta functions is the theory of partitions. A partition
of a nonnegative integer $n$ is a nonincreasing sequence of positive integers, called parts, 
that sum to $n$. We let $p(n)$ denote the number of partitions of $n$. 
As an example, $p(4)=5$ as the partitions of $4$ are
$4$, $3+1$, $2+2$, $2+1+1$, and $1+1+1+1$.
This 
definition is deceitful in its simplicity, as illustrated by the fact
that of the following two series, one is the generating function for $p(n)$ and
the other is a third order mock theta function of Ramanujan,
\begin{align*}
	&\sum_{n=0}^\infty
	\frac{q^{n^2}}{(1-q)^2(1-q^2)^2(1-q^3)^2\dotsm (1-q^n)^2},
	&&
	&\sum_{n=0}^\infty
	\frac{q^{n^2}}{(1+q)^2(1+q^2)^2(1+q^3)^2\dotsm (1+q^n)^2}.	
\end{align*}
We can recognize both of these series as special cases of the function $R(z;q)$
given by
\begin{align*}
	R(z;q)
	&=
	\sum_{n=0}^\infty	
	\frac{q^{n^2}}{(1-zq)(1-q/z)(1-zq^2)(1-q^2/z)\dotsm(1-zq^n)(1-q^n/z)}	
	.
\end{align*}
Other choices of $z$ give other third order mock theta functions, so we may call
$R(z;q)$ a universal mock theta function, but it also has another name. The
function $R(z;q)$ is actually the generating function for the rank of partitions
and was introduced by Dyson \cite{Dyson}. The rank of a partition is given by
taking the largest part of the partition and subtracting off the number of parts
of the partition. Now we see partitions and mock modular forms are somehow
linked.

Another partition function is the overpartition function. An overpartition of
$n$ is a partition of $n$ in which the first appearance of a part may be 
overlined. For example the overpartitions of $3$ are
$3$, $\overline{3}$, $2+1$, $2+\overline{1}$, $\overline{2}+1$, $\overline{2}+\overline{1}$,
$1+1+1$, and $\overline{1}+1+1$. We let $\overline{p}(n)$ denote the number of 
overpartitions of $n$. The previous example shows that $\overline{p}(3)=8$.
It turns out this function also fits well into the theory of $q$-series and modular 
forms. It particular the generating function is the rather elegant $\eta$-quotient
$\frac{\eta(2\tau)}{\eta(\tau)^2}$ and the inverse of this function generates 
the sequence of square numbers. This can be compared with the fact that the generating function for 
$p(n)$ is $\frac{1}{\eta(\tau)}$ and the inverse generates the sequence of pentagonal
numbers. Overpartitions also have ranks associated to them. There is the Dyson
rank of an overpartition, which ignores whether or not a part is overlined and
is just the largest part minus the number of parts. There is also the $M_2$-rank of
overpartitions, which is given by 
\begin{align*}
	\Ceil{\frac{l(\pi)}{2}} - \#(\pi) + \#(\pi_o) - \chi(\pi),
\end{align*}
where $l(\pi)$ is the largest part of $\pi$, 
$\#(\pi)$ is the number of parts,
$\#(\pi_o)$ is the number of
odd non-overlined parts, and $\chi(\pi) = 1$ when the largest part is odd and 
non-overlined and $\chi(\pi)=0$ otherwise.
Both of the generating functions of these ranks can also be considered as universal mock theta functions,
as specializations yield known mock theta functions.

In this article we consider a family of rank-type generating functions, the first of which is
the generating function for the Dyson rank for overpartitions and the second of which is the generating function for the $M_2$-rank of overpartitions. Generalizations and families of ranks is not unheard of, for example there is Garvan's $k$-rank \cite{Garvan3}, which is defined by the series
\begin{align*}
	\frac{1}{\aqprod{q}{q}{\infty}}
	\sum_{n=1}^\infty
	\frac{(-1)^{n-1} q^{\frac{n((2k-1)n-1)}{2}} (1-q^n)(1-q^{2n}) }
		{(1-zq^n)(1-z^{-1}q^n)}
	,
\end{align*}
and whose combinatorial interpetation is related to successive Durfee squares.
Also in \cite{Andrews2} Andrews introduced Durfee symbols and $k$-marked Durfee symbols, 
yielding ``Dyson-like'' ranks.

The function we consider is the $M_d$-rank of overpartitions, which is defined 
by the series
\begin{align}
	\label{EqDefinitionOfRank}
	\mathcal{O}_d(z;q)
	&=
		\sum_{n=0}^\infty\sum_{m=-\infty}^\infty M_d(m,n) z^m q^n
	\nonumber\\
	&=
		\frac{\aqprod{-q}{q}{\infty}}{\aqprod{q}{q}{\infty}}	
		\left(
			1
			+
			2\sum_{n=1}^\infty \frac{(1-z)(1-z^{-1})(-1)^n q^{n^2+dn}}
				{(1-zq^{dn})(1-z^{-1}q^{dn})}
		\right)		
	,
\end{align}
where here and throughout the article $q=e^{2\pi i\tau}$
and $d$ is a positive integer.
This series has been studied combinatorially by Morrill in
\cite{Morrill1}. Here we are studying the analytic properties of
$\mathcal{O}_{d}(z;q)$.

At this point we must define some notation.
For complex numbers $a, a_1, \dots,a_k,$ and $q$, with $|q|<1$, and
$n\in \mathbb{N}\cup\{\infty\}$, we let
\begin{align*}
	\aqprod{a}{q}{n}
	&=
	\prod_{k=0}^{n-1} (1-aq^k)
	, \\
	\aqprod{a_1,\dots,a_k}{q}{n}
	&=
	\aqprod{a_1}{q}{n}\dots \aqprod{a_k}{q}{n}	
	,
	\\
	[a_1,\dots,a_k;q]_n
	&=
	\aqprod{a_1,q/a_1, \dots, a_k,q/a_k}{q}{n}	
	.
\end{align*}

Our main result is that specializations of $z$ in $\mathcal{O}_d(z;q)$ 
give mock modular forms of weight 
$\frac{1}{2}$. This is the expected behavior of a rank-type function. In 
particular Bringmann and Ono established this behavior for the rank of partitions
\cite{BringmannOno}, which was further refined by Garvan \cite{Garvan2}.  Bringmann and Lovejoy established this behavior for the Dyson rank of overpartitions \cite{BringmannOnoRhoades}. To this end, we consider the following
function,
\begin{align}\label{DefinitionOfTildeO}
	\widetilde{\mathcal{O}}_d(a,b,c;\tau)
	&=
	\left\{\begin{array}{ll}
		\frac{(1+\z{c}{a}q^{b/c})}{(1-\z{c}{a}q^{b/c})}q^{-\frac{b^2}{c^2d^2}}
		\mathcal{O}_d(\z{c}{a}q^{\frac{b}{c}}; q)
		&\\
		+
		\left(
			\z{c}{a}q^{-\frac{d^2}{4}-\frac{b^2}{c^2d^2}+\frac{b}{c}}
			R\left( \tfrac{2a}{c} + (\tfrac{2b}{c}-d^2)\tau ; 2d^2\tau\right)
			-
			q^{-\frac{b^2}{c^2d^2}}
		\right)
		&
		\mbox{ if } d \mbox{ is odd}		
		,\\[2ex]
		\frac{(1+\z{c}{a}q^{b/c})}{(1-\z{c}{a}q^{b/c})}q^{-\frac{b^2}{c^2d^2}}
		\mathcal{O}_d(\z{c}{a}q^{\frac{b}{c}}; q) 
		&\\
		- 
		\left(
			\z{c}{\frac{a}{2}}q^{-\frac{d^2}{16}-\frac{b^2}{c^2d^2}+\frac{b}{2c}} 
			R\left( \left(\frac{a}{c}  -1\right) + \left(\frac{b}{c}-\frac{d^2}{4}\right)\tau ; \frac{d^2\tau}{2}\right) 
			+
			q^{-\frac{b^2}{c^2d^2}} 
		\right)
		&
		\mbox{ if } d \equiv 2 \pmod 4	
		,\\[2ex]
		\frac{(1+\z{c}{a}q^{b/c})}{(1-\z{c}{a}q^{b/c})}q^{-\frac{b^2}{c^2d^2}}\mathcal{O}_d(\z{c}{a}q^{\frac{b}{c}}; q) 
		&\\
		- 
		\left(
			i\z{c}{\frac{a}{2}}q^{-\frac{d^2}{16}-\frac{b^2}{c^2d^2}+\frac{b}{2c}} 
			R\left( \left(\frac{a}{c} -\frac12 \right) + \left(\frac{b}{c}-\frac{d^2}{4}\right)\tau ; \frac{d^2\tau}{2}\right) 
			+q^{-\frac{b^2}{c^2d^2}} 
		\right)
		&\mbox{ if } d\equiv 0 \pmod4		
		,
	\end{array}
	\right.
\end{align}
where here and throughout the rest of the article $\z{c}{a}=e^{\frac{2\pi ia}{c}}$ and 
the function $R(u;\tau)$ is not the rank as given earlier, but a function of
Zwegers later defined in (\ref{Rdef}).  We now state our main results in two cases, based on the parity of $d$.

\begin{theorem}\label{TheoremOverpartitionDRankModularity}
Suppose $a$, $b$, $c$, and $d$ are integers, 
$d$ is odd, and either $c\nmid 2a$ or $cd\nmid b$. 
Then the
following are true.
\begin{enumerate}
	\item\label{TheoremOverpartitionDRankModularityPartDOddMockModular}
	The function $\widetilde{\mathcal{O}}_d(a,b,c;\tau)$
	is a harmonic Maass form of weight $\frac{1}{2}$ on a certain
	congruence subgroup of $\SLTwo$.
	Furthermore the function
	\begin{align*}
		&
		\frac{(1+\z{c}{a}q^{\frac{b}{c}})}{(1-\z{c}{a}q^{\frac{b}{c}})}  
		q^{\frac{b^2}{c^2d^2}} 
		\mathcal{O}_d(\z{c}{a}q^{\frac{b}{c}};q)
		-
		q^{-\frac{b^2}{c^2d^2}}	
		+
		(-1)^n\z{c}{2na}\delta_{b,cd^2}
		\\&
		+
		2\sum_{m=1}^{|n|}
		(-1)^{m+n}\z{c}{(2n+1-{\rm sgn}(n)(2m-1))a}
		q^{ n\left(\frac{2b}{c}-d^2n-d^2\right) 
			- (2m-1){\rm sgn(n)}\left(\frac{b}{c}-d^2n-\frac{d^2}{2}\right)
			- m(m-1)d^2 - \frac{d^2}{2} + \frac{b}{c} - \frac{b^2}{c^2d^2}    
		}
		,		
	\end{align*}
	where $n=\Floor{\frac{b}{cd^2}}$ and $\delta_{b,cd^2}=1$ if $cd^2\mid b$
	and $\delta_{b,cd^2}=0$ otherwise,
	is the holomorphic part of $\widetilde{\mathcal{O}}_d(a,b,c;\tau)$,
	and as such is a mock modular form.
	In particular $\mathcal{O}_d(\z{c}{a};q)$ is a mock modular form.	
	
	\item\label{TheoremOverpartitionDRankModularityPartDOddDissection}
	We have that
	\begin{align*}
		\frac{(1+\z{c}{a})}{(1-\z{c}{a})}\mathcal{O}_d(\z{c}{a};q)
		-
		\sum_{r=1}^{c-1}(-1)^r\z{c}{-2ar}S_{d}(r,c;2c^2d^2\tau)
		&=		
		\widetilde{\mathcal{O}}_d(a,0,c;\tau)
		-
		\sum_{r=1}^{c-1}(-1)^r\z{c}{-2ar}\widetilde{S}_{d}(r,c;2c^2d^2\tau)
		.
	\end{align*}
	In particular, the function
	\begin{align*}
		\frac{\eta(\tau)^2}{\eta(2\tau)}\left(
			\frac{(1+\z{c}{a})}{(1-\z{c}{a})}\mathcal{O}_d(\z{c}{a};q)
			-
			\sum_{r=1}^{c-1}(-1)^r\z{c}{-2ar}S_{d}(r,c;2c^2d^2\tau)
		\right)
	\end{align*}
	is a weight $1$ modular form (without multiplier) on the congruence subgroup $\Gamma$,
	where 
	\begin{align*}
		\Gamma &=
		\left\{\begin{array}{cl}
			\Gamma_0(2c^2d^2)\cap\Gamma_1(\normalfont{\mbox{lcm}}(c,d))
			&
			\mbox{ if } c\equiv 1 \pmod{2}
			,\\
			\Gamma_0(4c^2d^2)\cap\Gamma_1(\normalfont{\mbox{lcm}}(c,d))
			&
			\mbox{ if } c\equiv 0 \pmod{2}
			.		
		\end{array}
		\right.
	\end{align*}	
\end{enumerate}
\end{theorem}

\begin{theorem}\label{thm:devenMain}
Suppose $a$, $b$, $c$, and $d$ are integers, 
$d$ is even, and either $c\nmid 2a$ or $cd\nmid b$. 
Then the following are true.
\begin{enumerate}
\item\label{devenMockModular}
The function $\widetilde{\mathcal{O}}_d(a,b,c;\tau)$ is a harmonic Maass form of weight $\frac{1}{2}$ on a certain congruence subgroup of $\SLTwo$.  Furthermore the function
\begin{align*}
&\frac{(1+\z{c}{a}q^{\frac{b}{c}})}{(1-\z{c}{a}q^{\frac{b}{c}})}  q^{-\frac{b^2}{c^2d^2}} \mathcal{O}_d(\z{c}{a}q^{\frac{b}{c}};q)
- 
q^{-\frac{b^2}{c^2d^2}} + (-1)^{\frac{dn}{2}}\z{c}{an}\delta_{b,cd^2} 
\\&
+ 2 (-1)^{\delta_d+\frac{dn}{2}} i^{\delta_d} 
	\sum_{m=1}^{|n|} (-1)^{m} i^{\delta_d(2m-1){\rm sgn}(n)} \zeta_{2c}^{a(2n+1-{\rm sgn}(n)(2m-1))} 
	\\&\qquad\times	
	q^{n(\frac{b}{c} - \frac{d^2n}{4} -\frac{d^2}{4}) - (2m-1){\rm sgn}(n)(\frac{b}{2c} - \frac{d^2n}{4} - \frac{d^2}{8}) 
	-\frac{d^2}{4} m(m-1) -\frac{d^2}{8} + \frac{b}{2c} - \frac{b^2}{c^2d^2}}	
\end{align*}
where $n=\Floor{\frac{2b}{cd^2}}$, $\delta_{b,cd^2}=1$ if $cd^2\mid b$ and 
$\delta_{b,cd^2}=0$ otherwise, and $\delta_d=1$ if $d\equiv 0 \pmod 4$ and 
$\delta_d=0$ if $d\equiv 2 \pmod 4$, is the holomorphic part of 
$\widetilde{\mathcal{O}}_d(a,b,c;\tau)$, and as such is a mock modular form.  
In particular $\mathcal{O}_d(\z{c}{a};q)$ is a mock modular form.	
	
\item\label{devenDissection}
We have that
\begin{align*}
\frac{(1+\z{c}{a})}{(1-\z{c}{a})}\mathcal{O}_d(\z{c}{a};q) 
- 
\sum_{r=1}^{c-1}(-1)^{\frac{dr}{2}}\z{c}{-ar}S_{d}\left(r,c; \frac{c^2d^2\tau}{2}\right) 
&=	 
\widetilde{\mathcal{O}}_d(a,0,c;\tau) 
\\&\quad
- 
\sum_{r=1}^{c-1}(-1)^{\frac{dr}{2}} \z{c}{-ar}\widetilde{S}_{d}\left(r,c; \frac{c^2d^2\tau}{2}\right).
\end{align*}
In particular, the function
\[
\frac{\eta(\tau)^2}{\eta(2\tau)}\left(
	\frac{(1+\z{c}{a})}{(1-\z{c}{a})}\mathcal{O}_d(\z{c}{a};q) 
	- 
	\sum_{r=1}^{c-1}(-1)^{\frac{dr}{2}}\z{c}{-ar}S_{d}\left(r,c; \frac{c^2d^2\tau}{2}\right)
\right)
\]
is a weight $1$ modular form (without multiplier) on the congruence subgroup $\Gamma$, where
\begin{align*}
	\Gamma &=
	\left\{\begin{array}{cl}
		\Gamma_0(c^2d^2)\cap\Gamma_1(\normalfont{\mbox{lcm}}(c,2d))
		&
		\mbox{ if } d\equiv 2 \pmod{4}
		,\\
		\Gamma_0( \normalfont{\mbox{lcm}}(2,c)^2 d^2)\cap\Gamma_1(\normalfont{\mbox{lcm}}(2c,2d))
		&
		\mbox{ if } d\equiv 0 \pmod{4}
		.		
	\end{array}
	\right.
\end{align*}	
\end{enumerate}
\end{theorem}

The divisibility conditions $c\nmid 2a$ or $cd\nmid b$ in Theorems \ref{TheoremOverpartitionDRankModularity} and \ref{thm:devenMain} are to assure
necessary functions defined later on are well defined and do not have any poles. 
The function $S_{d}(r,c;\tau)$ is a certain Lambert series defined in 
\eqref{EqDefinitionOfS} and $\widetilde{S}_{d}(r,c;\tau)$ is a certain 
harmonic Maass form defined in \eqref{EqDefinitionOfSTilde}. 
The point of interest with part 2 of Theorems
\ref{TheoremOverpartitionDRankModularity} and \ref{thm:devenMain}
is that they provide the necessary
information to determine the $c$-dissection of $\mathcal{O}_{d}(\z{c}{a};q)$.
This is a standard question for ranks. Identities equivalent to the $5$ and
$7$ dissections of the rank of partitions at $z=\z{5}{}$ and $z=\z{7}{}$ were
given by Atkin and Swinnerton-Dyer \cite{AtkinSwinnerton-Dyer},
identities equivalent to the $3$ and $5$ dissections of both the overpartition Dyson rank and 
overpartition $M_2$-rank at $z=\z{3}{}$ and $z=\z{5}{}$ were given by Lovejoy and Osburn 
\cite{LovejoyOsburn1,LovejoyOsburn2}, and the first author gave the $7$ 
dissection of the overpartition rank at $z=\z{7}{}$ \cite{JenningsShaffer2}.
To demonstrate the utility of our theorems, we deduce the $3$-dissection of 
$\mathcal{O}_{3}(\z{3}{};q)$ in the theorem below.

\begin{theorem}\label{TheoremOverpartition3Rank3Dissection}
Suppose $\z{3}{}$ is a primitive third root of unity. Then
\begin{align*}
	\mathcal{O}_3(\z{3}{};q)
	&=
	A_{0}(q^3) + qA_{1}(q^3) + q^2A_{2}(q^3)
,
\end{align*}
where
\begin{align*}
	A_0(q)
	&=
		-\frac{6q^{15}}{\aqprod{q^{54}}{q^{54}}{\infty}\jacprod{q^{27}}{q^{54}}}
		\sum_{n=-\infty}^\infty
		\frac{ (-1)^n q^{27n^2+54n} }{1-q^{54n+18}}
		\\&\quad
		+
		\frac{\aqprod{q^6}{q^6}{\infty} \aqprod{q^{54}}{q^{54}}{\infty}^2   }
		{\aqprod{q}{q}{\infty} \aqprod{q^3}{q^3}{\infty} \jacprod{q}{q^6}}
		\Bigg(
			-3\frac{\jacprod{q^{9},q^{12},q^{21},q^{27}}{q^{54}}}{\jacprod{q^{3},q^{24}}{q^{54}}}
			-
			6q\frac{\jacprod{q^{9},q^{12},q^{15},q^{27}}{q^{54}}}{\jacprod{q^{3},q^{24}}{q^{54}}}
			\\&\quad		
			+
			4\frac{\jacprod{q^{6},q^{12},q^{18},q^{24},q^{27}}{q^{54}}}{\jacprod{q^{3},q^{15},q^{21}}{q^{54}}}
			+
			3q^{3}\frac{\jacprod{q^{6},q^{9},q^{15},q^{27}}{q^{54}}}{\jacprod{q^{3},q^{24}}{q^{54}}}
			-
			2q^{3}\frac{\jacprod{q^{6},q^{9},q^{12},q^{18},q^{24}}{q^{54}}}{\jacprod{q^{3},q^{15},q^{21}}{q^{54}}}
			\\&\quad		
			+
			12q\frac{\jacprod{q^{9},q^{15},q^{18}}{q^{54}}}{\jacprod{q^{6}}{q^{54}}}
			-
			6q\frac{\jacprod{q^{9},q^{12},q^{27}}{q^{54}}}{\jacprod{q^{6}}{q^{54}}}
			+
			12q\frac{\jacprod{q^{15},q^{21},q^{24}}{q^{54}}}{\jacprod{q^{18}}{q^{54}}}
			\\&\quad		
			-
			6q\frac{\jacprod{q^{9},q^{24},q^{27}}{q^{54}}}{\jacprod{q^{12}}{q^{54}}}
			+
			12q^{4}\frac{\jacprod{q^{9},q^{15},q^{21},q^{24}}{q^{54}}}{\jacprod{q^{18},q^{27}}{q^{54}}}
		\Bigg)
	,\\
	A_1(q)
	&=
		\frac{\aqprod{q^6}{q^6}{\infty} \aqprod{q^{54}}{q^{54}}{\infty}^2   }
		{\aqprod{q}{q}{\infty} \aqprod{q^3}{q^3}{\infty} \jacprod{q}{q^6}}
		\Bigg(
			6\frac{\jacprod{q^{6},q^{15},q^{21},q^{27}}{q^{54}}}{\jacprod{q^{3},q^{24}}{q^{54}}}
			-
			4\frac{\jacprod{q^{6},q^{12},q^{18},q^{24}}{q^{54}}}{\jacprod{q^{3},q^{15}}{q^{54}}}
			\\&\quad
			-
			2q\frac{\jacprod{q^{6},q^{12},q^{18},q^{24}}{q^{54}}}{\jacprod{q^{3},q^{21}}{q^{54}}}
			+
			6q\frac{\jacprod{q^{9},q^{12},q^{21}}{q^{54}}}{\jacprod{q^{6}}{q^{54}}}
			+
			2q^{2}\frac{\jacprod{q^{15},q^{18},q^{21}}{q^{54}}}{\jacprod{q^{24}}{q^{54}}}
			\\&\quad			
			+
			4q^{2}\frac{\jacprod{q^{9},q^{15},q^{24}}{q^{54}}}{\jacprod{q^{12}}{q^{54}}}
			-
			2q^{5}\frac{\jacprod{q^{9},q^{12},q^{21}}{q^{54}}}{\jacprod{q^{24}}{q^{54}}}
		\Bigg)
	,\\
	A_2(q)
	&=
		\frac{\aqprod{q^6}{q^6}{\infty}^3 \aqprod{q^{54}}{q^{54}}{\infty}^2   }
		{\aqprod{q}{q}{\infty} \aqprod{q^3}{q^3}{\infty}^3 }
		\bigg(
			12\frac{\jacprod{q^{6},q^{15},q^{21},q^{27}}{q^{54}}}{\jacprod{q^{3},q^{24}}{q^{54}}}
			-
			8\frac{\jacprod{q^{6},q^{12},q^{18},q^{24}}{q^{54}}}{\jacprod{q^{3},q^{15}}{q^{54}}}
			\\&\quad	
			-
			4q\frac{\jacprod{q^{6},q^{12},q^{18},q^{24}}{q^{54}}}{\jacprod{q^{3},q^{21}}{q^{54}}}
			+
			12q\frac{\jacprod{q^{9},q^{12},q^{21}}{q^{54}}}{\jacprod{q^{6}}{q^{54}}}
			+
			4q^{2}\frac{\jacprod{q^{15},q^{18},q^{21}}{q^{54}}}{\jacprod{q^{24}}{q^{54}}}
			\\&\quad			
			+
			8q^{2}\frac{\jacprod{q^{9},q^{15},q^{24}}{q^{54}}}{\jacprod{q^{12}}{q^{54}}}
			-
			4q^{5}\frac{\jacprod{q^{9},q^{12},q^{21}}{q^{54}}}{\jacprod{q^{24}}{q^{54}}}
		\Bigg)
	.
\end{align*}
\end{theorem}
The rest of the article is organized as follows. In Section 2 we recall several 
useful modular forms, state the definition of a harmonic Maass form, and 
introduce the functions of Zwegers that are associated to harmonic Maass forms. 
In Section 3 we give a series of identities to rewrite 
$\mathcal{O}_{d}(z;q)$ in terms of known mock modular forms and recognize
$\widetilde{\mathcal{O}}_{d}(a,b,c;\tau)$ as the modular completion. In Section
4 we work out modular transformation formulas for various functions associated
to $\widetilde{\mathcal{O}}_{d}(a,b,c;\tau)$. In Section 5 we define the functions
$S(r,c;\tau)$ and $\widetilde{S}(r,c;\tau)$, and work out their modular 
properties. In Section 6 we give the proofs of Theorems \ref{TheoremOverpartitionDRankModularity}
and \ref{thm:devenMain}. In Section 7 we prove Theorem \ref{TheoremOverpartition3Rank3Dissection},
which is the $3$-dissection of $\mathcal{O}_{3}(\z{3}{};q)$. Lastly, in
Section 8, we give two additional results. The first result is that the
$q^{2n+1}$ terms of $\mathcal{O}_{4}(z;q)$ can be written as an infinite
product. The second result is that certain differences of $\mathcal{O}_{d}(z;q)$
and $\mathcal{O}_{2d}(z;q)$ will always be modular, rather than mock modular.

\section{Preliminaries}

In order to prove our main results, we require some background information about modular forms and harmonic Maass forms.  We will assume the reader is familiar with classical modular forms; for more detailed background see \cite{DS, Ono1}.

\subsection{Modular Forms}

One of the most famous modular forms is Dedekind's eta function, defined for $\tau \in \mathcal{H}$ 
(the complex upper half-plane) by 
\[
\eta(\tau) = q^{\frac{1}{24}}\aqprod{q}{q}{\infty}.
\]
For a matrix  $A=\begin{psmallmatrix}\alpha&\beta\\\gamma&\delta\end{psmallmatrix}\in\SLTwo$,
$\eta(A\tau)$ transforms as
\begin{equation}\label{eq:eta_mult}
\eta(A\tau) = \nu(A)\sqrt{\gamma\tau+\delta} \,\eta(\tau),
\end{equation}
where $\nu(A)$ is a $24$th root of unity which can be described in terms of Dedekind sums.

We will also make use of two particular families of modular forms, the Klein forms $\kModular{a_1}{a_2}{\tau}$ described in \cite{KubertLang} and the the generalized eta functions $f_{N,\rho}(\tau)$ studied by Biagioli \cite{Biagioli}.  We describe these two families below.

First, for $a_1,a_2 \in \mathbb{Q}$ the Klein forms $\kModular{a_1}{a_2}{\tau}$ are given by
\begin{align*}
	\kModular{a_1}{a_2}{\tau}
	&=
	-q^{\frac{1}{2}B_2(a_1)-\frac{1}{12} }\exp(\pi ia_2(a_1-1) )\frac{\jacprod{\zeta}{q}}{\aqprod{q}{q}{\infty}^2}	
	,
\end{align*}
where $B_2(x) = x^2-x+\frac{1}{6}$, and 
$\zeta = \exp(2\pi i(a_1\tau + a_2))$. 
For  $A=\begin{psmallmatrix}\alpha&\beta\\\gamma&\delta\end{psmallmatrix}\in\SLTwo$, we have that
\begin{align*}
	\kModular{a_1}{a_2}{A\tau}
	&=
	(\gamma\tau+\delta)^{-1}\kModular{a_1\alpha+a_2\gamma}{a_1\beta+a_2\delta}{\tau}
,
\end{align*}
and for integers $b_1$ and $b_2$
\begin{align}\label{EqTShift}
	\kModular{a_1+b_1}{a_2+b_2}{\tau}
	&=
	(-1)^{b_1b_2+b_1+b_2}\exp(-\pi i(b_1a_2-b_2a_1)) \kModular{a_1}{a_2}{\tau}	
.
\end{align}
Additionally, $\kModular{a_1}{a_2}{\tau}$ is holomorphic on $\mathcal{H}$ and has
no zeros nor poles on $\mathcal{H}$. Thus $\kModular{a_1}{a_2}{\tau}$ is a
modular form of weight $-1$ on some congruence subgroup of $\SLTwo$.

Next, as in work of Biagioli \cite{Biagioli}, we consider the generalized eta functions given by
\begin{align}\label{eqn:f_Ndef}
	f_{N,\rho}(\tau)
	&=
		q^{\frac{(N-2\rho)^2}{8N}}\aqprod{q^\rho,q^{N-\rho},q^N}{q^N}{\infty}
	,
\end{align}
where $N$ and $\rho$ are integers, $N$ is positive, and $N\nmid \rho$.
We have that 
\begin{equation}\label{eq:f_mod}
f_{N,\rho}(\tau)=f_{N,-\rho}(\tau)=f_{N,\rho+N}(\tau).
\end{equation} 
Moreover (see \cite[Lemma 2.1]{Biagioli}) for $
A=\begin{psmallmatrix}\alpha&\beta\\\gamma&\delta\end{psmallmatrix}\in\Gamma_0(N)$ 
we have
\begin{align}\label{EqBiagioliTransform}
	f_{N,\rho}(A\tau)
	&=
		(-1)^{\rho\beta + \Floor{\frac{\rho\alpha}{N}} + \Floor{\frac{\rho}{N}} }
		\exp\left(\frac{\pi i\alpha\beta\rho^2}{N}\right)
		\nu( \mA{N} )^{3}
		\sqrt{\gamma\tau+\delta}\,		
		f_{N,\alpha\rho}(\tau)
.
\end{align}
Here and the throughout the article, the matrix $\mA{m}$
is defined as
\begin{align*}
 	\mA{m} & = \TwoTwoMatrix{\alpha}{m\beta}{\gamma/m}{\delta}
	.
\end{align*}
The utility of the notation $\mA{m}$ is in the fact that $mA\tau=\mA{m}(m\tau)$.
Additionally, $f_{N,\rho}(\tau)$ is holomorphic on $\mathcal{H}$ and has
no zeros nor poles on $\mathcal{H}$. Thus $f_{N,\rho}(\tau)$ is a
modular form of weight $\frac{1}{2}$ on some subgroup of $\SLTwo$.

It will also be useful for us to recall some important facts about the 
multiplier systems for certain modular forms. 
Recall the $\eta$-multiplier $\nu(A)$, defined by \eqref{eq:eta_mult}.  
Noting for odd $m$ we have that $\frac{\eta(2m^2\tau)^3}{\eta(2\tau)^3}$ is
a modular form of weight zero on $\Gamma_0(2m^2)$ (see \cite[Thm. 1.64]{Ono1}), 
we see that $\nu( \mA{2m^2})^3=\nu(\mA{2})^3$ on $\Gamma_0(2m^2)$. We will use this
fact often without mention.
A convenient form for the $\nu(A)$, is given by \cite[Ch. 4, Thm. 2]{Knopp}
\begin{align}
	\label{EqEtaMultipler}
	\nu(A)
	&=
	\left\{
	\begin{array}{ll}
		\big(\frac{\delta}{|\gamma|} \big)
		\exp\left(\frac{\pi i}{12}\left(
			(\alpha+\delta)\gamma - \beta\delta(\gamma^2-1) - 3\gamma		
		\right)\right)
		&
		\mbox{ if } \gamma \equiv 1 \pmod{2},
		\\				
		\Jac{\gamma}{\delta}
		\exp\left(\frac{\pi i}{12}\left(
			(\alpha+\delta)\gamma - \beta\delta(\gamma^2-1) + 3\delta - 3 - 3\gamma\delta		
		\right)\right)
		&
		\mbox{ if } \delta\equiv 1\pmod{2},
	\end{array}
	\right.
\end{align}
where $\Jac{\gamma}{\delta}$ is the generalized Legendre symbol as in
\cite{Shimura1}.
One can verify that when $m$ is even, 
$\nu(\mA{2m^2})^3=i^{-\alpha\beta}\exp\left(-\frac{\pi i\gamma\delta}{8m^2}\right)\nu(\mA{2})^3$
for $A\in\Gamma_0(4m^2)$, and so 
$\nu(\mA{2m^2})^3=i^{-\alpha\beta}\nu(\mA{2})^3$
for $A\in\Gamma_0(16m^2)$.

We will also make use of the multiplier system for the eta quotient $\frac{\eta(2\tau)}{\eta(\tau)^2}$, which is the generating function for overpartitions.  This is given in the following proposition.

\begin{proposition}\label{prop:eta_transf}
For $A=\begin{pmatrix}\alpha & \beta \\ \gamma & \delta \end{pmatrix}\in\Gamma_0(2)$ we have
\begin{align*}
	\frac{\eta(2A\tau)}{\eta(A\tau)^2}
	&=
	(-1)^{\beta+\frac{\alpha-1}{2}} i^{-\alpha\beta}
	\nu( \mA{2} )^{-3}
	(\gamma\tau+\delta)^{-\frac{1}{2}}
	\frac{\eta(2\tau)}{\eta(\tau)^2}
	.	
\end{align*}
\end{proposition}
\begin{proof}
We begin by noting that
\begin{align*}
	\frac{\eta(2\tau)}{\eta(\tau)^2}
	&=
	\frac{1}{f_{2,1}(\tau)}
	.
\end{align*}
However by (\ref{EqBiagioliTransform}) we have
\begin{align*}
	f_{2,1}(A\tau)
	&=
		(-1)^{\beta+\frac{\alpha-1}{2}}i^{\alpha\beta}\nu(\mA{2})^3 
		\sqrt{\gamma\tau+\delta} f_{2,1}(\tau)
	.
\end{align*}
\end{proof}

\subsection{Harmonic Maass Forms}
Here we recall some basic facts about harmonic Maass forms. We begin with
their definition. This definition is essentially that of Bruinier
and Funke in \cite{BF}. First we recall that for $k\in\frac{1}{2}\Z$, the weight
$k$ hyperbolic Laplacian operator $\Delta_{k}$ is defined as
$$
\Delta_k = -y^2\left(\frac{\partial^2}{\partial x^2}+\frac{\partial^2}{\partial y^2} \right)
+iky\left(\frac{\partial}{\partial x}+i\frac{\partial}{\partial y}\right),
$$
where here and throughout the article $\tau=x+iy$.

Given $k\in \frac{1}{2}\Z$, a finite index subgroup $\Gamma\subset\SLTwo$, and 
a multiplier $\psi$ (so $\psi:\Gamma\rightarrow\mathbb{C}$ with $|\psi|=1$),
a \emph{harmonic Maass form} of weight $k$ for the subgroup $\Gamma$, with multiplier 
$\psi$, is any smooth function $\widehat{f}:\mathcal{H}\rightarrow\mathbb{C}$
such that the following three properties hold.
\begin{enumerate}
\item
For all $A=\begin{psmallmatrix}a & b\\c & d\end{psmallmatrix} \in \Gamma$ and 
$\tau \in \mathcal{H}$ we have $\widehat{f}(A \tau) = \psi(A)(c\tau+d)^{k} \widehat{f}(\tau)$. 
\item
We have that $\Delta_k(\widehat{f})=0$. 
\item
The function $\widehat{f}$ has linear exponential growth at the cusps in the 
following form. For each 
$A=\begin{psmallmatrix}a & b\\c & d\end{psmallmatrix}\in\SLTwo$, there exists a polynomial 
$p_{A}(\tau)\in \mathbb{C}[q^{-\frac{1}{N}}]$ (with $N$ a positive integer) such that 
$(c\tau+d)^{-k}\widehat{f}(\tau) - p_{A}(\tau) = O(e^{-\epsilon y})$ as $y\rightarrow \infty$ 
for some $\epsilon > 0$. 
\end{enumerate}

Throughout the article, we shall use the phrase ``linear exponential growth'' to mean this 
more restrictive condition in (3) rather than the more relaxed condition $O(e^{\epsilon y})$. We note that this relaxed condition is instead used for ``harmonic Maass forms of moderate growth''.
The Fourier series of a harmonic Maass form $\widehat{f}$ of weight $k$ naturally 
decomposes as the sum of a holomorphic and a non-holomorphic part.  Using 
terminology of Zagier \cite{Zagier}, the holomorphic part $f$ is called a 
\emph{mock modular form} of weight $k$.  
Furthermore, a harmonic Maass form $\widehat{f}$ of weight $k$ is mapped to a 
classical modular form of weight $2-k$ by the differential operator
$\xi_{k}=2iy^{k}\cdot\overline{\frac{\partial}{\partial{\overline{\tau}}}}$,
which is called the \emph{shadow map}. Here the image of $\widehat{f}$ 
under $\xi_{k}$ is called the \emph{shadow} of $f$.
In the special case when $k=\frac12$ and the shadow is a 
weight $\frac{3}{2}$ unary theta function,
$f$ (the holomorphic part) is called a \emph{mock theta function}.

Zwegers \cite{Zwegers} gave a method for constructing mock theta functions with
shadow related to the function $g_{a,b}(\tau)$, which is defined for $a,b\in\mathbb{R}$ by
\begin{align*}
	g_{a,b}(\tau) 
	&=
	\sum_{n\in\mathbb{Z}+a}
	n\exp( \pi in^2\tau + 2\pi inb )
.
\end{align*}
Zwegers did this by considering functions $\mu(u,v;\tau)$, which he defines for $u,v,z\in \mathbb{C}$, $\tau\in\mathcal{H}$, and
$u,v\not\in\mathbb{Z}+\tau\mathbb{Z}$ by
\begin{equation}\label{def_mu}
\mu(u,v;\tau) = \frac{\exp(\pi iu)}{\vartheta(v;\tau)} \sum_{n=-\infty}^\infty \frac{ (-1)^n \exp(\pi in(n+1)\tau + 2\pi inv) }{ 1 - \exp(2\pi in\tau + 2\pi iu ) }, 
\end{equation}
where
\begin{equation}\label{def_varTheta} 
\vartheta(z;\tau) 
= 
	\sum_{n\in \mathbb{Z}+\frac{1}{2}} \exp\left( \pi in^2\tau + 2\pi in\left(z+\frac{1}{2}\right) \right)
=
	-iq^{\frac{1}{8}}e^{-\pi iz}\aqprod{e^{2\pi iz},e^{-2\pi iz}q,q}{q}{\infty}
	.
\end{equation}

We note that $\mu(u,v;\tau)$ satisfies the following transformation formula (see \cite[Prop. 1.4]{Zwegers})
\begin{equation}\label{EqMuTranformation}
\mu(u,v;\tau) = -e^{2\pi i(u-v) - \pi i\tau}\mu(u,v+\tau;\tau) - ie^{ \pi i(u-v) - \frac{\pi i\tau}{4}},
\end{equation}
and $\vartheta(z;\tau)$ satisfies the transformation formula (see \cite[(80.31)]{Rademacher})
\begin{equation}\label{EqVarThetaTranformation}
\vartheta(z;\tau) = -e^{\pi i\tau + 2\pi iz}\vartheta(z+\tau;\tau).
\end{equation}

For $u,z\in\mathbb{C}$, $y=\mbox{Im}(\tau)$, and $a=\mbox{Im}(u)/\mbox{Im}(\tau)$, define
\begin{equation} \label{Rdef}
R(u;\tau) = \sum_{n\in \mathbb{Z}+\frac{1}{2}}\left( \mbox{sgn}(n) - E( (n+a)\sqrt{2y} )  \right)(-1)^{n-\frac{1}{2}}\exp( -\pi in^2\tau - 2\pi inu ), 
\end{equation}
where
\begin{equation}\label{eqn:Edef} 
E(z)  = 2\int_0^z \exp(-\pi w^2)dw.
\end{equation}

Then $\mu(u,v;\tau)$ can be completed to $\tilde{\mu}(u,v;\tau)$, for 
$u,v\notin\mathbb{Z}+\tau\mathbb{Z}$, by defining 
\begin{align*}
	\tilde{\mu}(u,v;\tau)
	&=
		\mu(u,v;\tau) + \frac{i}{2}R(u-v;\tau)	
	.
\end{align*}
The following essential properties are from \cite[Theorem 1.11]{Zwegers}.  If $k,l,m,n$ are integers then
\begin{align}
	\label{EqZTheorem1.11P1}
	\tmu\left(u+k\tau+l, v+m\tau+n; \tau \right)
	&=
	(-1)^{k+l+m+n}
	\exp\left(
		\pi i\tau(k-m)^2 + 2\pi i(k-m)(u-v)	
	\right)
	\tmu(u,v;\tau)
	.
\end{align}
Moreover if $A=\TwoTwoMatrix{\alpha}{\beta}{\gamma}{\delta}\in\SLTwo$, then
\begin{align}
	\label{EqZTheorem1.11P2}
	\tmu\left(
		\frac{u}{\gamma\tau+\delta},
		\frac{v}{\gamma\tau+\delta};
		A\tau		
	\right)
	&=
	\nu(A)^{-3}
	\exp\left(
		-\frac{\pi i\gamma(u-v)^2}{\gamma\tau+\delta}
	\right)	
	\sqrt{\gamma\tau+\delta}\,
	\tmu\left(u, v;\tau\right)	
.
\end{align}

Zwegers showed (see \cite[Theorem 1.16 (1)]{Zwegers}) that
\begin{align}\label{EqRToIntegral}
	R(a\tau-b;\tau)
	&=
	-\exp(\pi ia^2\tau - 2\pi ia(b+\tfrac{1}{2} )  )
	\int_{-\overline{\tau}}^{i\infty}
	\frac{g_{a+ \frac{1}{2} , b+ \frac{1}{2} }(z)}{\sqrt{-i(z+\tau)}}dz
,
\end{align}
for $a\in(-\frac{1}{2},\frac{1}{2})$. 
Following the proof we find
that for $a=-\frac{1}{2}$ we instead have
\begin{align}\label{EqRToIntegralBoundaryCase}
	R(-\tfrac{\tau}{2}-b;\tau)
	&=
	\exp( \tfrac{\pi i\tau}{4}  +\pi ib  )
	-
	\exp( \tfrac{\pi i\tau}{4}  +\pi i(b+\tfrac{1}{2} )  )
	\int_{-\overline{\tau}}^{i\infty}
	\frac{g_{0, b+\frac{1}{2} }(z)}{\sqrt{-i(z+\tau)}}dz
.
\end{align}

Moreover, Zwegers \cite[Proposition 1.9]{Zwegers} states that
\begin{align}\label{eqn:prop1.9Z}
	R(u;\tau) + \exp\left( -2\pi iu-\pi i\tau\right) R(u+\tau;\tau)
	&=
		2\exp\left( -\pi iu - \tfrac{\pi i\tau}{4} \right)
	,\\
	R(-u;\tau)
	&=
		R(u;\tau) \nonumber
	.
\end{align}
From \eqref{eqn:prop1.9Z} we deduce that for $n\in\mathbb{Z}$
\begin{multline}  \label{eq:Rshift}
R(u+n\tau;\tau)
=
	(-1)^{n}\exp\left(2\pi inu + \pi in^2\tau \right)
	R(u;\tau) 
	\\ 
	+
	2\exp\left( 2\pi inu + \pi in^2\tau \right)
	\sum_{m=1}^{|n|}
	(-1)^{m+n}\exp\left(
		-\pi i(2m-1)\mbox{sgn}(n)u - \pi im(m-1)\tau - \frac{\pi i\tau}{4}		
	\right)	
.
\end{multline}

One can check that for $a\in\mathbb{Q}$,
\begin{align}\label{eqn:growth_check}
R(a\tau-b) &= p(\tau) + O(y^{-\frac{1}{2}}e^{-\pi a^2y - \epsilon y}),
\end{align}
as $y\rightarrow\infty$, where $\epsilon>0$ and $p(\tau)$ is some 
rational function of a fractional power of $q$. As such,
for $u_1,u_2,v_1,v_2\in\mathbb{Q}$, we find that
$q^{-\frac{(u_1-v_1)^2}{2}}\tmu\left( u_1\tau+u_2,v_1\tau+v_2  ;\tau\right)$
meets the prescribed growth conditions of a harmonic Maass form.

\subsection{Invariant orders at cusps}\label{sec:invariant}

We recall for a modular form $f$ on some congruence subgroup $\Gamma$, 
the invariant order at
$\infty$ is the least power of $q$ appearing in the $q$-expansion at 
$i\infty$. That is, if
\begin{align*}
	f(\tau) &= \sum_{m=m_0}^\infty a(m)\exp(2\pi i \tau m/N )
,
\end{align*}
and $a(m_0)\not=0$, then the invariant order is $m_0/N$. For a modular form,
this is always a finite number.
For a harmonic Maass form, we cannot take such an expansion,
however we can do so for the holomorphic part.
If $f$ is a modular form of weight $k$,
$\gcd(\alpha,\gamma)=1$, and 
$A=\begin{psmallmatrix}\alpha&\beta\\\gamma&\delta\end{psmallmatrix}\in\SLTwo$,
then the invariant order of $f$ at the cusp $\frac{\alpha}{\gamma}$ is the invariant
order at $\infty$ of $(\gamma\tau+\delta)^{-k}f(A\tau)$. In the same fashion, 
if $f$ is a harmonic Maass form, then we define the invariant order of the
holomorphic part of $f$ at the cusp $\frac{\alpha}{\gamma}$ as the  is the invariant
order at $\infty$ of $(\gamma\tau+\delta)^{-k}f(A\tau)$. This value is independent of the choice of
$A$.


For a real number $w$, we let $\Floor{w}$ denote the greatest 
integer less than or equal to $w$ and $\Fractional{w}$ the fractional part of 
$w$. That is, 
$w=\Floor{w}+\Fractional{w}$, $\Floor{w}\in\mathbb{Z}$, and 
$0\le\Fractional{w}<1$.
We will make use of the following corollary in which a lower bound is 
established for the invariant order of the holomorphic part of a harmonic Maass 
form at $i\infty$. 

\begin{corollary}[\cite{JenningsShaffer2} Cor. 6.2]\label{PropTMuOrders}
If $f(\tau)=q^{\alpha}\tmu(u_1\tau+u_2,v_1\tau+v_2;\tau)$ is a harmonic Maass form,
with $u_i,v_i\in\mathbb{R}$, then the lowest power of 
$q$ appearing in the expansion of the holomorphic part of
$f(\tau)$ is at least $\alpha+\tilde{\nu}(u_1,v_1)$, where
\begin{align*}
	\tilde{\nu}(u,v)
	&=
		\frac{1}{2}\left( \Floor{u}-\Floor{v} \right)^2
		+
		\left(\Floor{u}-\Floor{v}\right)\left(\Fractional{u}-\Fractional{v}\right)		
		+
		k( u , v )
	,
	\\
	k(u, v)
	&=
		\left\{
		\begin{array}{cc}
			\vspace{5pt}
			\nu(\Fractional{u},\Fractional{v}) 
				& \mbox{ if } \Fractional{u}-\Fractional{v}\not=\pm\frac{1}{2} ,
			\\
			\min\left( \frac{1}{8},\nu(\Fractional{u},\Fractional{v}) \right)  
				&\mbox{ if } \Fractional{u}-\Fractional{v}=\pm\frac{1}{2} ,	
		\end{array}\right.
	\\
	\nu(u,v)
	&=
		\left\{
		\begin{array}{cc}
			\vspace{5pt}			
			\frac{u+v}{2}-\frac{1}{8} & \mbox{ if } u+v \le 1 ,
			\\
			\frac{7}{8}-\frac{u+v}{2} & \mbox{ if } u+v > 1.
		\end{array}\right.
\end{align*}
\end{corollary}

We also recall the following result of Biagioli \cite{Biagioli} which gives the 
invariant orders of $f_{N,\rho}(\tau)$.
\begin{proposition}[\cite{Biagioli} Lemma 3.2]\label{PropBiagioliInvariantOrder}
Suppose $N,\rho,\alpha, \gamma$ are integers such that $N$ is positive, $N\nmid\rho$, and $\gcd(\alpha,\gamma)=1$. Then
the invariant order of $f_{N,\rho}(\tau)$
at the cusp $\frac{\alpha}{\gamma}$ is
\begin{align*}
	&
		\frac{\gcd(N,\gamma)^2}{2N}
		\left(
			\Fractional{ \tfrac{\alpha\rho}{\gcd(N,\gamma)} } -\frac{1}{2}		
		\right)^2
.
\end{align*}
\end{proposition}


In order to prove Theorem \ref{TheoremOverpartition3Rank3Dissection}, we will 
use the valence formula for modular functions which is described as follows.
Suppose $f$ is a modular function on some congruence subgroup 
$\Gamma\subset\SLTwo$.
Suppose $A=\begin{psmallmatrix}\alpha&\beta\\\gamma&\delta\end{psmallmatrix} \in\SLTwo$,
we then have a cusp $\zeta=A(\infty)=\frac{\alpha}{\gamma}$.
We let $ord(f;\zeta)$ denote the invariant order of $f$ at
$\zeta$. We define the width of $\zeta$ with respect to $\Gamma$
as $width_\Gamma(\zeta):=w$, where $w$ is the least positive integer such that
$A\begin{psmallmatrix}1&w\\0&1\end{psmallmatrix}A^{-1}\in\Gamma$. We then define
the order of $f$ at $\zeta$ with respect to $\Gamma$ as
$ORD_\Gamma(f;\zeta)=ord(f;\zeta) width_\Gamma(\zeta)$.
For $z\in\mathcal{H}$ we let $ord(f;z)$ denote the order of $f$ at $z$ as a
meromorphic function. We then define the order of $f$ at $z$ with respect to
$\Gamma$ as $ORD_\Gamma(f;z)=ord(f;z)/m$ where $m$ is the order of $z$ as a
fixed point of $\Gamma$ (so $m=1$, $2$, or $3$).
If $f$ is not the zero function 
and $\mathcal{D}\subset\mathcal{H}\cup\mathbb{Q}\cup\{\infty\}$ is a
fundamental domain for the action of $\Gamma$ on $\mathcal{H}$ along 
with a complete set of inequivalent cusps for the action,
then
\begin{align}\label{eqn:valence}
	\sum_{\zeta \in \mathcal{D}} ORD_\Gamma(f;\zeta) = 0
.
\end{align}

\section{Generating Function for the $M_d$-rank of overpartitions}

We recall $d$ is a positive integer.
We denote the generating function of the $M_d$-rank of overpartitions by
\begin{align*}
	\mathcal{O}_d(z;q)
	&=
		\sum_{n=0}^\infty\sum_{m=-\infty}^\infty M_d(m,n) z^m q^n
	\\
	&=
		\frac{\aqprod{-q}{q}{\infty}}{\aqprod{q}{q}{\infty}}	
		\left(
			1
			+
			2\sum_{n=1}^\infty \frac{(1-z)(1-z^{-1})(-1)^n q^{n^2+dn}}
				{(1-zq^{dn})(1-z^{-1}q^{dn})}
		\right)
	\\
	&=
		\frac{\aqprod{-q}{q}{\infty}}{\aqprod{q}{q}{\infty}}	
		\sum_{n=-\infty}^\infty \frac{(1-z)(1-z^{-1})(-1)^n q^{n^2+dn}}
			{(1-zq^{dn})(1-z^{-1}q^{dn})}
	.
\end{align*}

From the fact that
\begin{align*}
	\frac{(1-z)(1-z^{-1})q^{dn}}{(1-zq^{dn})(1-z^{-1}q^{dn}) }
	&=
		1
		-
		\frac{(1-q^{dn})}{(1+q^{dn})}
		\sum_{m=0}^\infty z^{m}q^{dmn}
		-
		\frac{(1-q^{dn})}{(1+q^{dn})}
		\sum_{m=1}^\infty z^{-m}q^{dmn}		
	,
\end{align*}
we quickly deduce that
\begin{align*}
	\sum_{n=1}^\infty M_d(m,n) q^n
	&=
		2\frac{\aqprod{-q}{q}{\infty}}{\aqprod{q}{q}{\infty}}
		\sum_{n=1}^\infty
		\frac{(-1)^{n+1}(1-q^{dn})q^{n^2+d|m|n} }{(1+q^{dn})}
	,
\end{align*}
for all $m$. Similarly, if we define
\begin{align*}
M_d(r,m,n) &= \sum_{k\equiv r\pmod{m}} M_d(k,n)
\end{align*}
then for $m\ge0$ we find that
\begin{align*}	
	\sum_{n=1}^\infty M_d(r,m,n) q^n
	&=
		2\frac{\aqprod{-q}{q}{\infty}}{\aqprod{q}{q}{\infty}}
		\sum_{n=1}^\infty
		\frac{(-1)^{n+1}(1-q^{dn})q^{n^2}(q^{drn}+q^{d(m-r)n})  } 
		{(1+q^{dn})(1-q^{dmn})}
	.
\end{align*}

\begin{proposition}\label{PropM2RankAsLambertSeries}
We have that
\begin{align*}
	\mathcal{O}_d(z;q)
	&=
		\frac{(1-z)\aqprod{-q}{q}{\infty}}{(1+z)\aqprod{q}{q}{\infty}}
		\sum_{n=-\infty}^\infty
		\frac{(-1)^n q^{n^2} (1+zq^{dn})}{(1-zq^{dn})}
	.
\end{align*}
\end{proposition}
\begin{proof}
To begin we notice that
\begin{align*}
	\frac{-(1-z)(1+z)}{z(1-zq^{dn})(1-z^{-1}q^{dn})}
	&=
		\frac{z}{1-zq^{dn}}
		-
		\frac{z^{-1}}{1-z^{-1}q^{dn}}
	,
\end{align*}
and so
\begin{align*}
	\mathcal{O}_d(z;q)
	&=
		\frac{z\aqprod{-q}{q}{\infty}}{(1+z)\aqprod{q}{q}{\infty}}	
		\sum_{n=-\infty}^\infty
		(1-z^{-1})(-1)^{n+1} q^{n^2+dn}
		\left(
			\frac{z}{1-zq^{dn}}
			-
			\frac{z^{-1}}{1-z^{-1}q^{dn}}
		\right)
	\\
	&=
		\frac{(1-z)\aqprod{-q}{q}{\infty}}{(1+z)\aqprod{q}{q}{\infty}}	
		\sum_{n=-\infty}^\infty
		\frac{(-1)^{n} q^{n^2+dn}z}{1-zq^{dn}}
		-
		\frac{(1-z)\aqprod{-q}{q}{\infty}}{(1+z)\aqprod{q}{q}{\infty}}	
		\sum_{n=-\infty}^\infty
		\frac{(-1)^{n} q^{n^2+dn}z^{-1}}{1-z^{-1}q^{dn}}
.
\end{align*}
In the second series we let $n\mapsto -n$ to find that
\begin{align*}
	\mathcal{O}_d(z;q)
	&=
		\frac{(1-z)\aqprod{-q}{q}{\infty}}{(1+z)\aqprod{q}{q}{\infty}}	
		\sum_{n=-\infty}^\infty
		\frac{(-1)^{n} q^{n^2+dn}z}{1-zq^{dn}}
		-
		\frac{(1-z)\aqprod{-q}{q}{\infty}}{(1+z)\aqprod{q}{q}{\infty}}	
		\sum_{n=-\infty}^\infty
		\frac{(-1)^{n} q^{n^2-dn}z^{-1}}{1-z^{-1}q^{-dn}}
	\\
	&=
		\frac{(1-z)\aqprod{-q}{q}{\infty}}{(1+z)\aqprod{q}{q}{\infty}}	
		\sum_{n=-\infty}^\infty
		\frac{(-1)^{n} q^{n^2}(1+zq^{dn}) }{1-zq^{dn}}
.
\end{align*}
\end{proof}

\begin{corollary}\label{CorollaryMdRankInitialForm}
Suppose $d$ is odd, then we have
\begin{align*}
	\mathcal{O}_d(z;q)
	&=
		\frac{(1-z)\aqprod{-q}{q}{\infty}}{(1+z)\aqprod{q}{q}{\infty}}	
		\sum_{k=0}^{d-1}
		(-1)^k q^{k^2}
		\sum_{n=-\infty}^\infty
		\frac{(-1)^{n} q^{d^2n^2+2dkn}(1+zq^{d^2n+dk})^2 }{(1-z^2q^{2d^2n+2dk})}
.
\end{align*}
Suppose $d$ is even, then we have
\begin{align*}
	\mathcal{O}_d(z;q)
	&=
		\frac{(1-z)\aqprod{-q}{q}{\infty}}{(1+z)\aqprod{q}{q}{\infty}}
		\sum_{k=0}^{\frac{d}{2}-1}
		(-1)^{k} q^{k^2}
		\sum_{n=-\infty}^\infty
		\frac{ (-1)^{\frac{d}{2}n} q^{\frac{d^2}{4}n^2 + dkn} (1+zq^{\frac{d^2}{2}n+dk})  }
			{(1-zq^{\frac{d^2}{2}n+dk})}
.
\end{align*}
\end{corollary}
\begin{proof}
When $d$ is odd, we let $n\mapsto dn+k$ 
for $k=0,1,\dotsc, d-1$ in Proposition \ref{PropM2RankAsLambertSeries}
and simplify to obtain the result. When $d$ is even, we instead let 
$n\mapsto \frac{d}{2}n+k$ for $k=0,1,\dotsc,\frac{d}{2}-1$.
\end{proof}

The purpose of writing $\mathcal{O}_d(z;q)$ in this form is that the generalized
Lambert series can be easily rewritten in terms of Zwegers' function
$\mu(u,v;\tau)$ from Chapter 1 of \cite{Zwegers}. One could also use another
rearrangement of the series and use the functions from Chapter 3 of \cite{Zwegers}
or the functions $A_\ell(u,v;\tau)$ of \cite{Zwegers2}.
To allow for uniform proofs, we only use $\mu(u,v;\tau)$.
The exact form of $\mathcal{O}_d(z;q)$ in terms of $\mu(u,z;\tau)$ will
depend on $\gcd(d,4)$.

We let $u(z,d,k)$ and $v(d,k)$ be given
\begin{align}\label{eqn:uvdefs}
	u(z,d,k)
	&=
	\left\{\begin{array}{cl}
		\frac{\log(z)}{2\pi i} + dk\tau
		& \mbox{ if } \gcd(d,2) = 2,
		\\[.5em]
		\frac{\log(z)}{\pi i} + 2dk\tau
		& \mbox{ if } \gcd(d,2) = 1,
	\end{array}
	\right.
	\\
	v(d,k)
	&=
	\left\{\begin{array}{cl}
		\frac{1}{2} + \left( \frac{d^2}{4}+dk \right)\tau
		& \mbox{ if } \gcd(d,4) = 4,
		\\[.5em]
		1 + \left( \frac{d^2}{4}+dk \right)\tau
		& \mbox{ if } \gcd(d,4) = 2,
		\\[.5em]
		\left( d^2 + 2dk \right)\tau
		& \mbox{ if } \gcd(d,4) = 1,
	\end{array}
	\right.	
\end{align}
and so
\begin{align*}
	\exp(2\pi iu(z,d,k))
	&=
	\left\{\begin{array}{cl}
		zq^{dk}
		& \mbox{ if } \gcd(d,2) = 2,
		\\[.5em]
		z^2 q^{2dk}
		& \mbox{ if } \gcd(d,2) = 1,
	\end{array}
	\right.
	\\
	\exp(2\pi iv(d,k))
	&=
	\left\{\begin{array}{cl}
		-q^{\frac{d^2}{4}+dk }
		& \mbox{ if } \gcd(d,4) = 4,
		\\[.5em]
		q^{\frac{d^2}{4}+dk }
		& \mbox{ if } \gcd(d,4) = 2,
		\\[.5em]
		q^{d^2+2dk }
		& \mbox{ if } \gcd(d,4) = 1.
	\end{array}
	\right.	
\end{align*}

\begin{proposition}\label{PropMdRankToMu}
Suppose $d$ is odd, then
\begin{align*}
	\mathcal{O}_d(z;q)
	&=
		\frac{(1-z)\aqprod{-q}{q}{\infty}}{(1+z)\aqprod{q}{q}{\infty}}
		\sum_{k=0}^{d-1}
			(-1)^k q^{k^2 + dk} 
			\vartheta\left( v(d,k); 2d^2\tau  \right)
			\left(
				2z\mu\left( u(z,d,k), v(d,k); 2d^2\tau    \right)			
				+				
				iq^{\frac{d^2}{4}}
			\right)
		\\&\quad
		+
		\frac{2z(1-z)\aqprod{-q}{q}{\infty}\aqprod{q^{2d^2}}{q^{2d^2}}{\infty}^2}
			{(1+z)\aqprod{q}{q}{\infty}\jacprod{z^2}{q^{2d^2}}}
		\\&\quad		
		+
		\frac{2z(1-z)\aqprod{-q}{q}{\infty}\aqprod{q^{2d^2}}{q^{2d^2}}{\infty}^2\jacprod{z^2q^{2d^2}}{q^{2d^2}}}
			{(1+z)\aqprod{q}{q}{\infty}}
		\sum_{k=1}^{\frac{d-1}{2}}
		\frac{(-1)^k q^{k^2 + dk} \jacprod{q^{4dk}}{q^{2d^2}}  }
			{\jacprod{z^2q^{2dk}, z^{-2}q^{2dk}, q^{2dk}}{q^{2d^2}}}
.
\end{align*}
Suppose $d$ is even, then
\begin{align*}
	&\mathcal{O}_d(z;q)
	\\
	&=
		\frac{(1-z)\aqprod{-q}{q}{\infty}}{(1+z)\aqprod{q}{q}{\infty}}
		\sum_{k=0}^{\frac{d}{2}-1}
			(-1)^k q^{k^2 + \frac{dk}{2}} 
			\vartheta\left( v(d,k); \tfrac{d^2\tau}{2}  \right)
			\left(
				2z^{\frac{1}{2}}\mu\left( u(z,d,k), v(d,k); \tfrac{d^2\tau}{2} \right)
				-
				(-1)^{\frac{d}{2}}i^{1-\frac{4}{\gcd(d,4)}}q^{\frac{d^2}{16}} 
			\right)	
.
\end{align*}
\end{proposition}
\begin{proof}
We begin with the case when $d$ is odd.
By Corollary \ref{CorollaryMdRankInitialForm} we have that
\begin{align*}
	\mathcal{O}_d(z;q)
	&=
		\frac{(1-z)\aqprod{-q}{q}{\infty}}{(1+z)\aqprod{q}{q}{\infty}}
		\sum_{k=0}^{d-1} (-1)^k q^{k^2}
		\sum_{n=-\infty}^{\infty}
		\frac{(-1)^n q^{d^2n^2 + 2dkn}}{(1-z^2q^{2d^2n+2dk})}	
		\\&\quad
		+
		\frac{2z(1-z)\aqprod{-q}{q}{\infty}}{(1+z)\aqprod{q}{q}{\infty}}
		\sum_{k=0}^{d-1} (-1)^k q^{k^2}
		\sum_{n=-\infty}^{\infty}
		\frac{(-1)^n q^{d^2n^2 + d^2n + 2dkn + dk}}{(1-z^2q^{2d^2n+2dk})}	
		\\&\quad
		+	
		\frac{z^2(1-z)\aqprod{-q}{q}{\infty}}{(1+z)\aqprod{q}{q}{\infty}}
		\sum_{k=0}^{d-1} (-1)^k q^{k^2}
		\sum_{n=-\infty}^{\infty}
		\frac{(-1)^n q^{d^2n^2 + 2d^2n + 2dkn+2dk}}{(1-z^2q^{2d^2n+2dk})}	
	.
\end{align*}

It turns out the middle sum can be expressed entirely in terms of theta 
functions. Considering the special case of Theorem 2.1 of \cite{Chan} when $r=0$ and $s=1$, letting $q\mapsto q^{2d^2}$ and $b=z^2$ gives that
\begin{align*}
	\sum_{n=-\infty}^\infty
	\frac{(-1)^n q^{d^2n^2 + d^2n}}{1-z^2q^{2d^2n}}
	&=
	\frac{\aqprod{q^{2d^2}}{q^{2d^2}}{\infty}^2}{\jacprod{z^2}{q^{2d^2}}}
,
\end{align*}
which accounts for the $k=0$ term of the middle sum.
Next the special case of Theorem 2.1 of \cite{Chan} when $r=1$ and $s=2$ states that
\begin{align*}
	\frac{\jacprod{a}{q}\aqprod{q}{q}{\infty}^2}{\jacprod{b_1,b_2}{q}}
	&=
		\frac{\jacprod{a/b_1}{q}}{\jacprod{b_2/b_1}{q}}
		\sum_{n=-\infty}^\infty
		\frac{(-1)^n q^{\frac{n(n+1)}{2}} (a/b_2)^n  }
			{1 - b_1q^n}
		+
		\frac{\jacprod{a/b_2}{q}}{\jacprod{b_1/b_2}{q}}
		\sum_{n=-\infty}^\infty
		\frac{(-1)^n q^{\frac{n(n+1)}{2}} (a/b_1)^n  }
			{1 - b_2q^n}
.
\end{align*}
Noting that
$\frac{1}{\jacprod{b_1/b_2}{q}} = -\frac{b_2}{b_1\jacprod{b_2/b_1}{q}}$,
we have that
\begin{align}\label{EqChanIdent}
	&\frac{\jacprod{a, b_2/b_1}{q} \aqprod{q}{q}{\infty}^{ 2} }
		{\jacprod{b_1,b_2}{q}}
	\nonumber\\
	&=
		\jacprod{a/b_1}{q}
		\sum_{n=-\infty}^\infty
		\frac{(-1)^n q^{\frac{n(n+1)}{2}} (a/b_2)^n }
			{1-b_1q^n}
		-
		\frac{b_2}{b_1}\jacprod{a/b_2}{q}
		\sum_{n=-\infty}^\infty
		\frac{(-1)^n q^{\frac{n(n+1)}{2}} (a/b_1)^n }
			{1-b_2q^n}
.
\end{align}
In (\ref{EqChanIdent}), for $k\not\equiv 0\pmod{d}$,  
we let
$q\mapsto q^{2d^2}$, $b_1=z^2q^{2dk}$, $b_2=z^2q^{2d(d-k)}$,
$a=z^2q^{2d^2}$, and simplify to find that
\begin{align}
	\label{eqn_usefulprod}
	\frac{ \jacprod{z^2q^{2d^2}, q^{4dk}}{q^{2d^2}} \aqprod{q^{2d^2}}{q^{2d^2}}{\infty}^2}
		{\jacprod{z^2q^{2dk}, z^{-2}q^{2dk}, q^{2dk}}{q^{2d^2}}} 
	&= 
	\sum_{n\in\Z} 
	\frac{(-1)^n q^{d^2n^2+d^2n+2dkn}}{1-z^2 q^{2d^2n + 2dk}} 
	- 
	q^{2d(d-2k)}
	\sum_{n\in\Z} 
	\frac{(-1)^nq^{d^2n^2+d^2n+2d(d-k)n}}{1-z^2 q^{2d^2n + 2d(d-k)}}
	. 
\end{align}
From \eqref{eqn_usefulprod} we see that
\begin{align*}
	&
	(-1)^k q^{k^2} 
	\sum_{n\in\Z} \frac{(-1)^n q^{d^2n^2 + d^2n + 2dkn + dk}}{1-z^2q^{2d^2n+2dk}} 
	+ 
	(-1)^{d-k} q^{(d-k)^2} 
	\sum_{n\in\Z} \frac{(-1)^n q^{d^2n^2 + d^2n + 2d(d-k)n + d(d-k)}}{1-z^2q^{2d^2n+2d(d-k)}} 
	\\
	&= 
		\frac{ (-1)^k q^{k^2 + 2dk} 
			\jacprod{z^2q^{2d^2}, q^{4dk}}{q^{2d^2}} \aqprod{q^{2d^2}}{q^{2d^2}}{\infty}^2}
		{\jacprod{z^2q^{2dk}, z^{-2}q^{2dk}, q^{2dk}}{q^{2d^2}}}.
\end{align*}
Noting that when $1\leq k \leq \frac{d-1}2$, we have 
$\frac{d+1}2 \leq d-k \leq d-1$,
altogether we obtain that
\begin{align*}
	&
	\frac{2z(1-z)\aqprod{-q}{q}{\infty}}{(1+z)\aqprod{q}{q}{\infty}} 
	\sum_{k=0}^{d-1} 
	(-1)^kq^{k^2} 
	\sum_{n\in\Z} \frac{(-1)^nq^{d^2n^2+d^2n+2dkn+dk}}{1-z^2q^{2d^2n+2dk}} 
	\\
	&= 
		\frac{2z(1-z)\aqprod{-q}{q}{\infty}}{(1+z)\aqprod{q}{q}{\infty}} 
		\left[
			\frac{\aqprod{q^{2d^2}}{q^{2d^2}}{\infty}^2}{\jacprod{z^2}{q^{2d^2}}} 
			+ 
			\sum_{k=1}^{\frac{d-1}2} 
			\frac{(-1)^k q^{k^2+dk} \jacprod{z^2q^{2d^2}, q^{4dk}}{q^{2d^2}}\aqprod{q^{2d^2}}
			{q^{2d^2}}{\infty}^2}{\jacprod{z^2q^{2dk}, z^{-2}q^{2dk}, q^{2dk}}{q^{2d^2}}} \right]
	.
\end{align*}

Using (\ref{def_mu}) we have that
\begin{align*}
	\mu\left( u(z,k,d), v(d,k) - 2d^2\tau ; 2d^2\tau   \right)
	&=
		\frac{zq^{dk}}{\vartheta\left( v(d,k) - 2d^2\tau; 2d^2\tau \right)}
		\sum_{n=-\infty}^\infty
		\frac{ (-1)^n q^{d^2n^2 + 2dkn}}
		{1-z^2q^{2d^2n+2dk}}
	,\\
	\mu\left( u(z,k,d), v(d,k); 2d^2\tau   \right)
	&=
		\frac{zq^{dk}}{\vartheta\left( v(d,k); 2d^2\tau \right)}
		\sum_{n=-\infty}^\infty
		\frac{ (-1)^n q^{d^2n^2 + 2d^2n +2dkn}}
		{1-z^2q^{2d^2n+2dk}}
	.
\end{align*}
But by (\ref{EqVarThetaTranformation}) and (\ref{EqMuTranformation})
\begin{align*}
	\vartheta\left( v(d,k)-2d^2\tau ; 2d^2\tau \right)
	&=
		-q^{2dk}\vartheta\left( v(d,k); 2d^2\tau \right)
	,\\
	\mu\left( u(z,d,k), v(d,k)-2d^2\tau; 2d^2\tau \right)	
	&=
		-z^2\mu\left( u(z,d,k), v(d,k); 2d^2\tau  \right)
		-
		izq^{\frac{d^2}{4}}
	.
\end{align*}
We then have that
\begin{align*}
	\sum_{n=-\infty}^\infty
	\frac{(-1)^n q^{d^2n^2 + 2dkn}}
		{1-z^2q^{2d^2n +2dk }}
	&=
		z^{-1}q^{-dk}
		\vartheta\left(v(d,k)-2d^2\tau; 2d^2\tau \right)
		\mu\left( u(z,d,k), v(d,k) -2d^2\tau; 2d^2\tau  \right)
	\\
	&=
		zq^{dk}
		\vartheta\left(v(d,k); 2d^2\tau \right)
		\mu\left( u(z,d,k), v(d,k); 2d^2\tau  \right)
		+
		iq^{\frac{d^2}{4}+dk}\vartheta\left(v(d,k); 2d^2\tau \right)
.
\end{align*}
Lastly,
\begin{align*}
	z^2\sum_{n=-\infty}^\infty
	\frac{(-1)^n q^{d^2n^2 + 2d^2n + 2dkn + 2dk}}
		{1-z^2q^{2d^2n + 2dk}}
	&=
		zq^{dk}
		\vartheta\left( v(d,k); 2d^2\tau  \right)
		\mu\left( u(z,d,k), v(d,k); 2d^2\tau  \right)
	,	
\end{align*}
and so we arrive at
\begin{align*}
	\mathcal{O}_d(z;q)
	&=
		\frac{(1-z)\aqprod{-q}{q}{\infty}}
		{(1+z)\aqprod{q}{q}{\infty}	}
		\sum_{k=0}^{d-1}
		(-1)^k q^{k^2+dk} 
		\vartheta\left( v(d,k) ; 2d^2\tau  \right)
		\left( 
			2z\mu\left( u(z,d,k), v(d,k); 2d^2\tau  \right)
			+
			iq^{\frac{d^2}{4}}  
		\right)	
		\\&\quad	
		+
		\frac{2z(1-z)\aqprod{-q}{q}{\infty}\aqprod{q^{2d^2}}{q^{2d^2}}{\infty}^2}
		{(1+z)\aqprod{q}{q}{\infty}\jacprod{z^2}{q^{2d^2}}	}
		\\&\quad		
		+
		\frac{2z(1-z)\aqprod{-q}{q}{\infty}\aqprod{q^{2d^2}}{q^{2d^2}}{\infty}^2\jacprod{z^2q^{2d^2}}{q^{2d^2}}}
		{(1+z)\aqprod{q}{q}{\infty}}
		\sum_{k=0}^{\frac{d-1}{2}}
		\frac{ (-1)^k q^{k^2+dk} \jacprod{q^{4dk}}{q^{2d^2}} }
		{\jacprod{ z^2q^{2dk}, z^{-2}q^{2dk}, q^{2dk} }{q^{2d^2}}}
	.
\end{align*}
This establishes the case when $d$ is odd.

Now suppose that $d$ is even.
Temporarily we will need the additional notation
that
$
	v^{-}(d,k)
	=
		\frac{2}{\gcd(d,4)} 
		+
		(
			-\tfrac{d^2}{4}
			+dk
		)\tau
	,
$
so $e^{2\pi iv^{-}(d,k)} = (-1)^{1+\frac{d}{2}}q^{-\frac{d^2}{4}+dk}$.

We note that by \eqref{def_mu},
\begin{align*}
	\mu\left(u(z,d,k), v(d,k); \tau \right)
	&=
		\frac{z^{\frac{1}{2}}q^{\frac{dk}{2}}}
			{\vartheta(v(d,k); \frac{d^2\tau}{2} )}	
		\sum_{n=-\infty}^\infty
		\frac{ (-1)^{\frac{d}{2}n} q^{\frac{d^2}{4}n^2 + \frac{d^2}{2}n + dkn } }
			{1-zq^{\frac{d^2}{2}n+dk} }
	,\\
	\mu\left(u(z,d,k), v^{-}(d,k); \tau \right)
	&=	
		\frac{z^{\frac{1}{2}}q^{\frac{dk}{2}}}
			{\vartheta(v^{-}(d,k); \frac{d^2\tau}{2} )}	
		\sum_{n=-\infty}^\infty
		\frac{ (-1)^{\frac{d}{2}n} q^{\frac{d^2}{4}n^2 + dkn} }
			{1-zq^{\frac{d^2}{2}n+dk} }
.
\end{align*}
Thus
\begin{align*}
	\mathcal{O}(z;q)
	&=
		\frac{(1-z)\aqprod{-q}{q}{\infty}}{(1+z)\aqprod{q}{q}{\infty}}
		\sum_{k=0}^{\frac{d}{2}-1}
			(-1)^k q^{k^2}
		\left(
			z^{-\frac{1}{2}}q^{-\frac{dk}{2}}
			\vartheta(v^{-}_{d,k}, \tfrac{d^2\tau}{2})
			\mu\left( u_{z,d,k}, v^{-}_{d,k}; \tfrac{d^2\tau}{2}  \right)		
			\right.\\&\qquad\left.			
			+
			z^{\frac{1}{2}}q^{\frac{dk}{2}}
			\vartheta(v_{d,k}, \tfrac{d^2\tau}{2})
			\mu\left( u_{z,d,k}, v_{d,k}; \tfrac{d^2\tau}{2}  \right)		
		\right)
.
\end{align*}
However $v(d,k) = v^{-}(d,k)+\frac{d^2\tau}{2}$,
so by (\ref{EqVarThetaTranformation}) and (\ref{EqMuTranformation}) we have that
\begin{align*}
	\vartheta\left( v^-_{d,k}; \tfrac{d^2\tau}{2}  \right)
	&=
		(-1)^{\frac{d}{2}}q^{dk}\vartheta\left( v_{d,k}; \tfrac{d^2\tau}{2}  \right)
	,\\
	\mu\left( u_{z,d,k}, v^-_{d,k}; \tfrac{d^2\tau}{2}   \right)
	&=
		(-1)^{\frac{d}{2}} z \mu\left( u_{z,d,k}, v_{d,k}; \tfrac{d^2\tau}{2}   \right)
		-
		i^{1-\frac{4}{gcd(d,4)}} z^{\frac{1}{2}} q^{\frac{d^2}{16}}
	.
\end{align*}
This gives us that
\begin{align*}
	z^{-\frac{1}{2}}q^{-\frac{dk}{2}}
	\vartheta(v^{-}(d,k), \tfrac{d^2\tau}{2})
	\mu\left( u(z,d,k), v^{-}(d,k); \tfrac{d^2\tau}{2}  \right)		
	&=
		z^{\frac{1}{2}}q^{\frac{dk}{2}}
		\vartheta(v(d,k), \tfrac{d^2\tau}{2})
		\mu\left( u(z,d,k), v(d,k); \tfrac{d^2\tau}{2}  \right)		
		\\&\quad
		-
		(-1)^{\frac{d}{2}}i^{1-\frac{4}{\gcd(d,4)}} q^{\frac{d^2}{16}+\frac{dk}{2}}
		\vartheta(v(d,k), \tfrac{d^2\tau}{2})
.
\end{align*}
Thus for even $d$ we have that
\begin{align*}
	&\mathcal{O}_d(z;q)
	\\
	&=
		\frac{(1-z)\aqprod{-q}{q}{\infty}}{(1+z)\aqprod{q}{q}{\infty}}
		\sum_{k=0}^{\frac{d}{2}-1}
			(-1)^k q^{k^2 + \frac{dk}{2}} 
			\vartheta\left( v(d,k); \tfrac{d^2\tau}{2}  \right)
			\left(
				2z^{\frac{1}{2}}\mu\left( u(z,d,k), v(d,k); \tfrac{d^2\tau}{2} \right)
				-
				(-1)^{\frac{d}{2}}i^{1-\frac{4}{\gcd(d,4)}}q^{\frac{d^2}{16}} 
			\right)	
.
\end{align*}
\end{proof}

\subsection{Rewriting the generating function in terms of harmonic Maass forms}

In order to rewrite the functions occurring in Proposition \ref{PropMdRankToMu}, we make the following definition.  Let 
\begin{equation}\label{def:Mtilde}
\widetilde{M}_{d,k}(a,b,c;\tau) := 
\begin{cases}
2\z{c}{a} q^{-\frac{(2b-cd^2)^2}{8c^2d^4}} \tmu\left(\frac{2a}{c} + \left( \frac{b}{cd^2}+\frac{k}{d} \right)\tau, \left(\frac{1}{2} + \frac{k}{d} \right)\tau; \tau \right) & \text{ if } d \text{ odd}, \\
 2\zeta_{2c}^a q^{-\frac{(4b-cd^2)^2}{8c^2d^4}} \tmu\left(\frac{a}c + (\frac{2b}{cd^2} + \frac{2k}{d})\tau, 1 + (\frac12 + \frac{2k}d)\tau; \tau \right)      & \text{ if } d \equiv 2 \pmod 4, \\
  2\zeta_{2c}^a q^{-\frac{(4b-cd^2)^2}{8c^2d^4}} \tmu\left(\frac{a}c + (\frac{2b}{cd^2} + \frac{2k}{d})\tau, \frac12 + (\frac12 + \frac{2k}d)\tau; \tau \right)     & \text{ if } d\equiv 0 \pmod 4, 
\end{cases}
\end{equation} 
where in each case $a$, $b$, $c$, and $k$ are integers such that the 
corresponding $\tmu(u,v;\tau)$ satisfies $u,v\not\in \Z + \Z\tau$.  
From here to the end of the article, any mention of 
$\widetilde{M}_{d,k}(a,b,c;\tau)$ implicitly makes this assumption
about $a$, $b$, $c$ and $k$.
We note 
that
\begin{equation}\label{eqn:evaluatedMtilde}
\begin{cases}
\widetilde{M}_{d,k}(a,b,c;2d^2\tau) 
= 
	2\z{c}{a}q^{\frac{b}{c}} q^{ -\frac{d^2}{4}- \frac{b^2}{c^2d^2} } 
	\tmu\left(  u( \z{c}{a}q^{\frac{b}{c}},d,k ), v(d,k); 2d^2\tau \right)  
	& \text{ if } d \text{ odd}
, \\
\widetilde{M}_{d,k}\left(a,b,c;\frac{d^2\tau}{2}\right) 
= 
	2\zeta_{2c}^a q^{\frac{b}{2c}-\frac{d^2}{16} - \frac{b^2}{c^2d^2}} 
	\tmu\left( u(\zeta_c^a q^\frac{b}c, d, k), v(d,k); \frac{d^2\tau}{2} \right)  
	& \text{ if } d \text{ even},
\end{cases}
\end{equation}
and that the corresponding $\mu$ functions are of the form appearing in the 
expansion of $\mathcal{O}_d(\zeta_c^a q^\frac{b}c; q)$ given in 
Propositions \ref{PropMdRankToMu}. This enables us to rewrite 
$\mathcal{O}_d(\zeta_c^a q^\frac{b}c; q)$ in terms of modular objects.

Additionally, when $d$ is odd, we define $P_{d,k}(a,b,c;\tau)$ by
\begin{align}\label{def:Pdk}
P_{d,k}(a,b,c;\tau)
&=
\begin{cases}\displaystyle	
	\frac{2\exp\left(\tfrac{2\pi iab}{c^2d^2}\right)   
		\eta(2\tau) f_{2d^2,4dk}(\tau) \kModular{-\frac{b}{cd^2}}{-\frac{2a}{c}}{2d^2\tau} }
	{\eta(\tau)^2 f_{2d^2,2dk}(\tau) \kModular{\frac{b}{cd^2}+\frac{k}{d}}{\frac{2a}{c}}{2d^2\tau}
			\kModular{-\frac{b}{cd^2}+\frac{k}{d}}{-\frac{2a}{c}}{2d^2\tau}}
	&
	\mbox{if } d\nmid k
	,\\[3ex] \displaystyle
	2\exp\left(\tfrac{2\pi iab}{c^2d^2}\right)
	\frac{ \eta(2\tau) }
	{\eta(\tau)^2 \kModular{\frac{b}{cd^2}}{\frac{2a}{c}}{2d^2\tau}}
	&
	\mbox{if }d\mid k,
\end{cases}
\end{align}
where $a$, $b$, $c$, and $k$ are integers chosen to avoid any division by zero,
From here to the end of the article, any mention of 
$P_{d,k}(a,b,c;\tau)$ implicitly makes this assumption
about $a$, $b$, $c$ and $k$.
One can use the properties of $\mathfrak{k}$ and $f$ to see that the value 
of $P_{d,k}$ depends on $k$ modulo $d$ rather than the exact value of $k$.

\begin{proposition}\label{prop:OtoMockModular}
If $d$ is odd, then
\begin{multline*}
\frac{(1+\z{c}{a}q^{b/c})}{(1-\z{c}{a}q^{b/c})}q^{-\frac{b^2}{c^2d^2}} \mathcal{O}_d(\z{c}{a}q^{\frac{b}{c}}; q)
= i\sum_{k=0}^{d-1} \frac{(-1)^{k+1} \eta(2\tau) f_{2d^2, d^2+2dk}(\tau)}{\eta(\tau)^2} \widetilde{M}_{d,k}(a,b,c;2d^2\tau) \\
+ \sum_{k=0}^{\frac{d-1}{2}} (-1)^{k+1} P_{d,k}(a,b,c;\tau)
- \left( \z{c}{a}q^{-\frac{d^2}{4}-\frac{b^2}{c^2d^2}+\frac{b}{c}} R\left( \tfrac{2a}{c} + (\tfrac{2b}{c}-d^2)\tau ; 2d^2\tau\right)
- q^{-\frac{b^2}{c^2d^2}} \right).
\end{multline*}
If $d\equiv 2 \pmod 4$, then
\begin{align*}
\frac{(1+\z{c}{a}q^{b/c})}{(1-\z{c}{a}q^{b/c})}q^{-\frac{b^2}{c^2d^2}}\mathcal{O}_d(\z{c}{a}q^{\frac{b}{c}}; q) 
&= 
i\sum_{k=0}^{\frac{d}2-1} (-1)^{k} \frac{\eta(2\tau) f_{\frac{d^2}{2}, \frac{d^2}{4}+dk}(\tau)}{\eta(\tau)^2} \cdot \widetilde{M}_{d,k}\left(a,b,c;\frac{d^2\tau}{2}\right) \\&\quad
+ \left(\z{c}{\frac{a}{2}}q^{-\frac{d^2}{16}-\frac{b^2}{c^2d^2}+\frac{b}{2c}}
R\left( \left(\tfrac{a}{c} -1\right) + \left(\tfrac{b}{c}-\tfrac{d^2}{4}\right)\tau ; \tfrac{d^2\tau}{2}\right)
+q^{-\frac{b^2}{c^2d^2}} \right).
\end{align*}
If $d\equiv 0 \pmod 4$, then 
\begin{align*}
\frac{(1+\z{c}{a}q^{b/c})}{(1-\z{c}{a}q^{b/c})}q^{-\frac{b^2}{c^2d^2}}\mathcal{O}_d(\z{c}{a}q^{\frac{b}{c}}; q) 
&= 
\sum_{k=0}^{\frac{d}2-1} (-1)^{k+1} \frac{ \eta(2\tau) \eta(\frac{d^2\tau}{2})^2 f_{d^2, \frac{d^2}{2}+2dk}(\tau)} {\eta(\tau)^2 \eta(d^2\tau) f_{\frac{d^2}{2}, \frac{d^2}{4}+dk}(\tau)} 
\cdot \widetilde{M}_{d,k}\left(a,b,c;\frac{d^2\tau}{2}\right) \\&\quad
+ \left( i \z{c}{\frac{a}{2}}q^{-\frac{d^2}{16}-\frac{b^2}{c^2d^2}+\frac{b}{2c}}
R\left( \left(\tfrac{a}{c} -\tfrac12\right) + \left(\tfrac{b}{c}-\tfrac{d^2}{4}\right)\tau ; \tfrac{d^2\tau}{2}\right)
+q^{-\frac{b^2}{c^2d^2}} \right).
\end{align*}
\end{proposition}
\begin{proof}
This proposition is little more than converting our notation.
First, when $d$ is odd, with $z=\z{c}{a}q^{\frac{b}{c}}$ we have
\begin{align*}
	u(z,d,k)
	&=
		\frac{2a}{c}+\frac{2b\tau}{c}+2dk\tau
	, \\
	v(d,k)
	&=
		d^2\tau + 2dk\tau
	.
\end{align*}

To begin we note that
\begin{align*}
	&
	q^{-\frac{d^2}{4}-\frac{b^2}{c^2d^2}}
	\left(
		2z\mu\left( u(z,d,k), v(d,k) ;2d^2\tau\right)
		+
		iq^{\frac{d^2}{4}}
	\right)
	\\
	&=
		q^{-\frac{d^2}{4}-\frac{b^2}{c^2d^2}}
		\left(
			2z\tmu\left( u(z,d,k), v(d,k) ;2d^2\tau\right)
			-
			izR\left( \tfrac{2a}{c} + \tfrac{2b\tau}{c} - d^2\tau; 2d^2\tau \right)			
			+
			iq^{\frac{d^2}{4}}
		\right)
	\\
	&=
		\widetilde{M}_{d,k}(a,b,c;2d^2\tau)
		-
		i\z{c}{a}q^{-\frac{d^2}{4}-\frac{b^2}{c^2d^2}+\frac{b}{c}}
		R\left( \tfrac{2a}{c} + \tfrac{2b\tau}{c} - d^2\tau; 2d^2\tau \right)			
		+
		iq^{-\frac{b^2}{c^2d^2}}
.
\end{align*}
Next since
\begin{align}\label{eqn:EtaIdentity}
	\frac{\aqprod{-q}{q}{\infty}}{\aqprod{q}{q}{\infty}}
	&=
	\frac{\aqprod{q^2}{q^2}{\infty}}{\aqprod{q}{q}{\infty}^2}
	=
	\frac{\eta(2\tau)}{\eta(\tau)^2}
	,
\end{align}
we find that
\begin{align*}
	\frac{(-1)^k q^{k^2+dk+\frac{d^2}{4}} 
	\aqprod{-q}{q}{\infty}\vartheta\left( v(d,k);2d^2\tau \right)}
	{\aqprod{q}{q}{\infty}}
	&=
		\frac{ (-1)^{k+1}i \eta(2\tau) f_{2d^2, d^2+2dk}(\tau)}{\eta(\tau)^2}			
	,\\
	\frac{2z q^{-\frac{b^2}{c^2d^2}} \aqprod{-q}{q}{\infty} \aqprod{q^{2d^2}}{q^{2d^2}}{\infty}^2 }
		{\aqprod{q}{q}{\infty} \jacprod{z^2}{q^{2d^2}}}
	&=
		-\frac{2\exp\left(\frac{2\pi iab}{c^2d^2}\right)\eta(2\tau)}
		{\eta(\tau)^2 \kModular{\frac{b}{cd^2}}{\frac{2a}{c}}{2d^2\tau}}
	=
	-	
	P_{d,0}(a,b,c;\tau)
,
\end{align*}
and similarly
\begin{align*}
	\frac{2z(-1)^k q^{k^2+dk-\frac{b^2}{c^2d^2}} \aqprod{-q}{q}{\infty} 
		\aqprod{q^{2d^2}}{q^{2d^2}}{\infty}^2 \jacprod{z^2q^{2d^2}, q^{4dk}}{q^{2d^2}}}
	{\aqprod{q}{q}{\infty}\jacprod{z^2q^{2dk}, z^{-2}q^{2dk}, q^{2dk}}{q^{2d^2}}}
	&=
		(-1)^{k+1} P_{d,k}(a,b,c;\tau)
	.
\end{align*}
Thus by Proposition \ref{PropMdRankToMu} we have
\begin{align*}
	&
	\frac{(1+\z{c}{a}q^{b/c})}{(1-\z{c}{a}q^{b/c})}q^{-\frac{b^2}{c^2d^2}}
	\mathcal{O}_d(\z{c}{a}q^{\frac{b}{c}}; q)
	\\
	&=
		i\sum_{k=0}^{d-1}
		\frac{(-1)^{k+1} \eta(2\tau) f_{2d^2, d^2+2dk}(\tau)}{\eta(\tau)^2}
		\widetilde{M}_{d,k}(a,b,c;2d^2\tau)
		+
		\sum_{k=0}^{\frac{d-1}{2}} 
		(-1)^{k+1}		
		P_{d,k}(a,b,c;\tau)
		\\&\quad		
		+
		\sum_{k=0}^{d-1} 
		\frac{(-1)^{k+1} \eta(2\tau) f_{2d^2, d^2+2dk}(\tau)}{\eta(\tau)^2}
		\left(
			\z{c}{a}q^{-\frac{d^2}{4}-\frac{b^2}{c^2d^2}+\frac{b}{c}}
			R\left( \tfrac{2a}{c} + (\tfrac{2b}{c}-d^2)\tau ; 2d^2\tau\right)
			-
			q^{-\frac{b^2}{c^2d^2}}
		\right)
.
\end{align*}
However, for $d$ odd we notice that
\begin{align*}
	\frac{\eta(\tau)^2}{\eta(2\tau)}
	&=
		\sum_{n=-\infty}^\infty (-1)^n q^{n^2}
	=
		\sum_{k=0}^{d-1}
		(-1)^k q^{k^2}
		\sum_{n=-\infty}^\infty (-1)^{n} q^{d^2n^2+2dkn}
	\\
	&=
		\sum_{k=0}^{d-1}
		(-1)^kq^{k^2}
		\aqprod{q^{d^2+2dk},q^{d^2-2dk},q^{2d^2}}{q^{2d^2}}{\infty}	
	=
		\sum_{k=0}^{d-1}
		(-1)^k f_{2d^2,d^2+2dk}(\tau)
.
\end{align*}
Thus
\begin{align*}
	&
	\frac{(1+\z{c}{a}q^{b/c})}{(1-\z{c}{a}q^{b/c})}q^{-\frac{b^2}{c^2d^2}}
	\mathcal{O}_d(\z{c}{a}q^{\frac{b}{c}}; q)
	\\
	&=
		i\sum_{k=0}^{d-1}
		\frac{(-1)^{k+1} \eta(2\tau) f_{2d^2, d^2+2dk}(\tau)}{\eta(\tau)^2}
		\widetilde{M}_{d,k}(a,b,c;2d^2\tau)
		+
		\sum_{k=0}^{\frac{d-1}{2}} 
		(-1)^{k+1}		
		P_{d,k}(a,b,c;\tau)
		\\&\quad		
		-
		\left(
			\z{c}{a}q^{-\frac{d^2}{4}-\frac{b^2}{c^2d^2}+\frac{b}{c}}
			R\left( \tfrac{2a}{c} + (\tfrac{2b}{c}-d^2)\tau ; 2d^2\tau\right)
			-
			q^{-\frac{b^2}{c^2d^2}}
		\right)
,
\end{align*}
which is our goal.

The proof when $d\equiv 2 \pmod 4$ is very similar to the case when $d$ is odd, so we omit it. 
When $d\equiv 0 \pmod 4$,  we have that
\begin{align*}
\mathcal{O}_d(\z{c}{a}q^{\frac{b}{c}}; q) 
&= 
	\frac{(1-\z{c}{a}q^{b/c})}{(1+\z{c}{a}q^{b/c})}\frac{\eta(2\tau)}{\eta(\tau)^2} 
	\sum_{k=0}^{\frac{d}2-1} (-1)^{k}q^{k^2+\frac{dk}{2}} \vartheta \left(\frac12+\left(\frac{d^2}{4} + dk\right)\tau;\frac{d^2\tau}{2}\right) 
	\\&\quad\times 
 	\left( 2\z{c}{\frac{a}{2}}q^{\frac{b}{2c}} 
 		\mu\left( u(\zeta_c^a q^\frac{b}c, d, k), v(d,k); \frac{d^2\tau}{2}\right) 
 		-
 		q^{\frac{d^2}{16}} 
 	\right).
\end{align*}
Analyzing the theta function in this case requires a bit more finesse, but 
upon noting that
\begin{align*}
\vartheta\left(\frac12+\left(\frac{d^2}{4} + dk\right)\tau;\frac{d^2\tau}{2}\right) 
&= 
	-q^{-\frac{d^2}{16}-\frac{dk}{2}} 
	\left(-q^{\frac{d^2}{4}+dk}, -q^{\frac{d^2}{4}-dk},q^{\frac{d^2}{2}}; q^{\frac{d^2}{2}}\right) 
\\
&=
	-q^{-\frac{d^2}{16}-\frac{dk}{2} -k^2}\frac{\eta(\frac{d^2\tau}{2})^2f_{d^2, \frac{d^2}{2}+2dk}(\tau)}
	{\eta(d^2\tau)f_{\frac{d^2}{2}, \frac{d^2}{4}+dk}(\tau)},
\end{align*}
we see the calculations will indeed be similar to before, so we omit the details for this case also.
\end{proof}

We now notice that the definition of $\widetilde{\mathcal{O}}_d(a,b,c;\tau)$ from the introduction
was made in order to analyze the functions occurring in Proposition \ref{prop:OtoMockModular},
as now we have that
\begin{align}\label{eqn:OviaM}
&\widetilde{\mathcal{O}}_d(a,b,c;\tau) 
\nonumber\\
&=
\begin{cases}\displaystyle
	i\sum_{k=0}^{d-1} \frac{(-1)^{k+1} \eta(2\tau) f_{2d^2, d^2+2dk}(\tau)}{\eta(\tau)^2} 
	\widetilde{M}_{d,k}(a,b,c;2d^2\tau)  
	+
	\sum_{k=0}^{\frac{d-1}{2}} (-1)^{k+1} P_{d,k}(a,b,c;\tau)  
	& 
	\text{if } d \text{ is odd}
	, \\\displaystyle
	i\sum_{k=0}^{\frac{d}2-1} (-1)^{k} \frac{\eta(2\tau) f_{\frac{d^2}{2}, \frac{d^2}{4}+dk}(\tau)}{\eta(\tau)^2} 
	\widetilde{M}_{d,k}\left(a,b,c;\frac{d^2\tau}{2}\right) 
	& 
	\text{if } d\equiv 2 \pmod 4
	,\\\displaystyle
	\sum_{k=0}^{\frac{d}2-1} (-1)^{k+1} 
	\frac{ \eta(2\tau) \eta(\frac{d^2\tau}{2})^2 f_{d^2, \frac{d^2}{2}+2dk}(\tau)} 
		{\eta(\tau)^2 \eta(d^2\tau) f_{\frac{d^2}{2}, \frac{d^2}{4}+dk}(\tau)} 
	\widetilde{M}_{d,k}\left(a,b,c;\frac{d^2\tau}{2}\right) 
	& 
	\text{if } d\equiv 0 \pmod 4.
\end{cases}
\end{align}
From here to the end of the article, any mention of 
$\widetilde{\mathcal{O}}_{d}(a,b,c;\tau)$ implicitly makes the
assumption that $a$, $b$, and $c$ are chosen so that each
$\widetilde{M}_{d,k}(a,b,c;\tau)$ and $P_{d,k}(a,b,c;\tau)$ is well defined.
One can check that this assumption is equivalent to the condition that $c\nmid 2a$ or $cd\nmid b$.

\subsection{Rewriting the error terms}

We now rewrite the error terms involving $R(u;\tau)$ occurring in the 
original definition of $\widetilde{\mathcal{O}}_d(a,b,c;\tau)$.
We begin with the case when $d$ is odd.

\begin{proposition}\label{PropositionOddDMdRankModularAndNonModular}
Suppose $d$ is odd, and let $n=\Floor{\frac{b}{cd^2}}$.
If $cd^2\mid b$, then
\begin{align*}
	\widetilde{\mathcal{O}}_d(a,b,c;\tau)
	&=
		\frac{(1+\z{c}{a}q^{b/c})}{(1-\z{c}{a}q^{b/c})}q^{-\frac{b^2}{c^2d^2}}
		\mathcal{O}_d(\z{c}{a}q^{\frac{b}{c}}; q)		
		-
		q^{-\frac{b^2}{c^2d^2}}
		+
		(-1)^{n}\z{c}{2na}
		\\&\quad		
		+
		2\sum_{m=1}^{|n|}
		(-1)^{m+n}\z{c}{(2n+1-{\rm sgn}(n)(2m-1))a}
		q^{d^2\left( \frac{{\rm sgn}(n)(2m-1)}{2}-m(m-1)-\frac{1}{2} \right)}
		\\&\quad		
		-
		(-1)^{n}i\sqrt{2}d\z{c}{2an}
		\int_{-\overline{\tau}}^{i\infty}
			\frac{g_{0,\frac{1}{2}-\frac{2a}{c}}(2d^2w)}
			{\sqrt{-i(w+\tau)}}		
		dw
	;
\end{align*}
in particular if $b=0$ then
\begin{align*}
	\widetilde{\mathcal{O}}_d(a,0,c;\tau)
	&=
		\frac{(1+\z{c}{a})}{(1-\z{c}{a})}
		\mathcal{O}_d(\z{c}{a}; q)
		-
		i\sqrt{2}d
		\int_{-\overline{\tau}}^{i\infty}
			\frac{g_{0,\frac{1}{2}-\frac{2a}{c}}(2d^2w)}
			{\sqrt{-i(w+\tau)}}		
		dw
	.
\end{align*}
If $cd^2\nmid b$, then
\begin{align*}
	&\widetilde{\mathcal{O}}_d(a,b,c;\tau)
	=
		\frac{(1+\z{c}{a}q^{b/c})}{(1-\z{c}{a}q^{b/c})}q^{-\frac{b^2}{c^2d^2}}
		\mathcal{O}_d(\z{c}{a}q^{\frac{b}{c}}; q)
		-
		q^{-\frac{b^2}{c^2d^2}}	
		\\&\quad
		+
		2\sum_{m=1}^{|n|}
		(-1)^{m+n}\z{c}{(2n+1-{\rm sgn}(n)(2m-1))a}
		q^{ n\left(\frac{2b}{c}-d^2n-d^2\right) 
			- (2m-1){\rm sgn(n)}\left(\frac{b}{c}-d^2n-\frac{d^2}{2}\right)
			- m(m-1)d^2 - \frac{d^2}{2} + \frac{b}{c} - \frac{b^2}{c^2d^2}    
		}		
		\\&\quad
		-
		i\sqrt{2}d\exp\left(2\pi i\left( \tfrac{2ab}{c^2d^2}-\tfrac{b}{2cd^2} \right)\right)
		\int_{-\overline{\tau}}^{i\infty}
			\frac{g_{\frac{b}{cd^2}-n,\frac{1}{2}-\frac{2a}{c}} (2d^2w)}
			{\sqrt{-i(w+\tau)}}
		dw	
	;
\end{align*}
in particular if $0<b<cd^2$ then
\begin{align*}
	&\widetilde{\mathcal{O}}_d(a,b,c;\tau)
	\\
	&=
		\frac{(1+\z{c}{a}q^{b/c})}{(1-\z{c}{a}q^{b/c})}q^{-\frac{b^2}{c^2d^2}}
		\mathcal{O}_d(\z{c}{a}q^{\frac{b}{c}}; q)
		-
		q^{-\frac{b^2}{c^2d^2}}	
		-
		i\sqrt{2}d\exp\left(2\pi i\left( \tfrac{2ab}{c^2d^2}-\tfrac{b}{2cd^2} \right)\right)
		\int_{-\overline{\tau}}^{i\infty}
			\frac{g_{\frac{b}{cd^2},\frac{1}{2}-\frac{2a}{c}} (2d^2w)}
			{\sqrt{-i(w+\tau)}}
		dw	
	.
\end{align*}
\end{proposition}
\begin{proof}
We are to determine an alternate form for the error in \eqref{DefinitionOfTildeO}, namely
\begin{align*}
	\z{c}{a}q^{-\frac{d^2}{4}-\frac{b^2}{c^2d^2}+\frac{b}{c}}
	R\left( \tfrac{2a}{c}+\left(\tfrac{2b}{c}-d^2\right)\tau; 2d^2\tau \right)
	-
	q^{-\frac{b^2}{c^2d^2}}	
.
\end{align*}

We then set $n=\Floor{\frac{b}{cd^2}}$ so that
\begin{align*}
	b &= (b-ncd^2) + ncd^2,
\end{align*}
and then from \eqref{eq:Rshift} we have that
\begin{align*}
	&R\left(
		\tfrac{2a}{c} + \left(\tfrac{2b}{c}-d^2\right)\tau ; 2d^2\tau	
	\right)
	\\
	&=
		R\left(
			\tfrac{2a}{c} + \left(\tfrac{b}{cd^2}-\tfrac{1}{2}\right)2d^2\tau ; 2d^2\tau	
		\right)	
	\\
	&=
		R\left(
			\tfrac{2a}{c} + \left(\tfrac{b}{cd^2}-n-\tfrac{1}{2}\right)2d^2\tau + 2d^2n\tau ; 2d^2\tau	
		\right)	
	\\
	&=
		(-1)^{n}
		\exp\left(
			2\pi in\left( \tfrac{2a}{c}+\left(\tfrac{2b}{c}-d^2n-d^2\right)\tau\right)  
		\right)
		R\left(
			\tfrac{2a}{c} + \left(\tfrac{b}{cd^2}-n-\tfrac{1}{2}\right)2d^2\tau ; 2d^2\tau	
		\right)	
		\\&\quad
		+
		2
		\exp\left(
			2\pi in\left( \tfrac{2a}{c}+\left(\tfrac{2b}{c}-d^2n-d^2\right)\tau\right)  
		\right)
		\\&\quad\times
		\sum_{m=1}^{|n|}
		(-1)^{m+n}\exp\left(
			-\pi i(2m-1)\mbox{sgn}(n)\left( \tfrac{2a}{c}+\left(\tfrac{2b}{c}-2d^2n-d^2\right)\tau\right)
			- 2\pi im(m-1)d^2\tau - \tfrac{\pi id^2\tau}{2}		
		\right)	
	\\
	&=
		(-1)^{n}
		\z{c}{2an}
		q^{n\left(\frac{2b}{c}-d^2n-d^2\right)}
		R\left(
			\tfrac{2a}{c} + \left(\tfrac{b}{cd^2}-n-\tfrac{1}{2}\right)2d^2\tau ; 2d^2\tau	
		\right)	
		\\&\quad
		+
		2
		\z{c}{2an}
		q^{n\left(\frac{2b}{c}-d^2n-d^2\right)}
		\sum_{m=1}^{|n|}
		(-1)^{m+n}
		\z{c}{-{\rm sgn}(n)(2m-1)a }		
		q^{ -(2m-1) {\rm sgn}(n)  \left( \frac{b}{c} - d^2n - \frac{d^2}{2}  \right) 
			- m(m-1)d^2 - \frac{d^2}{4}
		}	
.
\end{align*}

If $cd^2\mid b$, then $-\frac{1}{2} = \frac{b}{cd^2}-n-\frac{1}{2}$,
and so by (\ref{EqRToIntegralBoundaryCase}) we have that
\begin{align*}
	R\left(  
		\tfrac{2a}{c} + \left(\tfrac{b}{cd^2}-n-\tfrac{1}{2} \right)2d^2\tau ; 2d^2\tau
	\right)
	&=
	R\left(  
		\tfrac{2a}{c} - d^2\tau ; 2d^2\tau
	\right)
	\\	
	&=
		\z{c}{-a}q^{\frac{d^2}{4}}
		-
		i\sqrt{2}d\z{c}{-a}q^{\frac{d^2}{4}}
		\int_{-\overline{\tau}}^{i\infty}
			\frac{g_{0, \frac{1}{2}-\frac{2a}{c} } (2d^2w) }
			{\sqrt{-i(w+\tau)}}			
		dw
.
\end{align*}
Thus for $cd^2\mid b$ we have that
\begin{align*}
	&
	\z{c}{a}q^{-\frac{d^2}{4}-\frac{b^2}{c^2d^2}+\frac{b}{c}}
	R\left( \tfrac{2a}{c}+\left(\tfrac{2b}{c}-d^2\right)\tau; 2d^2\tau \right)
	-
	q^{-\frac{b^2}{c^2d^2}}	
	\\
	&=
		-q^{-\frac{b^2}{c^2d^2}}
		+
		(-1)^{n} \z{c}{(2n+1)a} q^{-\frac{d^2}{4}}
		R\left( \tfrac{2a}{c} - d^2\tau; 2d^2\tau	\right)
		\\&\quad		
		+
		2\z{c}{(2n+1)a} q^{-\frac{d^2}{4}}
		\sum_{m=1}^{|n|}
		(-1)^{m+n}\z{c}{-{\rm sgn}(n)(2m-1)a}
		q^{d^2\left( \frac{{\rm sgn}(n)(2m-1)}{2}-m(m-1)-\frac{1}{4} \right)}
	\\
	&=
		-q^{-\frac{b^2}{c^2d^2}}
		+
		(-1)^{n}\z{c}{2na}
		+
		2\sum_{m=1}^{|n|}
		(-1)^{m+n}\z{c}{(2n+1-{\rm sgn}(n)(2m-1))a}
		q^{d^2\left( \frac{{\rm sgn}(n)(2m-1)}{2}-m(m-1)-\frac{1}{2} \right)}
		\\&\quad		
		-
		(-1)^{n}i\sqrt{2}d\z{c}{2an}
		\int_{-\overline{\tau}}^{i\infty}
			\frac{g_{0,\frac{1}{2}-\frac{2a}{c}}(2d^2w)}
			{\sqrt{-i(w+\tau)}}		
		dw
	.
\end{align*}

If $cd^2\nmid b$ we instead have that $-\frac{1}{2} < \frac{b}{cd^2}-n-\frac{1}{2} < \frac{1}{2}$,
and so by (\ref{EqRToIntegral}) we have that
\begin{align*}
	&R\left(  
		\tfrac{2a}{c} + \left(\tfrac{b}{cd^2}-n-\tfrac{1}{2} \right)2d^2\tau ; 2d^2\tau
	\right)
	\\
	&=
		-\exp\left(
			\pi i\left(\tfrac{b}{cd^2}-n-\tfrac{1}{2}\right)^2 2d^2\tau		
			-			
			2\pi i\left(\tfrac{b}{cd^2}-n-\tfrac{1}{2}\right)\left(\tfrac{1}{2}-\tfrac{2a}{c}\right)		
		\right)
		\int_{-\overline{2d^2\tau}}^{i\infty}
			\frac{g_{\frac{b}{cd^2}-n, \frac{1}{2}-\frac{2a}{c} } (w) }
			{\sqrt{-i(w+2d^2\tau)}}			
		dw
	\\
	&=
		(-1)^{n+1}i\sqrt{2}d\z{c}{-(2n+1)a}			
		\exp\left( 2\pi i\left(
			\tfrac{2ab}{c^2d^2} - \tfrac{b}{2cd^2}
		\right)\right)
		q^{d^2\left( \frac{b}{cd^2}-n-\frac{1}{2} \right)^2}
		\int_{-\overline{\tau}}^{i\infty}
			\frac{g_{\frac{b}{cd^2}-n, \frac{1}{2}-\frac{2a}{c} } (2d^2w) }
			{\sqrt{-i(w+\tau)}}			
		dw
.
\end{align*}
Thus for $cd^2\nmid b$, we have that
\begin{align*}
	&
	\z{c}{a}q^{-\frac{d^2}{4}-\frac{b^2}{c^2d^2}+\frac{b}{c}}
	R\left( \tfrac{2a}{c}+\left(\tfrac{2b}{c}-d^2\right)\tau; 2d^2\tau \right)
	-
	q^{-\frac{b^2}{c^2d^2}}	
	\\
	&=
		-
		q^{-\frac{b^2}{c^2d^2}}	
		\\&\quad
		+
		2\sum_{m=1}^{|n|}
		(-1)^{m+n}\z{c}{(2n+1-{\rm sgn}(n)(2m-1))a}
		q^{ n\left(\frac{2b}{c}-d^2n-d^2\right) 
			- (2m-1){\rm sgn(n)}\left(\frac{b}{c}-d^2n-\frac{d^2}{2}\right)
			- m(m-1)d^2 - \frac{d^2}{2} + \frac{b}{c} - \frac{b^2}{c^2d^2}    
		}		
		\\&\quad
		-
		i\sqrt{2}d\exp\left(2\pi i\left( \tfrac{2ab}{c^2d^2}-\tfrac{b}{2cd^2} \right)\right)
		\int_{-\overline{\tau}}^{i\infty}
			\frac{g_{\frac{b}{cd^2}-n,\frac{1}{2}-\frac{2a}{c}} (2d^2w)}
			{\sqrt{-i(w+\tau)}}
		dw	
	.
\end{align*}
\end{proof}

\noindent Following are the cases when $d$ is even. Since the proofs are quite
similar to the proof when $d$ is odd, we only state the results.

\begin{proposition}\label{prop:Oerr2}
Suppose $d\equiv 2 \pmod 4$, and let $n=\lfloor \frac{2b}{cd^2} \rfloor$.  If $cd^2 \mid 2b$, then 
\begin{align*}
\widetilde{\mathcal{O}}_d(a,b,c;\tau)
&= 
	\frac{(1+\z{c}{a}q^{b/c})}{(1-\z{c}{a}q^{b/c})}q^{-\frac{b^2}{c^2d^2}}\mathcal{O}_d(\z{c}{a}q^{\frac{b}{c}}; q) 
	- q^{\frac{-b^2}{c^2d^2}} + (-1)^n \zeta_c^{an} 
	\\&\quad 
	+ 2 \sum_{m=1}^{|n|} (-1)^{m+n} \zeta_{2c}^{a(2n+1-{\rm sgn}(n)(2m-1))} q^{\frac{d^2}{8}((2m-1){\rm sgn}(n) -2m(m-1) -1)} 
	\\&\quad
	 - \frac{d}{\sqrt{2}} i (-1)^n \zeta_c^{an} 
	 \int_{-\bar{\tau}}^{i\infty} \frac{g_{0,\frac32 -\frac{a}{c}}(\frac{d^2w}{2})}{\sqrt{-i(w+\tau)}}dw;
\end{align*}
in particular, if $b=0$, then
\[
\widetilde{\mathcal{O}}_d(a,0,c;\tau)
= 
	\frac{(1+\z{c}{a})}{(1-\z{c}{a})}\mathcal{O}_d(\z{c}{a}; q) 
	- 
	\frac{d}{\sqrt{2}} i \int_{-\bar{\tau}}^{i\infty} \frac{g_{0,\frac32 -\frac{a}{c}}(\frac{d^2w}{2})}{\sqrt{-i(w+\tau)}}dw.
\]
If $cd^2 \nmid 2b$, then 
\begin{align*}
&\widetilde{\mathcal{O}}_d(a,b,c;\tau)
= 
	\frac{(1+\z{c}{a}q^{b/c})}{(1-\z{c}{a}q^{b/c})}q^{-\frac{b^2}{c^2d^2}}\mathcal{O}_d(\z{c}{a}q^{\frac{b}{c}}; q) 
	- q^{\frac{-b^2}{c^2d^2}} 
	\\&\quad
	+2 \sum_{m=1}^{|n|} (-1)^{m+n} \zeta_{2c}^{a(2n+1-{\rm sgn}(n)(2m-1))} 
	q^{n(\frac{b}{c} - \frac{d^2n}{4} -\frac{d^2}{4}) - (2m-1){\rm sgn}(n)(\frac{b}{2c} 
		- \frac{d^2n}{4} - \frac{d^2}{8}) -\frac{d^2}{4} m(m-1) -\frac{d^2}{8} + \frac{b}{2c} - \frac{b^2}{c^2d^2}} 
	\\&\quad
	- \frac{d}{\sqrt{2}} i \exp\left(2\pi i \left(\tfrac{2ab}{c^2d^2} - \tfrac{3b}{cd^2}\right)\right) 
	\int_{-\bar{\tau}}^{i\infty} \frac{g_{\frac{2b}{cd^2} -n,\frac32 -\frac{a}{c}}(\frac{d^2w}{2})}{\sqrt{-i(w+\tau)}}dw;
\end{align*}
in particular, if $0<b<cd^2/2$, then
\begin{align*}
\widetilde{\mathcal{O}}_d(a,b,c;\tau)
&= 
	\frac{(1+\z{c}{a}q^{b/c})}{(1-\z{c}{a}q^{b/c})}q^{-\frac{b^2}{c^2d^2}}\mathcal{O}_d(\z{c}{a}q^{\frac{b}{c}}; q) 
	- q^{\frac{-b}{c^2d^2}} 
	\\&\quad
	- \frac{d}{\sqrt{2}} i \exp\left(2\pi i \left(\tfrac{2ab}{c^2d^2} - \tfrac{3b}{cd^2}\right)\right) 
	\int_{-\bar{\tau}}^{i\infty} \frac{g_{\frac{2b}{cd^2},\frac32 -\frac{a}{c}}(\frac{d^2w}{2})}{\sqrt{-i(w+\tau)}}dw.
\end{align*}
\end{proposition}

\begin{proposition}\label{prop:Oerr4}
Suppose $d\equiv 0 \pmod 4$, and let $n=\lfloor \frac{2b}{cd^2} \rfloor$.  If $cd^2 \mid 2b$, then 
\begin{align*}
\widetilde{\mathcal{O}}_d(a,b,c;\tau)
&= 
	\frac{(1+\z{c}{a}q^{b/c})}{(1-\z{c}{a}q^{b/c})}q^{-\frac{b^2}{c^2d^2}}\mathcal{O}_d(\z{c}{a}q^{\frac{b}{c}}; q)
	- q^{\frac{-b^2}{c^2d^2}}  + \zeta_{c}^{an} 
	\\&\quad
	-2i \sum_{m=1}^{|n|} (-1)^{m}i^{(2m-1){\rm sgn}(n)} \zeta_{2c}^{a(2n+1-{\rm sgn}(n)(2m-1))}
		q^{\frac{d^2}{8}((2m-1){\rm sgn}(n) -2m(m-1) -1)} 
	\\&\quad
	-\frac{d}{\sqrt{2}} i \zeta_{c}^{an}  
	\int_{-\bar{\tau}}^{i\infty} \frac{g_{0,1-\frac{a}{c}}(\frac{d^2w}{2})}{\sqrt{-i(w+\tau)}}dw;
\end{align*}
in particular, if $b=0$, then
\[
\widetilde{\mathcal{O}}_d(a,0,c;\tau)
= 
	\frac{(1+\z{c}{a})}{(1-\z{c}{a})}\mathcal{O}_d(\z{c}{a}; q) 
	-
	\frac{d}{\sqrt{2}} i \int_{-\bar{\tau}}^{i\infty} \frac{g_{0,1-\frac{a}{c}}(\frac{d^2w}{2})}{\sqrt{-i(w+\tau)}}dw.
\]
If $cd^2 \nmid 2b$, then 
\begin{align*}
\widetilde{\mathcal{O}}_d(a,b,c;\tau)
&= 
	\frac{(1+\z{c}{a}q^{b/c})}{(1-\z{c}{a}q^{b/c})}q^{-\frac{b^2}{c^2d^2}}\mathcal{O}_d(\z{c}{a}q^{\frac{b}{c}}; q)
	- q^{\frac{-b^2}{c^2d^2}} 
	\\&\quad
	-2i \sum_{m=1}^{|n|} (-1)^{m}i^{(2m-1){\rm sgn}(n)} \zeta_{2c}^{a(2n+1-{\rm sgn}(n)(2m-1))} 
	\\&\qquad\times
	q^{n(\frac{b}{c} - \frac{d^2n}{4} -\frac{d^2}{4}) - (2m-1){\rm sgn}(n)(\frac{b}{2c} - \frac{d^2n}{4} - \frac{d^2}{8}) 
		-\frac{d^2}{4} m(m-1) -\frac{d^2}{8} + \frac{b}{2c} - \frac{b^2}{c^2d^2}} 
	\\&\quad
	-\frac{d}{\sqrt{2}}  i \exp\left(2\pi i \left(\tfrac{2ab}{c^2d^2} - \tfrac{2b}{cd^2}\right)\right) 
	\int_{-\bar{\tau}}^{i\infty} \frac{g_{\frac{2b}{cd^2} -n, 1-\frac{a}{c}}(\frac{d^2w}{2})}{\sqrt{-i(w+\tau)}}dw;
\end{align*}
in particular, if $0<b<cd^2/2$, then
\begin{align*}
\widetilde{\mathcal{O}}_d(a,b,c;\tau)
&= 
	\frac{(1+\z{c}{a}q^{b/c})}{(1-\z{c}{a}q^{b/c})}q^{-\frac{b^2}{c^2d^2}}\mathcal{O}_d(\z{c}{a}q^{\frac{b}{c}}; q)
	- q^{\frac{-b^2}{c^2d^2}} 
	\\&\quad
	-\frac{d}{\sqrt{2}}  i \exp\left(2\pi i \left(\tfrac{2ab}{c^2d^2} - \tfrac{2b}{cd^2}\right)\right) 
	\int_{-\bar{\tau}}^{i\infty} \frac{g_{\frac{2b}{cd^2}, 1-\frac{a}{c}}(\frac{d^2w}{2})}{\sqrt{-i(w+\tau)}}dw.
\end{align*}
\end{proposition}

\section{ Modular interpretation of $\mathcal{O}_d(\zeta_c^a q^\frac{b}c; q)$}

In this section we obtain explicit transformation formulas for the functions 
$\widetilde{\mathcal{O}}_d(a,b,c;\tau)$.  In light of \eqref{eqn:OviaM} we 
begin by studying transformation formulas for the functions 
$\widetilde{M}_{d,k}(a,b,c;\tau)$.  First, we note the following elliptic 
properties of these functions.

\begin{proposition}\label{PropDOddMShifts}
We have that
\begin{align*}
	\widetilde{M}_{d,k+d}(a,b,c;\tau)
	&=
		\widetilde{M}_{d,k}(a,b,c;\tau)
	,\\
	\widetilde{M}_{d,k}(a+c,b,c;\tau)
	&=
		\widetilde{M}_{d,k}(a,b,c;\tau)
	.
\end{align*}	
Moreover if $d$ is odd,
\[
\widetilde{M}_{d,k}(a,b+cd^2,c;\tau)= - \z{c}{2a}\widetilde{M}_{d,k}(a,b,c;\tau),
\]
and if $d$ is even,
\[
\widetilde{M}_{d,k}\left(a,b+\frac{cd^2}{2},c;\tau\right)= (-1)^{\frac{d}{2}} \z{c}{a}\widetilde{M}_{d,k}(a,b,c;\tau).
\]

\end{proposition}
\begin{proof}
These all follow immediately from (\ref{EqZTheorem1.11P1}).
\end{proof}

We now determine a transformation for $\widetilde{M}_{d,k}(a,b,c;\tau)$, under the action of $\SLTwo$, in terms of $\tmu(u,v;\tau)$.

\begin{proposition}\label{prop:MtransformationSL2}
Let $A=\begin{pmatrix}\alpha & \beta \\ \gamma & \delta \end{pmatrix}\in\SLTwo$.  Then, when $d$ is odd,
\begin{align*}
\widetilde{M}_{d,k}(a,b,c;A\tau) 
&= 
	2\z{c}{a} \exp\left( -\tfrac{\pi i }{c^2} 
	\left( \tfrac{\alpha\beta(2b-cd^2)^2 }{4d^4}	+ 4a^2\gamma\delta+ \tfrac{2a\beta\gamma(2b-cd^2)}{d^2}	 \right)\right) 
	\nu(A)^{-3} \sqrt{\gamma\tau+\delta} 
	\\&\quad\times
	q^{-\frac{(4ad^2\gamma + 2b\alpha - cd^2\alpha)^2}{8d^4c^2}} 
	\tmu\left( \tfrac{2a\delta}{c} + \tfrac{b\beta}{cd^2} + \tfrac{k\beta}{d} + \left(\tfrac{2ad^2\gamma+b\alpha}{cd^2}+\tfrac{k\alpha}{d}\right)\tau
		, \left( \tfrac{1}{2} + \tfrac{k}{d} \right)\alpha\tau + \tfrac{\beta}{2} + \tfrac{k\beta}{d}; \tau \right)
	,
\end{align*}
when $d\equiv 2 \pmod 4$,
\begin{align*}
&\widetilde{M}_{d,k}\left(a,b,c;A\tau\right) 
=
	2\zeta_c^{\frac{a}{2}} \exp\left(-\tfrac{\pi i}{c^2} 
	\left(\tfrac{\alpha\beta(4b-cd^2)^2}{4d^4} + \gamma\delta(a-c)^2 + \tfrac{\beta\gamma(a-c)(4b-cd^2)}{d^2} \right) \right) 
	\nu(A)^{-3} \sqrt{\gamma\tau+\delta} 
	\\&~\times
	q^{- \frac{(2d^2\gamma(a-c) + \alpha(4b-cd^2))^2}{8c^2d^4}} 
	\tmu\left( \left(\tfrac{a\delta}{c} + \tfrac{2b\beta}{cd^2} + \tfrac{2k\beta}{d} \right) 
		+ \left(\tfrac{a\gamma}{c} + \tfrac{2b\alpha}{cd^2} + \tfrac{2k\alpha}{d} \right) \tau
		, \left(\delta+ \tfrac{\beta}{2} + \tfrac{2k\beta}{d}\right) + \left( \gamma + \tfrac{\alpha}{2} + \tfrac{2k\alpha}{d}\right) \tau 
		; \tau \right),
\end{align*}
and when $d \equiv 0 \pmod 4$, 
\begin{align*}
&\widetilde{M}_{d,k}\left(a,b,c;A\tau\right) 
= 
	2\zeta_c^{\frac{a}{2}}\exp\left(-\frac{\pi i}{c^2} 
	\left( \tfrac{\alpha\beta(4b-cd^2)^2}{4d^4} + \tfrac{\gamma\delta(2a-c)^2}{4} + \tfrac{\beta\gamma(2a-c)(4b-cd^2)}{2d^2} \right) \right) 
	\nu(A)^{-3} \sqrt{\gamma\tau+\delta}
	\\&~ \times
	q^{- \frac{(d^2\gamma(2a-c) + \alpha(4b-cd^2))^2}{8c^2d^4}} 
 	\tmu\left( \left(\tfrac{a\delta}{c} + \tfrac{2b\beta}{cd^2} + \tfrac{2k\beta}{d} \right) + \left(\tfrac{a\gamma}{c} 
 	+ \tfrac{2b\alpha}{cd^2} + \tfrac{2k\alpha}{d} \right) \tau
 		, \left(\tfrac{\delta}{2} + \tfrac{\beta}{2} + \tfrac{2k\beta}{d}\right) 
 		+ \left( \tfrac{\gamma}{2} + \tfrac{\alpha}{2} + \tfrac{2k\alpha}{d}\right) \tau ; \tau  \right)
 	.
\end{align*}
\end{proposition}
\begin{proof}
We only give the proof for $d$ odd, as the other cases are similar.
When $d$ is odd, we see by (\ref{EqZTheorem1.11P2}) that
\begin{align*}
	&\widetilde{M}_{d,k}(a,b,c;A\tau)
	\\	
	&=
		2\z{c}{a}\exp\left(-\tfrac{\pi i(2b-cd^2)^2A\tau}{4d^4c^2}\right)
		\tmu\left(
			\tfrac{2a}{c} + \left(\tfrac{b}{cd^2}+\tfrac{k}{d}\right)A\tau,
			\left( \tfrac{1}{2} + \tfrac{k}{d} \right)A\tau ; A\tau	
		\right)
	\\
	&=
		2\z{c}{a}\exp\left(-\tfrac{\pi i(2b-cd^2)^2A\tau}{4d^4c^2}\right)
		\exp\left(	
			-\tfrac{\pi i\gamma}{\gamma\tau+\delta} 
			\left(  
			\tfrac{2a(\gamma\tau+\delta)}{c} + (\tfrac{b}{cd^2}-\tfrac{1}{2})(\alpha\tau+\beta)
			 \right)^2
		\right)
		\nu(A)^{-3}
		\sqrt{\gamma\tau+\delta}
		\\&\quad\times
		\tmu\left(
			\tfrac{2a(\gamma\tau+\delta)}{c} + \left(\tfrac{b}{cd^2}+\tfrac{k}{d}\right)(\alpha\tau+\beta),
			\left( \tfrac{1}{2} + \tfrac{k}{d} \right)(\alpha\tau+\beta) ; \tau	
		\right)
	\\
	&=
		2\z{c}{a}
		\exp\left( -\tfrac{\pi i(\alpha^2\tau+\alpha\beta)(2b-cd^2)^2 }{4d^4c^2} \right)
		\exp\left(	
			-\pi i\gamma 
			\left(  
				\tfrac{4a^2(\gamma\tau+\delta)}{c^2}			
				+
				\tfrac{2a(\alpha\tau+\beta)(2b-cd^2)}{c^2d^2}						
			 \right)
		\right)
		\nu(A)^{-3}
		\sqrt{\gamma\tau+\delta}
		\\&\quad\times		
		\tmu\left(
			\tfrac{2a(\gamma\tau+\delta)}{c} + \left(\tfrac{b}{cd^2}+\tfrac{k}{d}\right)(\alpha\tau+\beta),
			\left( \tfrac{1}{2} + \tfrac{k}{d} \right)(\alpha\tau+\beta) ; \tau	
		\right)	
	.
\end{align*}
The result then follows after simplifying.
\end{proof}

In order to obtain transformation properties for the 
$\widetilde{M}_{d,k}(a,b,c;\tau)$ functions at the appropriate parameters 
coming from \eqref{eqn:OviaM}, it will be useful to first define the following 
groups and note their implications for elements.  We define
\[
\G_{a,b,c,d} :=   \G_0\left( \frac{c^2d^2}{\gcd(a^2,c^2d^2)} \right)  \cap \G_1 \left( \frac{c}{\gcd(a,c)} \right)   \cap \G^0 \left(\frac{c}{\gcd(b,c)}  \right) \cap \G^0 \left(\frac{c^2d^2}{\gcd(b^2,c^2d^2)}  \right). 
\]
Thus if $A=\begin{psmallmatrix}\alpha&\beta\\\gamma&\delta\end{psmallmatrix}\in\G_{a,b,c,d}$, 
then the entries in $A$ satisfy the following congruences:
\begin{align}\label{eqn:c1}
a^2\gamma &\equiv 0\pmod{c^2d^2},& 
a(\alpha -1) \equiv a(\delta -1) &\equiv 0 \pmod{c}
,\nonumber\\
b\beta & \equiv 0 \pmod{c},&  
b^2\beta &\equiv 0 \pmod{c^2d^2}.  
\end{align} 
Furthermore, we define 
\[
\G_{a,b,c,d}' := \begin{cases}
\G_0\left( \frac{2cd^2}{\gcd(a,2cd^2)} \right) \cap \G_1 \left(\frac{c^2d^2}{\gcd(ab,c^2d^2)}  \right) \cap \G_1\left( \frac{2cd^2}{\gcd(b,2cd^2)} \right)  & \text{ when } d \text{ odd}, \\
\G_0\left( \frac{cd^2}{\gcd(a,cd^2)} \right) \cap \G_1 \left(\frac{c^2d^2}{\gcd(2ab,c^2d^2)}  \right) \cap \G_1\left( \frac{cd^2}{\gcd(b,cd^2)} \right)  & \text{ when } d\equiv 2 \pmod{4}, \\
\G_0\left( \frac{cd^2}{2\gcd(a,\frac{cd^2}{2})} \right) \cap \G_1 \left(\frac{c^2d^2}{\gcd(2ab,c^2d^2)}  \right) \cap \G_1\left( \frac{cd^2}{2\gcd(b,\frac{cd^2}{2})} \right)  & \text{ when }d\equiv 0 \pmod{4},
\end{cases}
\]
so that if $A=\begin{psmallmatrix}\alpha&\beta\\\gamma&\delta\end{psmallmatrix}\in\G_{a,b,c,d}'$, 
then when $d$ is odd we have that
\begin{align}\label{eqn:codd}
a\gamma &\equiv 0 \pmod {cd^2}
,&  
ab(\alpha -1) \equiv ab(\delta -1) &\equiv 0 \pmod{c^2d^2}
,\nonumber\\ 
b(\alpha-1) \equiv b(\delta -1) & \equiv 0 \pmod{cd^2};
\end{align}
when $d\equiv 2\pmod{4}$ we have that
\begin{align}\label{eqn:c2}
a\gamma &\equiv 0 \pmod {cd^2}  
,&
2ab(\alpha -1) \equiv 2ab(\delta -1) &\equiv 0 \pmod{c^2d^2} 
,\nonumber\\
b(\alpha-1) \equiv b(\delta -1) & \equiv 0 \pmod{cd^2} ;
\end{align}
and when $d\equiv 0\pmod{4}$ we have that
\begin{align}\label{eqn:c0}
a\gamma &\equiv 0 \pmod {cd^2/2}
,&
2ab(\alpha -1) \equiv 2ab(\delta -1) &\equiv 0 \pmod{c^2d^2}
,\nonumber\\
b(\alpha-1) \equiv b(\delta -1) & \equiv 0 \pmod{cd^2/2} 
. 
\end{align}

We now give transformations for $\widetilde{M}_{d,k}(a,b,c;\tau)$
as a corollary of Proposition \ref{prop:MtransformationSL2}, depending on the $2$-divisibility of $d$.

\begin{corollary}\label{cor:Mtransf1}
If $d$ is odd
and $A=\begin{pmatrix}\alpha & \beta \\ \gamma & \delta \end{pmatrix}
\in \Gamma_0(2d^2)\cap \G_{a,b,c,d} \cap \G_{a,b,c,d}' \cap \G_1\left(\frac{d}{\gcd(d,k)} \right)$, then
\begin{align*}
	\widetilde{M}_{d,k}(a,b,c; 2d^2A\tau)
	&=
		(-1)^{\beta + \frac{\alpha-1}{2}}
		i^{-\alpha\beta}		
		\,\nu( \mA{2} )^{-3} \sqrt{\gamma\tau+\delta}
		\widetilde{M}_{d,k}(a, b, c; 2d^2\tau)		
	.
\end{align*}
If $d \equiv 2 \pmod{4}$ and $A=\begin{pmatrix}\alpha & \beta \\ \gamma & \delta \end{pmatrix}\in \G_0(\frac{d^2}{2}) \cap \G_{a,b,c,d} \cap \G_{a,b,c,d}' \cap \G_1\left(\frac{d}{\gcd(d,k)} \right)$, then
\begin{align*}
\widetilde{M}_{d,k}\left(a,b,c; \frac{d^2}{2}A\tau\right) 
&= 
	(-1)^{\beta + \frac{\alpha-1}{2}} i^{-\alpha \beta} \nu(\mA{2})^{-3} \sqrt{\gamma\tau+\delta} 
	\widetilde{M}_{d,k}\left(a,b,c; \frac{d^2\tau}{2}\right).
\end{align*}
If $d\equiv 0 \pmod{4}$ and $A=\begin{pmatrix}\alpha & \beta \\ \gamma & \delta \end{pmatrix}\in\Gamma_0(4d^2)\cap \G_{a,b,c,d} \cap \G_{a,b,c,d}' \cap \G_1\left(\frac{d}{\gcd(d,k)} \right)$, then
\begin{align*}
\widetilde{M}_{d,k}\left(a,b,c; \frac{d^2}{2}A\tau\right) 
=  
	(-1)^{\beta+\frac{\alpha -1}{2}}i^{-\alpha\beta} \nu(\mA{2})^{-3} \sqrt{\gamma\tau+\delta} 
	\widetilde{M}_{d,k}\left(a, b ,c; \frac{d^2\tau}{2}\right).
\end{align*}
\end{corollary}
\begin{proof}
First we consider when $d$ is odd. Since 
$A=\begin{psmallmatrix}\alpha&\beta\\\gamma&\delta\end{psmallmatrix}\in\Gamma_0(2d^2)$, by 
Proposition \ref{prop:MtransformationSL2} we have that
\begin{align*}
	&\widetilde{M}_{d,k}(a,b,c; 2d^2 A\tau)
	=
	\widetilde{M}_{d,k}(a,b,c; \mA{2d^2} 2d^2\tau)
	\\
	&=
		2\z{c}{a} 
		\exp\left(-\tfrac{\pi i}{c^2d^2}\left( 
			\tfrac{\alpha\beta(2b-cd^2)^2}{2}
			+		
			2a^2\gamma\delta	
			+
			2a\beta\gamma(2b-cd^2)			
		\right)\right)
		\nu(\mA{2d^2})^{-3} \sqrt{\gamma\tau+\delta}
		\\&\quad\times
		q^{ - \frac{(\alpha(2b-cd^2) + 2a\gamma)^2}{4c^2d^2} }		
		\tmu\left( 
			\tfrac{2a\delta+2b\beta}{c} + 2kd\beta + \tfrac{2a\gamma+2b\alpha}{c}\tau + 2dk\alpha\tau, 
			d^2\alpha\tau + 2dk\alpha\tau + d^2\beta + 2kd\beta			
			; 2d^2\tau 
		\right)	
.
\end{align*}
To proceed we note that $\alpha\equiv 1\pmod{2}$, so writing
$d^2\alpha\tau = d^2\tau + \frac{(\alpha-1)}{2}2d^2\tau$, together with 
(\ref{EqZTheorem1.11P1}) gives that
\begin{align*}
	&\tmu\left( 
		\tfrac{2a\delta+2b\beta}{c} + 2kd\beta + \tfrac{2a\gamma+2b\alpha}{c}\tau + 2dk\alpha\tau,  
		d^2\alpha\tau + 2dk\alpha\tau  + d^2\beta + 2kd\beta ; 2d^2\tau		
	\right)
	\\
	&=
		(-1)^{\beta+\frac{\alpha-1}{2}}	
		\exp\left( -\tfrac{2\pi i(\alpha-1)(a\delta+b\beta)}{c}	\right)	
		q^{ \frac{d^2(\alpha-1)^2}{4} + \frac{d^2(\alpha-1)}{2} - \frac{(\alpha-1)(a\gamma+b\alpha)}{c} }
		\\&\quad\times		
		\tmu\left( 
			\tfrac{2a\delta+2b\beta}{c} + \tfrac{2a\gamma+2b\alpha}{c}\tau + 2dk\alpha\tau,  
			d^2\tau + 2dk\alpha\tau ; 2d^2\tau		
		\right)
.
\end{align*}
Along with the fact that $\nu(\mA{2d^2})^3=\nu(\mA{2})^3$, we then have that
\begin{align*}
	&\widetilde{M}_{d,k}(a,b,c; 2d^2 A\tau)
	\\
	&=
		2(-1)^{\beta + \frac{\alpha-1}{2}}
		i^{-\alpha\beta}
		\z{c}{a - (\alpha-1)(a\delta+b\beta) + a\beta(b+\gamma)}
		\z{c^2d^2}{-b^2\alpha\beta-a^2\gamma\delta-2ab\beta\gamma} 
		\,\nu(\mA{2})^{-3} \sqrt{\gamma\tau+\delta}
		\\&\quad\times
		q^{ - \frac{ (2a\gamma+2b\alpha-cd^2)^2}{4c^2d^2} }		
		\tmu\left( \tfrac{2a\delta+2b\beta}{c}
			+\tfrac{2a\gamma+2b\alpha}{c}\tau + 2dk\alpha\tau, 
			d^2\tau + 2dk\alpha\tau; 2d^2\tau 
		\right),	
\end{align*}		
and so we see that 
\begin{align}\label{PropDOddMTransformationInitial}
\widetilde{M}_{d,k}(a,b,c; 2d^2 A\tau)
&= 
	(-1)^{\beta + \frac{\alpha-1}{2}}i^{-\alpha\beta}	
	\z{c}{a - \alpha(a\delta+b\beta) + a\beta(b+\gamma)}\z{c^2d^2}{-b^2\alpha\beta-a^2\gamma\delta-2ab\beta\gamma} 
	\nonumber\\&\quad\times
	\nu(\mA{2})^{-3} \sqrt{\gamma\tau+\delta} \widetilde{M}_{d,k\alpha}(a\delta+b\beta, a\gamma+b\alpha, c; 2d^2\tau)
.		
\end{align}

But now considering our full group $\Gamma_0(2d^2)\cap \G_{a,b,c,d} \cap \G_{a,b,c,d}' \cap \G_1\left(\frac{d}{\gcd(d,k)} \right)$, we observe $k\alpha \equiv k\pmod{d}$, as well as the congruences from \eqref{eqn:c1} and \eqref{eqn:codd}.  Together with Proposition \ref{PropDOddMShifts}, this allows us to simplify \eqref{PropDOddMTransformationInitial} to obtain 
\begin{align*}
	&\widetilde{M}_{d,k}(a,b,c;2d^2A\tau)
	\\
	&=	
		(-1)^{\beta + \frac{\alpha-1}{2} + \frac{a\gamma}{cd^2} + \frac{b(\alpha-1)}{cd^2}}
		i^{-\alpha\beta}		
		\z{c}{a - \alpha(a\delta+b\beta) + a\beta(b+\gamma) + \frac{2a^2\gamma}{cd^2} + \frac{2ab(\alpha-1)}{cd^2}  }
		\z{c^2d^2}{-b^2\alpha\beta - a^2\gamma\delta - 2ab\beta\gamma } 
		\,\nu(\mA{2})^{-3} 
		\\&\quad\times		
		\sqrt{\gamma\tau+\delta}
		\widetilde{M}_{d,k}(a, b, c; 2d^2\tau)		
	\\
	&=
		(-1)^{\beta + \frac{\alpha-1}{2} }
		i^{-\alpha\beta}		
		\,\nu( \mA{2} )^{-3} \sqrt{\gamma\tau+\delta}
		\widetilde{M}_{d,k}(a, b, c; 2d^2\tau)		
\end{align*}
as desired.

The proofs when $d$ is even are much the same. However,
when $d\equiv 2 \pmod{4}$ we have $\frac{d^2}{2}=2m^2$ for odd $m$, and so 
the discussion preceding \eqref{EqEtaMultipler} gives that 
$\nu(\mA{\frac{d^2}{2}})^{3}= \nu(\mA{2})^{3}$ for 
$A\in \Gamma_0(\frac{d^2}{2})$.
When $d\equiv 0\pmod{4}$ we have  that $\frac{d^2}{2}=2m^2$ for even $m$, and 
so the discussion proceeding \eqref{EqEtaMultipler} gives that
$\nu(\mA{\frac{d^2}{2}})^{3}= i^{-\alpha\beta}\nu(\mA{2})^{3}$ for 
$A\in \Gamma_0(4d^2)$.
\end{proof}

\subsection{Maass form properties for $\widetilde{M}_{d,k}(a,b,c;\tau)$}

Furthermore, we obtain Maass form properties for the functions 
$\widetilde{M}_{d,k}(a,b,c;\tau)$ at the appropriate parameters coming from 
\eqref{eqn:OviaM}.  We first consider when $d$ is odd.

\begin{proposition}\label{PropDOddMdkMaassForm}
Suppose $d$ is odd. Then $\widetilde{M}_{d,k}(a,b,c;2d^2\tau)$ has at most linear exponential growth at
the cusps and is annihilated by $\Delta_{\frac{1}{2}}$.
\end{proposition}
\begin{proof}
We first check the growth condition at the cusps. Suppose $\alpha$ and $\gamma$
are integers with $\gcd(\alpha,\gamma)=1$, then take
$A=\begin{psmallmatrix}\alpha&\beta\\\gamma&\delta\end{psmallmatrix}\in\SLTwo$.   To obtain the growth of $\widetilde{M}_{d,k}(a,b,c;2d^2\tau)$ at $\tau=\frac{\alpha}{\gamma}$, we study the growth of 
$\widetilde{M}_{d,k}(a,b,c;2d^2A\tau)$ at infinity. We set
$g=\gcd(2d^2,\gamma)$, $m=\frac{2d^2\alpha}{g}$, and $u=\frac{\gamma}{g}$. We 
then take
$L=\begin{psmallmatrix}m&n\\u&v\end{psmallmatrix}\in\SLTwo$
and set
\begin{align*}
	B
	&=
		L^{-1}\TwoTwoMatrix{2d^2\alpha}{2d^2\beta}{\gamma}{\delta}
	=
		\TwoTwoMatrix{g}{2d^2\beta v- \delta n}{0}{2d^2/g}
.
\end{align*}
By Proposition \ref{prop:MtransformationSL2} we have
\begin{align}\label{EqDOddMdkCuspCondition}
	(\gamma\tau+\delta)^{-\frac{1}{2}}\widetilde{M}_{d,k}(a,b,c;2d^2A\tau)
	&=
		(\gamma\tau+\delta)^{-\frac{1}{2}}\widetilde{M}_{d,k}(a,b,c;LB\tau)
	\nonumber\\	
	&=
		\varepsilon_1
		\exp\left( -2\pi iB\tau\left(
			\tfrac{(4ad^2u+2bm-cd^2m)^2}{8d^4c^2}	
		\right)\right)
		\nonumber\\&\quad\times		
		\tmu\left(
			\tfrac{2av}{c}+\tfrac{bn}{cd^2}+\tfrac{kn}{d}
			+
			\left(\tfrac{2ad^2u+bm}{cd^2}+\tfrac{km}{d}\right)B\tau,
			\left(\tfrac{1}{2}+\tfrac{k}{d}\right)mB\tau 
			+ 
			\tfrac{n}{2}+\tfrac{kn}{d}; B\tau		
		\right)
,
\end{align}
\sloppy where $\varepsilon_1$ is some nonzero constant. Noting
$\exp(2\pi iB\tau) = \varepsilon_2 q^{\frac{g^2}{2d^2}}$, where
$|\varepsilon_2|=1$, we see that 
$(\gamma\tau+\delta)^{-\frac{1}{2}}\widetilde{M}_{d,k}(a,b,c;2d^2A\tau)$
has at worst linear exponential growth as $\mbox{Im}(\tau)\rightarrow\infty$,
since it is of the form
$q^{-\frac{(u_1-v_1)^2}{2}}\tmu\left(u_1\tau+u_2,v_1\tau+v_2 ;\tau\right)$ as in \eqref{eqn:growth_check}.

\fussy
To verify $\widetilde{M}_{d,k}(a,b,c;2d^2\tau)$ is annihilated by 
$\Delta_{\frac{1}{2}}$, we just need to verify that
the function
\begin{align*}
	\sqrt{\mbox{Im}(\tau)}\frac{\partial}{\partial\overline{\tau}}
	q^{\frac{b}{c}-\frac{d^2}{4}-\frac{b^2}{c^2d^2}}
	R\left( (\tfrac{b}{cd^2}-\tfrac{1}{2})2d^2\tau + \tfrac{2a}{c} ; 2d^2\tau \right)
\end{align*}
is holomorphic in $\overline{\tau}$. Using Lemma 1.8 of \cite{Zwegers} we find 
that
\begin{align*}
	\frac{\partial}{\partial\overline{\tau}}
	R\left( (\tfrac{b}{cd^2}-\tfrac{1}{2})2d^2\tau + \tfrac{2a}{c} ; 2d^2\tau \right)
	&=
	\frac{1}{\sqrt{\mbox{Im}(\tau)}}
	\exp\left( 2\pi i(\tau - \overline{\tau})\left(
		\tfrac{b^2}{c^2d^2}-\tfrac{b}{c}+\tfrac{d^2}{4}
	\right)\right)
	A(\overline{\tau})
,
\end{align*}
where $A(\overline{\tau})$ is a series defining a holomorphic function of $\overline{\tau}$.
Thus the result follows.
\end{proof}

We next consider when $d$ is even.

\begin{proposition}\label{prop:devenMaass}
Suppose $d$ is even. Then $\widetilde{M}_{d,k}(a,b,c;\frac{d^2\tau}{2})$ has at most linear exponential growth at the cusps and is annihilated by $\Delta_{\frac12}$.
\end{proposition}
\begin{proof}
First, suppose $\alpha$ and $\gamma$ are integers such that $\gcd(\alpha,\gamma)=1$, and 
take $A=\begin{psmallmatrix}\alpha&\beta\\\gamma&\delta\end{psmallmatrix}\in\SLTwo$.  
To obtain the growth of $\widetilde{M}_{d,k}(a,b,c;\frac{d^2\tau}{2})$ at 
$\tau=\frac{\alpha}{\gamma}$, we study the growth of 
$\widetilde{M}_{d,k}(a,b,c;\frac{d^2}{2}A\tau)$ at infinity.  
Let $g=\gcd(\frac{d^2\alpha}{2}, \gamma)$ so that $m=\frac{d^2\alpha}{2g}$ and 
$u=\frac{\gamma}{g}$ are relatively prime, and take 
$L=\begin{psmallmatrix}m&n\\u&v\end{psmallmatrix}\in\SLTwo$.  Set
\[
B=L^{-1}\TwoTwoMatrix{\frac{d^2\alpha}{2}}{\frac{d^2\beta}{2}}{\gamma}{\delta} = \TwoTwoMatrix{g}{\frac{d^2\beta}{2} v- \delta n}{0}{\frac{d^2}{2g}},
\]
so that $\frac{d^2}{2}A\tau = LB\tau$.  
We then use Proposition \ref{prop:MtransformationSL2}
and find the proof follows much the same as
the proof of Proposition \ref{PropDOddMdkMaassForm}.
\end{proof}

We next determine invariant orders at cusps of the holomorphic parts of the $\widetilde{M}_{d,k}$ functions.

\begin{proposition}\label{PropDOddMTildeInvariantOrder}
Suppose $\alpha$ and $\gamma$ are integers and $\gcd(\alpha,\gamma)=1$.
If $d$ is odd, then
the invariant order of the holomorphic part of $\widetilde{M}_{d,k}(a,b,c;2d^2\tau)$
at the cusp $\frac{\alpha}{\gamma}$ is at least
\begin{align*}
	\frac{g^2}{2d^2}
	\left(
		-\frac{(4ad^2u+2bm-cd^2m)^2}{8d^4c^2}
		+
		\widetilde{\nu}\left(
			\frac{2ad^2u+bm}{cd^2}+\frac{km}{d},
			\frac{m}{2}+\frac{km}{d}
		\right)
	\right)
,
\end{align*}
where $g=\gcd(2d^2,\gamma)$, $m=\frac{2d^2\alpha}{g}$, and $u=\frac{\gamma}{g}$.  

Similarly, if $d$ is even, then the invariant order of the holomorphic part of $\widetilde{M}_{d,k}(a,b,c;\frac{d^2}{2}A\tau)$ at the cusp $\frac{\alpha}{\gamma}$ is at least 
\[
\frac{2g^2}{d^2} \left( - \frac{(2d^2(a-c)u + (4b-cd^2)m)^2}{8c^2d^4} + \widetilde{\nu}\left(\tfrac{au}{c} 
+ \tfrac{2bm}{cd^2}+ \tfrac{2km}{d}, u + \tfrac{m}{2}+\tfrac{2km}{d} \right)  \right)
\]
when $d\equiv 2 \pmod 4$, 
and 
\[
\frac{2g^2}{d^2} \left( - \frac{(d^2(2a-c)u + (4b-cd^2)m)^2}{8c^2d^4} + \widetilde{\nu}\left(\tfrac{au}{c} 
+ \tfrac{2bm}{cd^2}+ \tfrac{2km}{d}, \tfrac{u}{2} + \tfrac{m}{2}+\tfrac{2km}{d} \right)  \right)
\]
when $d\equiv 0 \pmod 4$, where in both cases $g=\gcd(\tfrac{d^2\alpha}{2}, \gamma)$, $m=\tfrac{d^2\alpha}{2g}$, and $\frac{\gamma}{g}$.
\end{proposition}
\begin{proof}
In the case when $d$ is odd, as in Proposition \ref{PropDOddMdkMaassForm} we let
$A=\begin{psmallmatrix}\alpha&\beta\\\gamma&\delta\end{psmallmatrix}, 
L=\begin{psmallmatrix}m&n\\u&v\end{psmallmatrix}\in\SLTwo$
and
\begin{align*}
	B
	&=
		L^{-1}\TwoTwoMatrix{2d^2\alpha}{2d^2\beta}{\gamma}{\delta}
	=
		\TwoTwoMatrix{g}{2d^2\beta v- \delta n}{0}{2d^2/g}
.
\end{align*}
As in (\ref{EqDOddMdkCuspCondition}), we have
\begin{align*}
	(\gamma\tau+\delta)^{-\frac{1}{2}}\widetilde{M}_{d,k}(a,b,c;2d^2A\tau)	
	&=
		\varepsilon_1
		\exp\left( -2\pi iB\tau\left(
			\tfrac{(4ad^2u+2bm-cd^2m)^2}{8d^4c^2}	
		\right)\right)
		\\&\quad		
		\tmu\left(
			\tfrac{2av}{c}+\tfrac{bn}{cd^2}+\tfrac{kn}{d}
			+
			\left(\tfrac{2ad^2u+bm}{cd^2}+\tfrac{km}{d}\right)B\tau,
			\left(\tfrac{1}{2}+\tfrac{k}{d}\right)mB\tau 
			+ 
			\tfrac{n}{2}+\tfrac{kn}{d}; B\tau		
		\right)
,
\end{align*}
\sloppy where $\varepsilon_1$ is some nonzero constant. Noting
$\exp(2\pi iB\tau) = \varepsilon_2 q^{\frac{g^2}{2d^2}}$, where
$|\varepsilon_2|=1$, the result then follows from
Corollary \ref{PropTMuOrders}.  The cases when $d$ is even follow similarly. 
\end{proof}

In the case when $d$ is odd, we see the functions $P_{d,k}(a,b,c;\tau)$, which 
were defined in \eqref{def:Pdk}, appearing along side the 
$\widetilde{M}_{d,k}(a,b,c;\tau)$ functions in Proposition 
\ref{prop:OtoMockModular}.  Thus we also need to study their modular properties.

\begin{proposition}\label{PropositionPTransformation}
Suppose $d$ is odd and 
$A=\begin{pmatrix}\alpha & \beta \\ \gamma & \delta \end{pmatrix}
\in \Gamma_0(2d^2)\cap \G_{a,b,c,d} \cap \G_{a,b,c,d}' \cap \G_1\left(\frac{d}{\gcd(d,k)} \right)$.  Then
\begin{align*}
	P_{d,k}(a,b,c;A\tau)
	&=
		(-1)^{\beta+\frac{(\alpha-1)}{2}}
		i^{-\alpha\beta}
		\nu( \mA{2} )^{-3}		
		\sqrt{\gamma\tau+\delta}
		P_{d,k}(a,b,c;\tau)
.
\end{align*}
\end{proposition}
\begin{proof}
We only give the proof for the case when $d\nmid k$, and note that the proof
for the case when $d\mid k$ is simpler.
By Proposition \ref{prop:eta_transf}, on the subgroup $\Gamma_0(2d^2)$ we have that
\begin{align*}
	\frac{\eta(2A\tau)}{\eta(A\tau)^2}
	&=
		(-1)^{\beta + \frac{\alpha-1}{2}} 
		i^{-\alpha\beta} 
		\nu( \mA{2} )^{-3}
		(\gamma\tau+\delta)^{-\frac{1}{2}}
		\frac{\eta(2\tau)}{\eta(\tau)^2}
.
\end{align*}	
On the subgroup $\Gamma_0(2d^2)\cap \Gamma_1\left(\frac{d}{\gcd(d,k)}\right)$
we have $4dk\alpha \equiv 4dk \pmod{2d^2}$, and $2dk\alpha \equiv 2dk \pmod{2d^2}$, so that
\begin{align*}
	\frac{f_{2d^2,4dk}(A\tau)}{f_{2d^2,2dk}(A\tau)}
	&=
		(-1)^{ \Floor{\frac{2k\alpha}{d}} +\Floor{\frac{k\alpha}{2d}}+\Floor{\frac{2k}{d}}+\Floor{\frac{k}{d}}  }
		\frac{f_{2d^2,4dk}(\tau)}{f_{2d^2,2dk}(\tau)}
	=
		\frac{f_{2d^2,4dk}(\tau)}{f_{2d^2,2dk}(\tau)}
	.
\end{align*}
On $\Gamma_0(2d^2)\cap \G_{a,b,c,d} \cap \G_{a,b,c,d}' \cap \G_1\left(\frac{d}{\gcd(d,k)} \right)$ we have that
\begin{align*}
	\kModular{-\frac{b}{cd^2}}{-\frac{2a}{c}}{ \mA{2d^2} 2d^2\tau }
	&=
		(\gamma\tau+\delta)^{-1}
		\kModular{-\frac{b\alpha}{cd^2}-\frac{a\gamma}{cd^2} }{ -\frac{2b\beta}{c} - \frac{2a\delta}{c}  }{2d^2\tau}
	\\
	&=
		(-1)^{\frac{a\gamma}{cd^2} + \frac{b(\alpha-1)}{cd^2}}
		\z{c}{-\frac{ab(\alpha-1)}{cd^2}-\frac{a^2\gamma}{cd^2}+\frac{ab(\delta-1)}{cd^2}}
		\z{c^2d^2}{b^2\beta}
		(\gamma\tau+\delta)^{-1}
		\kModular{-\frac{b}{cd^2} }{ -\frac{2a}{c}  }{2d^2\tau}
	,
\end{align*}	
and additionally
\begin{align*}
	&\kModular{\frac{b}{cd^2}+\frac{k}{d}}{\frac{2a}{c}}{ \mA{2d^2} 2d^2\tau}^{-1}
	\kModular{-\frac{b}{cd^2}+\frac{k}{d}}{-\frac{2a}{c}}{ \mA{2d^2} 2d^2\tau}^{-1}
	\\
	&=
		(\gamma\tau+\delta)^2
		\kModular{ \frac{b\alpha}{cd^2} + \frac{k\alpha}{d} + \frac{a\gamma}{cd^2} }
			{\frac{2b\beta}{c}+2kd\beta+\frac{2a\delta}{c}}   
			{2d^2\tau}^{-1}
		\kModular{ -\frac{b\alpha}{cd^2} + \frac{k\alpha}{d} - \frac{a\gamma}{cd^2} }
			{-\frac{2b\beta}{c}+2kd\beta-\frac{2a\delta}{c}}   
			{2d^2\tau}^{-1}
	\\
	&=
		\z{c}{\frac{2ab(\alpha-1)}{cd^2} + \frac{2a^2\gamma}{cd^2} - \frac{2ab(\delta-1)}{cd^2}}
		\z{c^2d^2}{-2b^2\beta}
		(\gamma\tau+\delta)^2
		\kModular{ \frac{b}{cd^2} + \frac{k}{d}}{\frac{2a}{c}}{2d^2\tau}^{-1}
		\kModular{ -\frac{b}{cd^2} + \frac{k}{d}}{-\frac{2a}{c}}{2d^2\tau}^{-1}
.
\end{align*}

\sloppy
Thus
\begin{align*}
	P_{d,k}(a,b,c;A\tau)
	&=
		(-1)^{\beta+\frac{(\alpha-1)}{2} + \frac{a\gamma}{cd^2} + \frac{b(\alpha-1)}{cd^2}}
		i^{-\alpha\beta}
		\z{c}{\frac{ab(\alpha-1)}{cd^2} + \frac{a^2\gamma}{cd^2} - \frac{ab(\delta-1)}{cd^2}}
		\z{c^2d^2}{-b^2\beta}
		\nu( \mA{2} )^{-3}		
		\sqrt{\gamma\tau+\delta}
		P_{d,k}(a,b,c;A\tau)
	\\
	&=
		(-1)^{\beta+\frac{(\alpha-1)}{2}}
		i^{-\alpha\beta}
		\nu( \mA{2} )^{-3}		
		\sqrt{\gamma\tau+\delta}
		P_{d,k}(a,b,c;\tau)
	,
\end{align*}
where the last equality follows from the fact that $\Gamma_0(2d^2)\cap \G_{a,b,c,d} \cap \G_{a,b,c,d}' \cap \G_1\left(\frac{d}{\gcd(d,k)} \right)$ is contained in
$
\Gamma_0\left(\frac{2cd^2}{\gcd(a, 2cd^2)}\right)
\Gamma_1\left(\frac{2cd^2}{\gcd(b, 2cd^2)}  \right)
\Gamma_1\left(\frac{c^2d^2}{\gcd(ab, c^2d^2)}  \right)
\Gamma_0\left(\frac{c^2d^2}{\gcd(a^2, c^2d^2)}\right)
\Gamma^0\left(\frac{c^2d^2}{\gcd(b^2, c^2d^2)}  \right)
$.	
	
\fussy	
\end{proof}

\begin{proposition}\label{PropDOddPInvariantOrder}
Suppose $d$ is odd, and $\alpha, \gamma$ are integers with $\gcd(\alpha,\gamma)=1$. 
Then
the invariant order of $P_{d,k}(a,b,c;\tau)$
at the cusp $\frac{\alpha}{\gamma}$ is
\begin{align*}
	&
	-\frac{1}{12} 
	+ 
	\frac{\gcd(2,\gamma)^2}{48}
	+
	\frac{g^2}{4d^2}
	\bigg(
		\Fractional{ \tfrac{4dk\alpha}{g} }^2 - \Fractional{ \tfrac{4dk\alpha}{g} } 
		-
		\Fractional{ \tfrac{2dk\alpha}{g} }^2 + \Fractional{ \tfrac{2dk\alpha}{g} } 
		+
		\Fractional{-\tfrac{bm}{cd^2}-\tfrac{2au}{c}}^2 -\Fractional{-\tfrac{bm}{cd^2}-\tfrac{2au}{c}} 
		\\&\quad			
		-
		\Fractional{\tfrac{bm}{cd^2}+\tfrac{km}{d}+\tfrac{2au}{c}}^2 +\Fractional{\tfrac{bm}{cd^2}+\tfrac{km}{d}+\tfrac{2au}{c}} 
		-
		\Fractional{-\tfrac{bm}{cd^2}+\tfrac{km}{d}-\tfrac{2au}{c}}^2 +\Fractional{-\tfrac{bm}{cd^2}+\tfrac{km}{d}-\tfrac{2au}{c}} 
	\bigg)
,
\end{align*}
where $g=\gcd(2d^2,\gamma)$, $m=\frac{2d^2\alpha}{g}$, and $u=\frac{\gamma}{g}$.
\end{proposition}
\begin{proof}
As in the proofs of Propositions \ref{PropDOddMdkMaassForm} and \ref{prop:devenMaass}, we let
$A=\begin{psmallmatrix}\alpha&\beta\\\gamma&\delta\end{psmallmatrix}, 
L=\begin{psmallmatrix}m&n\\u&v\end{psmallmatrix}\in\SLTwo$
and
\begin{align*}
	B
	&=
		L^{-1}\TwoTwoMatrix{2d^2\alpha}{2d^2\beta}{\gamma}{\delta}
	=
		\TwoTwoMatrix{g}{2d^2\beta v- \delta n}{0}{2d^2/g}
.
\end{align*}
We note that
\begin{align*}
	(\gamma\tau+\delta)\kModular{a_1}{a_2}{2d^2 A\tau}
	&=
		(\gamma\tau+\delta)\kModular{a_1}{a_2}{LB\tau}
	=
		\frac{2d^2}{g}\kModular{a_1m+a_2u}{a_1n+a_2v}{B\tau}
	,
\end{align*}
so that the invariant order of $\kModular{a_1}{a_2}{2d^2\tau}$ at the cusp
$\frac{\alpha}{\gamma}$ is
\begin{align*}
	\frac{g^2}{2d^2}
	\left(
		\frac{1}{2}B_2\left(\Fractional{a_1m+a_2u}\right) - \frac{1}{12}
	\right)
	&=
	\frac{g^2}{4d^2}
	\left(
		\Fractional{a_1m+a_2u}^2 -\Fractional{a_1m+a_2u} 
	\right)
	.
\end{align*}
By Proposition \ref{PropBiagioliInvariantOrder}, the invariant order of $f_{N,\rho}(\tau)$ at 
$\frac{\alpha}{\gamma}$ is
\begin{align*}
	\frac{\gcd(N,\gamma)^2}{2N}\left(\Fractional{\frac{\alpha\rho}{\gcd(N,\gamma)}}-\frac{1}{2} \right)^2.
\end{align*}
Additionally it well known that the invariant order of $\frac{\eta(2\tau)}{\eta(\tau)^2}$
at the cusp $\frac{\alpha}{\gamma}$ is
\begin{align*}
	-\frac{1}{12} + \frac{\gcd(2,\gamma)^2}{48}
.
\end{align*}
We can now deduce the invariant order of $P_{d,k}(a,b,c;\tau)$ at the cusp 
$\frac{\alpha}{\gamma}$.

For $k=0$ we find that the invariant order of $P_{d,0}(a,b,c;\tau)$ at the cusp
$\frac{\alpha}{\gamma}$ is
\begin{align*}
	&
	-\frac{1}{12} 
	+ 
	\frac{\gcd(2,\gamma)^2}{48}
	-
	\frac{g^2}{4d^2}
	\left(
		\Fractional{\tfrac{bm}{cd^2}+\tfrac{2au}{c}}^2 -\Fractional{\tfrac{bm}{cd^2}+\tfrac{2au}{c}} 
	\right)
	.
\end{align*}

For $d\nmid k$ we find that the invariant order of $P_{d,k}(a,b,c;\tau)$ at the cusp
$\frac{\alpha}{\gamma}$ is
\begin{align*}
	&
	-\frac{1}{12} 
	+ 
	\frac{\gcd(2,\gamma)^2}{48}
	+
	\frac{g^2}{4d^2}
	\left(
		\Fractional{ \tfrac{4dk\alpha}{g} } - \frac{1}{2} 
	\right)^2
	-
	\frac{g^2}{4d^2}
	\left(
		\Fractional{ \tfrac{2dk\alpha}{g} } - \frac{1}{2} 
	\right)^2
	+
	\frac{g^2}{4d^2}
	\left(
		\Fractional{-\tfrac{bm}{cd^2}-\tfrac{2au}{c}}^2 -\Fractional{-\tfrac{bm}{cd^2}-\tfrac{2au}{c}} 
	\right)
	\\&\quad	
		-
		\frac{g^2}{4d^2}
		\left(
			\Fractional{\tfrac{bm}{cd^2}+\tfrac{km}{d}+\tfrac{2au}{c}}^2 -\Fractional{\tfrac{bm}{cd^2}+\tfrac{km}{d}+\tfrac{2au}{c}} 
		\right)
		-
		\frac{g^2}{4d^2}
		\left(
			\Fractional{-\tfrac{bm}{cd^2}+\tfrac{km}{d}-\tfrac{2au}{c}}^2 -\Fractional{-\tfrac{bm}{cd^2}+\tfrac{km}{d}-\tfrac{2au}{c}} 
		\right)
	\\
	&=
		-\frac{1}{12} 
		+ 
		\frac{\gcd(2,\gamma)^2}{48}
		+
		\frac{g^2}{4d^2}
		\bigg(
			\Fractional{ \tfrac{4dk\alpha}{g} }^2 - \Fractional{ \tfrac{4dk\alpha}{g} } 
			-
			\Fractional{ \tfrac{2dk\alpha}{g} }^2 + \Fractional{ \tfrac{2dk\alpha}{g} } 
			+
			\Fractional{-\tfrac{bm}{cd^2}-\tfrac{2au}{c}}^2 -\Fractional{-\tfrac{bm}{cd^2}-\tfrac{2au}{c}} 
			\\&\quad			
			-
			\Fractional{\tfrac{bm}{cd^2}+\tfrac{km}{d}+\tfrac{2au}{c}}^2 +\Fractional{\tfrac{bm}{cd^2}+\tfrac{km}{d}+\tfrac{2au}{c}} 
			-
			\Fractional{-\tfrac{bm}{cd^2}+\tfrac{km}{d}-\tfrac{2au}{c}}^2 +\Fractional{-\tfrac{bm}{cd^2}+\tfrac{km}{d}-\tfrac{2au}{c}} 
		\bigg)
.
\end{align*}
\end{proof}

\subsection{Transformations of $\widetilde{\mathcal{O}}$}

We are now able to prove transformation formulas for $\widetilde{\mathcal{O}}_{d}(a,b,c;\tau)$.

\begin{proposition}\label{prop:FullOTransf}
Suppose
\[
A=\begin{pmatrix}\alpha & \beta \\ \gamma & \delta \end{pmatrix}\in  
\begin{cases}
\G_0(2d^2) \cap \G_{a,b,c,d} \cap \G_{a,b,c,d}' \cap \G_1(d) & \text{ if } d \text { odd}, \\
\G_0(d^2/2) \cap \G_{a,b,c,d} \cap \G_{a,b,c,d}' \cap \G_1(2d) & \text{ if } d\equiv 2 \pmod 4,\\
\G_0(4d^2) \cap \G_{a,b,c,d} \cap \G_{a,b,c,d}' \cap \G_1(2d)& \text{ if } d\equiv 0 \pmod 4.
\end{cases}
\]
Then
\[
\widetilde{\mathcal{O}}_{d}(a,b,c;A\tau) = 
(-1)^{\beta + \frac{\alpha-1}{2}} i^{-\alpha \beta} \nu(\mA{2})^{-3} \sqrt{\gamma\tau+\delta} \widetilde{\mathcal{O}}_d(a,b,c;\tau) 
.
\]
Moreover, for such $A$,
we have that 
\[
\frac{\eta(A\tau)^2}{\eta(2A\tau)} \widetilde{\mathcal{O}}_{d}(a,b,c;A\tau) 
= (\gamma\tau+\delta)\frac{\eta(\tau)^2}{\eta(2\tau)}\widetilde{\mathcal{O}}_{d}(a,b,c;\tau).
\]		
\end{proposition}
\begin{proof}

In order to understand how $\widetilde{\mathcal{O}}_d(a,b,c;\tau)$ transforms when $d$ odd, we see from \eqref{eqn:OviaM} that we need to consider how
\[
\frac{\eta(2\tau) f_{2d^2, d^2+2dk}(\tau)}{\eta(\tau)^2}
\]
transforms.  Proposition \ref{prop:eta_transf} gives the transformation for $\frac{\eta(2\tau)}{\eta(\tau)^2}$, so it remains to consider $ f_{2d^2, d^2+2dk}(\tau)$.
Since $A \in\Gamma_0(2d^2)\cap\Gamma_1(d)$, after simplifications we deduce from (\ref{EqBiagioliTransform}) that
\begin{align*}
	&f_{2d^2,d^2+2dk}(A\tau)
	\\
	&=
		(-1)^{ (d^2+2dk)\beta + \Floor{\frac{(d^2+2dk)\alpha}{2d^2}} + \Floor{\frac{(d^2+2dk)}{2d^2}} }
		\exp\left( \frac{ \pi i\alpha\beta(d^2+2dk)^2}{2d^2}  \right)
		\nu( \mA{2d^2} )^3
		\sqrt{\gamma\tau+\delta}
		f_{2d^2,d^2+2dk}(\tau)	
	\\
	&=
		(-1)^{ \beta + \frac{(\alpha-1)}{2} }
		i^{ \alpha\beta }
		\nu( \mA{2}  )^3
		\sqrt{\gamma\tau+\delta}
		f_{2d^2,d^2+2dk}(\tau)	
	.
\end{align*}
From this, together with Corollary \ref{cor:Mtransf1}, Proposition \ref{PropositionPTransformation},
Proposition \ref{prop:eta_transf}, and \eqref{eqn:OviaM}, we obtain the stated 
transformation for $\widetilde{\mathcal{O}}_d(a,b,c;\tau)$
after cancellations.


In order to understand how $\widetilde{\mathcal{O}}_d(a,b,c;\tau)$ transforms when $d\equiv 2 \pmod4$, we see from \eqref{eqn:OviaM} that we need to consider how 
\[
\frac{\eta(2\tau)f_{\frac{d^2}{2}, \frac{d^2}{4}+dk}(\tau)}{\eta(\tau)^2}  
\]
transforms.  Again it remains only to consider $f_{\frac{d^2}{2}, \frac{d^2}{4}+dk}(\tau)$.  
Since $A\in\Gamma_0(\frac{d^2}{2})\cap\Gamma_1(2d)$, by
\eqref{EqBiagioliTransform}, after cancellations  we find  that
\begin{align*}
f_{\frac{d^2}{2}, \frac{d^2}{4} + dk}(A\tau) 
&= 
	(-1)^{( d^2/4 +dk)\beta + \lfloor \frac{\alpha( d+4k)}{2d} \rfloor + \lfloor \frac{(d+4k)}{2d} \rfloor }   
	\exp\left( \frac{2\pi i \alpha\beta \left( d^2/4 +dk \right)^2}{d^2} \right)  
	\nu( \mA{\frac{d^2}{2}} )^3  \sqrt{\gamma\tau+\delta} f_{\frac{d^2}{2}, \frac{d^2}{4} + dk}(\tau)
\\
&= 
	(-1)^{\beta +\frac{\alpha-1}{2} } i^{\alpha \beta} \nu( \mA{2} )^3  
	\sqrt{\gamma\tau+\delta} f_{\frac{d^2}{2}, \frac{d^2}{4} + dk}(\tau).
\end{align*}
Using this together with Corollary \ref{cor:Mtransf1},
Proposition \ref{prop:eta_transf},
and \eqref{eqn:OviaM}, we obtain the transformation for 
$\widetilde{\mathcal{O}}_d(a,b,c;\tau)$
after cancellations.

In order to understand how $\widetilde{\mathcal{O}}_d(a,b,c;\tau)$ transforms when $d\equiv 0 \pmod4$, we see from \eqref{eqn:OviaM} that we need to consider how
\[
\frac{ \eta(2\tau) \eta(\frac{d^2\tau}{2})^2 f_{d^2, \frac{d^2}{2}+2dk}(\tau)} {\eta(\tau)^2 \eta(d^2\tau) f_{\frac{d^2}{2}, \frac{d^2}{4}+dk}(\tau)} 
\]
transforms.  First, we observe that  $\frac{\eta(\frac{d^2\tau}{2})^2}{\eta(d^2\tau)} = f_{d^2,\frac{d^2}{2}}(\tau)$. Using the fact that $A \in \G_0(4d^2)\cap\Gamma_1(2d)$, we have by \eqref{EqBiagioliTransform} that 
\begin{equation}\label{eq:eta2_transf}
 f_{d^2,\frac{d^2}{2}}(A\tau) = \nu(\mA{d^2})^{3} \sqrt{\gamma\tau+\delta} f_{d^2,\frac{d^2}{2}}(\tau).
\end{equation}
Furthermore,
\begin{align}\label{prop:f_transf4}
f_{d^2, \frac{d^2}{2} + 2dk}(A\tau) &= (-1)^{ \frac{(\alpha-1)}{2}} \nu( \mA{d^2} )^3 \sqrt{\gamma\tau+\delta}  
	f_{d^2, \frac{d^2}{2} + 2dk}(\tau), \nonumber \\
f_{\frac{d^2}{2}, \frac{d^2}{4} + dk}(A\tau) &= (-1)^{ \frac{(\alpha-1)}{2}} \nu( \mA{\frac{d^2}{2}} )^3 \sqrt{\gamma\tau+\delta} 
	f_{\frac{d^2}{2}, \frac{d^2}{4} +2dk}(\tau).
\end{align}
With Proposition \ref{prop:eta_transf}, \eqref{prop:f_transf4}, and \eqref{eq:eta2_transf} we find that 
\[
\frac{ \eta(2A\tau) \eta(\frac{d^2A\tau}{2})^2 f_{d^2, \frac{d^2}{2}+2dk}(A\tau)} {\eta(A\tau)^2 \eta(d^2A\tau) f_{\frac{d^2}{2}, \frac{d^2}{4}+dk}(A\tau)}  = (-1)^\beta i^{-\alpha\beta} \left( \frac{\nu( \mA{d^2} )^6}{\nu( \mA{2} )^3 \nu(\mA{\frac{d^2}{2}})^3} \right) \frac{ \eta(2\tau) \eta(\frac{d^2\tau}{2})^2 f_{d^2, \frac{d^2}{2}+2dk}(\tau)} {\eta(\tau)^2 \eta(d^2\tau) f_{\frac{d^2}{2}, \frac{d^2}{4}+dk}(\tau)}. 
\]

By the discussion following \eqref{EqEtaMultipler}, the fact that $\frac{d}{2}$ is even yields that
$\nu( \mA{\frac{d^2}{2}} )^3 = i^{-\alpha\beta}\nu(\mA{2})^3$ on $\Gamma_0(4d^2)$.
Furthermore, by Theorem 1.64 of \cite{Ono1},
the eta-quotient $\frac{\eta(d^2\tau)^6}{\eta(8\tau)^6}$ is a weight $0$ modular form
on $\G_0(4d^2)$,  and so $\nu(\mA{d^2})^6 = \nu(\mA{8})^6$.  This gives that
\[
\frac{\nu( \mA{d^2} )^6}{\nu( \mA{2} )^3 \nu(\mA{\frac{d^2}{2}})^3} = (-1)^\beta i^{\alpha\beta},
\]
which yields that
\begin{equation}\label{cor:mult4}
\frac{ \eta(2A\tau) \eta(\frac{d^2A\tau}{2})^2 f_{d^2, \frac{d^2}{2}+2dk}(A\tau)} {\eta(A\tau)^2 \eta(d^2A\tau) f_{\frac{d^2}{2}, \frac{d^2}{4}+dk}(A\tau)}  =  \frac{ \eta(2\tau) \eta(\frac{d^2\tau}{2})^2 f_{d^2, \frac{d^2}{2}+2dk}(\tau)} {\eta(\tau)^2 \eta(d^2\tau) f_{\frac{d^2}{2}, \frac{d^2}{4}+dk}(\tau)}. 
\end{equation}

Using \eqref{cor:mult4} together with 
Corollary \ref{cor:Mtransf1}, 
Proposition \ref{prop:eta_transf},
and \eqref{eqn:OviaM}, 
we obtain the stated transformation for $\widetilde{\mathcal{O}}_d(a,b,c;\tau)$
after cancellations and noting that $A\in\Gamma_1(4)$.

We see that the multiplier for 
$\widetilde{\mathcal{O}}_d(a,b,c;\tau)$ is identical to that of
$\frac{\eta(2\tau)}{\eta(\tau)^2}$ as stated in Proposition \ref{prop:eta_transf}
on the given congruence subgroups, and as such
$\frac{\eta(\tau)^2}{\eta(2\tau)}\widetilde{\mathcal{O}}_d(a,b,c;\tau)$ 
has trivial multiplier.
\end{proof}

\section{Dissection terms and their modularity}

We begin by determining an alternate form for the non-holomorphic parts of the 
$\widetilde{M}_{d,k}(a,b,c; \tau)$ for use in dissection formulas.  
We have the following $R$ functions appearing in the definition of 
$\widetilde{\mathcal{O}}_d(a,b,c;\tau)$ in \eqref{DefinitionOfTildeO}:
\begin{align*}
&R\left(  \left( \tfrac{2b}{c} - d^2 \right) \tau + \tfrac{2a}{c}; 2d^2\tau \right) \mbox{ when } d \mbox{ odd},\\
&R\left( \left( \tfrac{b}{c} - \tfrac{d^2}{4}\right)\tau + \left(\tfrac{a}{c} -1 \right); \tfrac{d^2\tau}{2} \right) \mbox{ when } d\equiv 2 \!\! \pmod 4,\\
&R\left( \left( \tfrac{b}{c} - \tfrac{d^2}{4}\right)\tau + \left(\tfrac{a}{c} - \tfrac12 \right); \tfrac{d^2\tau}{2}\right) \mbox{ when } d\equiv 0 \!\! \pmod 4.
\end{align*} 
For convenience, we make the following definitions to use in the dissection of these $R$ functions.  Let $a,c,r,s$ be integers such that 
$r$ is nonnegative and $s$ is positive.  We then define
\[
\epsilon_{d,r,s}:= 
\begin{cases}
(-1)^r					& \text{ if } s \text { odd}, d \not\equiv 0 \!\! \pmod 4, \\
(-1)^{\frac{s-1}{2}} 	&\text{ if } s \text { odd}, d \equiv 0 \!\! \pmod 4, \\
-i(-1)^{r+d} 			&\text{ if } s \text { even}, d \not\equiv 0 \!\! \pmod 4, \\
-(-1)^{\frac{s}{2}} 	&\text{ if } s \text { even}, d \equiv 0 \!\! \pmod 4,
\end{cases}
\]
and 
\[
v_{a,c,d,s}:= 
\begin{cases}
\tfrac{2as}{c} 									&\text{ if } s \text { odd}, d \text{ odd}, \\
\tfrac{as}{c} - s 								&\text{ if } s \text { odd}, d \equiv 2 \!\! \pmod 4, \\
\tfrac{as}{c} - \tfrac{s}{2} 					&\text{ if } s \text { odd}, d \equiv 0 \!\! \pmod 4, \\
\tfrac{2as}{c} + \tfrac12 						&\text{ if } s \text { even}, d \text{ odd}, \\
\tfrac{as}{c} - s + \tfrac12 					&\text{ if } s \text { even}, d \equiv 2 \!\! \pmod 4, \\
\tfrac{as}{c} - \tfrac{s}{2} + \tfrac12 		&\text{ if } s \text { even}, d \equiv 0 \!\! \pmod 4.
\end{cases}
\]

\begin{proposition}\label{prop:Rdissection}
Let $s$ be a positive integer. If $d$ is odd, then
\begin{multline*}
R\left( \left(\tfrac{2b}{c}-d^2\right)\tau + \tfrac{2a}{c}; 2d^2\tau \right) = \\
\sum_{r=0}^{s-1} \epsilon_{d,r,s} \z{c}{-a(2r+1-s)}q^{-\frac{d^2(2r+1-s)^2}{4} - \frac{(2r+1-s)(2b-cd^2)}{2c} } R\left( ( 2d^2rs + \tfrac{2bs}{c} - d^2s^2)\tau + v_{a,c,d,s}; 2d^2s^2\tau \right),
\end{multline*}
if $d \equiv 2 \pmod 4$, then
\begin{multline*}
R\left( \left( \tfrac{b}{c} - \tfrac{d^2}{4}\right)\tau + \left(\tfrac{a}{c} -1 \right); \tfrac{d^2\tau}{2} \right) = \\
\sum_{r=0}^{s-1} \epsilon_{d,r,s} \zeta_{2c}^{-a(2r+1-s)}q^{-\frac{d^2(2r+1-s)^2}{16} -\frac{(2r+1-s)(4b-cd^2)}{8c}} \cdot R\left( \left( \tfrac{d^2rs}{2} + \tfrac{bs}{c} - \tfrac{d^2s^2}{4}\right)\tau + v_{a,c,d,s}; \tfrac{d^2s^2\tau}{2} \right),
\end{multline*}
and if $d \equiv 0 \pmod 4$, then
\begin{multline*}
R\left( \left( \tfrac{b}{c} - \tfrac{d^2}{4}\right)\tau + \left(\tfrac{a}{c} -\tfrac12 \right); \tfrac{d^2\tau}{2} \right) = \\
\sum_{r=0}^{s-1} \epsilon_{d,r,s} \zeta_{2c}^{-a(2r+1-s)}q^{-\frac{d^2(2r+1-s)^2}{16} - \frac{(2r+1-s)(4b-cd^2)}{8c}} \cdot R\left( \left( \tfrac{d^2rs}{2} + \tfrac{bs}{c} - \tfrac{d^2s^2}{4}\right)\tau + v_{a,c,d,s}; \tfrac{d^2s^2\tau}{2} \right).
\end{multline*}
\end{proposition}
\begin{proof}
We only give the proof for $d$ odd, as all three cases follow from only elementary rearrangements. We have
\begin{align*}
	&R\left( \left(\tfrac{b}{cd^2}-\tfrac{1}{2}\right)\tau + \tfrac{2a}{c}; \tau \right)
	\\
	&=
		\sum_{n=-\infty}^\infty
		\left(
			\mbox{sgn}(n+\tfrac{1}{2}) - E\left( (n+\tfrac{b}{cd^2})\sqrt{2\mbox{Im}(\tau)} \right)
		\right)
		\\&\quad\times			
		(-1)^n
		\exp\left(
			-\pi i( n+\tfrac{1}{2} )^2\tau 
			-2\pi i( n+\tfrac{1}{2} )\left( (\tfrac{b}{cd^2}-\tfrac{1}{2})\tau + \tfrac{2a}{c}   \right) 			
		\right)		
	\\
	&=
		\sum_{r=0}^{s-1}
		\sum_{n=-\infty}^\infty
		\left(
			\mbox{sgn}(sn+r+\tfrac{1}{2}) - E\left( (sn+r+\tfrac{b}{cd^2})\sqrt{2\mbox{Im}(\tau)} \right)
		\right)
		\\&\quad\times			
		(-1)^{sn+r}
		\exp\left(
			-\pi i( sn+r+\tfrac{1}{2} )^2\tau 
			-2\pi i( sn+r+\tfrac{1}{2} )\left( (\tfrac{b}{cd^2}-\tfrac{1}{2})\tau + \tfrac{2a}{c}   \right) 			
		\right)		
	\\
	&=
		\sum_{r=0}^{s-1}
		\sum_{n=-\infty}^\infty
		\left(
			\mbox{sgn}(n+\tfrac{r}{s}+\tfrac{1}{2s}) 
			- 
			E\left( (n+\tfrac{1}{2}+(\tfrac{r}{s}+\tfrac{b}{cd^2s}-\tfrac{1}{2}))\sqrt{2\mbox{Im}(s^2\tau)} \right)
		\right)
		\\&\quad\times			
		(-1)^{sn+r}
		\exp\left(
			-\pi i( n+\tfrac{1}{2})^2 s^2\tau 
			-2\pi i( n+\tfrac{1}{2} )
			\left( (\tfrac{b}{cd^2}-\tfrac{1}{2})s\tau + \tfrac{2as}{c} 
				 +(\tfrac{r}{s}+\tfrac{1}{2s}-\tfrac{1}{2})s^2\tau    
			\right) 			
		\right)		
		\\&\quad\times
		\exp\left(
			-\pi i(\tfrac{r}{s}+\tfrac{1}{2s}-\tfrac{1}{2} )^2 s^2\tau 
			-2\pi i(\tfrac{r}{s}+\tfrac{1}{2s}-\tfrac{1}{2} ) 
			\left( (\tfrac{b}{cd^2}-\tfrac{1}{2})s\tau + \tfrac{2as}{c}   \right) 			
		\right)		
	\\
	&=
		\sum_{r=0}^{s-1}
		(-1)^{r}
		\z{c}{-a(2r-s+1)}
		q^{-\frac{(2r-s+1)^2}{8} - \frac{(2r-s+1)(2b-cd^2)}{4cd^2}  }
		\\&\quad\times
		\sum_{n=-\infty}^\infty
		\left(
			\mbox{sgn}(n+\tfrac{1}{2}) 
			- 
			E\left( (n+\tfrac{1}{2}+(\tfrac{r}{s}+\tfrac{b}{cd^2s}-\tfrac{1}{2}))\sqrt{2\mbox{Im}(s^2\tau)} \right)
		\right)
		\\&\quad\times			
		(-1)^{sn}
		\exp\left(
			-\pi i( n+\tfrac{1}{2})^2 s^2\tau 
			-2\pi i( n+\tfrac{1}{2} )
			\left( (\tfrac{r}{s} + \tfrac{b}{cd^2s} - \tfrac{1}{2})s^2\tau + \tfrac{2as}{c} 
			\right) 			
		\right)		
	,
\end{align*}
where in the last equality we have used the fact that
$\mbox{sgn}(n+\tfrac{1}{2}) = \mbox{sgn}(n+\frac{r}{s}+\frac{1}{2s})$ which follows from
$0 < \frac{r}{s}+\frac{1}{2s} < 1$. When $s$ is odd, the above is exactly
\begin{align*}
	R\left( \left(\tfrac{b}{cd^2}-\tfrac{1}{2}\right)\tau + \tfrac{2a}{c}; \tau \right)
	&=
		\sum_{r=0}^{s-1}
		(-1)^{r}
		\z{c}{-a(2r-s+1)}
		q^{-\frac{(2r-s+1)^2}{8} - \frac{(2r-s+1)(2b-cd^2)}{4cd^2}  }
		R\left(
			(\tfrac{r}{s} + \tfrac{b}{cd^2s} - \tfrac{1}{2})s^2\tau + \tfrac{2as}{c}
			;
			s^2\tau
		\right)
	.
\end{align*}
However, when $s$ is even we have $(-1)^{sn}=1$, so we instead find that
\begin{align*}
	&R\left( \left(\tfrac{b}{cd^2}-\tfrac{1}{2}\right)\tau + \tfrac{2a}{c}; \tau \right)
	\\
	&=
		i\sum_{r=0}^{s-1}
		(-1)^{r}
		\z{c}{-a(2r-s+1)}
		q^{-\frac{(2r-s+1)^2}{8} - \frac{(2r-s+1)(2b-cd^2)}{4cd^2}  }
		R\left(
			(\tfrac{r}{s} + \tfrac{b}{cd^2s} - \tfrac{1}{2})s^2\tau + \tfrac{2as}{c} +\tfrac{1}{2}
			;
			s^2\tau
		\right)
	.
\end{align*}
\end{proof}

\subsection{Definition of $\widetilde{S}(r,c; \tau)$ and modular properties}

We are ultimately interested in when $b=0$ and $s=c$,
so that we can work with the $c$-dissection of $\mathcal{O}_d(\z{c}{a};q)$.  
To handle the functions appearing in the dissections, we let
\begin{align}
	S(r,c;\tau)
	&=	
		\frac{2(-1)^{c+1} q^{\frac{2cr-r^2}{2c^2}} }
			{\aqprod{q^{\frac{1}{2}}, q^{\frac{1}{2}}, q }{q}{\infty}}
		\sum_{n=-\infty}^\infty
		\frac{(-1)^{n}q^{\frac{n(n+2)}{2} } }
			{1-(-1)^{c+1}q^{n+\frac{r}{c}}}						
	\label{EqDefinitionOfS}
	=
	\left\{
	\begin{array}{ll}
		\displaystyle
		-2iq^{-\frac{(c-2r)^2}{8c^2}} 
		\mu\left( \tfrac{r\tau}{c}, \tfrac{\tau}{2}; \tau \right)
		&
		\mbox{ if } c \mbox{ is odd},		
		\\
		\displaystyle
		2q^{-\frac{(c-2r)^2}{8c^2}} 
		\mu\left( \tfrac{r\tau}{c}+\tfrac{1}{2}, \tfrac{\tau}{2}; \tau \right)
		&
		\mbox{ if } c \mbox{ is even},		
	\end{array}
	\right.
	\\
	\label{EqDefinitionOfSTilde}
	\widetilde{S}(r,c;\tau)
	&=
	\left\{
	\begin{array}{ll}
		\displaystyle
		-2iq^{-\frac{(c-2r)^2}{8c^2}} 
		\tmu\left( \tfrac{r\tau}{c}, \tfrac{\tau}{2}; \tau \right)
		&
		\mbox{ if } c \mbox{ is odd},		
		\\
		\displaystyle
		2q^{-\frac{(c-2r)^2}{8c^2}} 
		\tmu\left( \tfrac{r\tau}{c}+\tfrac{1}{2}, \tfrac{\tau}{2}; \tau \right)
		&
		\mbox{ if } c \mbox{ is even}.		
	\end{array}
	\right.
\end{align}
We note that for $c$ odd, we cannot take $r$ with $c\mid r$. However, we will 
see this case does not occur in our calculations and identities.

We now determine a transformation for $\widetilde{S}(r,c; \tau)$, under the action of $\SLTwo$, in terms of $\tmu(u,v;\tau)$.

\begin{proposition}\label{PropositionSTransformationInitial}
Suppose $c$ and $r$ are integers and 
$A=\begin{pmatrix}\alpha & \beta \\ \gamma & \delta \end{pmatrix}\in\SLTwo$.
If $c$ is odd and $c\nmid r$, then
\begin{align*}
	\widetilde{S}(r,c;A\tau)
	&=
		-2i\exp\left(-\frac{\pi i\alpha\beta(c-2r)^2}{4c^2}\right) \nu(A)^{-3}
		\sqrt{\gamma\tau+\delta}
		q^{-\frac{\alpha^2(c-2r)^2}{8c^2}}
		\tmu\left( \tfrac{r(\alpha\tau+\beta)}{c}, 
			\tfrac{(\alpha\tau+\beta)}{2}; \tau \right)
	.
\end{align*}
If $c$ is even, then
\begin{align*}
	\widetilde{S}(r,c;A\tau)
	&=
		2\exp\left(-\pi i\left(
			\frac{\alpha\beta(c-2r)^2}{4c^2}
			+
			\frac{\beta\gamma(2r-c)}{2c}
			+
			\frac{\gamma\delta}{4}
		\right)	\right) 
		\nu(A)^{-3}
		\sqrt{\gamma\tau+\delta}
		\\&\quad\times
		q^{-\frac{\alpha^2(c-2r)^2}{8c^2} - \frac{\alpha\gamma(2r-c)}{4c}-\frac{\gamma^2}{8} }
		\tmu\left( \tfrac{r(\alpha\tau+\beta)}{c} + \tfrac{(\gamma\tau+\delta)}{2}, 
			\tfrac{(\alpha\tau+\beta)}{2} 
			; \tau \right)
	.
\end{align*}
\end{proposition}
\begin{proof}
Let $c$ be odd.  By (\ref{EqZTheorem1.11P2}) we have that
\begin{align*}
	\widetilde{S}(r,c;A\tau)
	&=
		-2i\exp\left(-\frac{\pi iA\tau(c-2r)^2}{4c^2}\right)
		\tmu\left( \tfrac{rA\tau}{c}, 
			\tfrac{A\tau}{2}; A\tau \right)
	\\
	&=	
		-2i\exp\left(-\frac{\pi iA\tau(c-2r)^2}{4c^2}\right)
		\nu(A)^{-3}
		\exp\left(-\frac{\pi i\gamma}{(\gamma\tau+\delta)}\left(
			\frac{r(\alpha\tau+\beta)}{c}-\frac{(\alpha\tau+\beta)}{2}			
		\right)^2\right)
		\\&\quad\times
		\sqrt{\gamma\tau+\delta}
		\tmu\left( \tfrac{r(\alpha\tau+\beta)}{c}, 
			\tfrac{(\alpha\tau+\beta)}{2}; \tau \right)
	\\
	&=	
		-2i\exp\left(-\frac{\pi i\alpha\beta(c-2r)^2}{4c^2}\right) \nu(A)^{-3}
		\sqrt{\gamma\tau+\delta}
		q^{-\frac{\alpha^2(c-2r)^2}{8c^2}}
		\tmu\left( \tfrac{r(\alpha\tau+\beta)}{c}, 
			\tfrac{(\alpha\tau+\beta)}{2}; \tau \right)
.
\end{align*}	
The proof for $c$ even follows in a similar fashion.
\end{proof}

Using Proposition \ref{PropositionSTransformationInitial} we deduce the following transformations for $\widetilde{S}(r,c;mc^2\tau)$ under the action of a certain subgroup of $\SLTwo$.

\begin{proposition}\label{PropositionSTransformationFinal}
Suppose $c$ and $m$ are integers with $m$ even. For
$A=\begin{pmatrix}\alpha&\beta\\\gamma&\delta\end{pmatrix}\in
\Gamma_0(m\gcd(c,2)c^2)\cap\Gamma_1(c)$,
\begin{align*}
\widetilde{S}(r,c;mc^2A\tau)
&=
\begin{cases}
	(-1)^{\frac{m\beta}{2}+\frac{(\alpha-1)}{2}} \exp\left(-\frac{\pi im\alpha\beta}{4}\right)\nu( \mA{mc^2})^{-3}
	\sqrt{\gamma\tau+\delta} \widetilde{S}(r,c;mc^2\tau)
	&\mbox{if $c$ is odd}
	,\\
	(-1)^{\frac{(\alpha-1)}{2}} \exp\left(-\frac{\pi i\gamma\delta}{4mc^2}\right)\nu( \mA{mc^2})^{-3}
	\sqrt{\gamma\tau+\delta} \widetilde{S}(r,c;mc^2\tau)
	&\mbox{if $c$ is even}
	.
\end{cases}
\end{align*}
\end{proposition}
\begin{proof}
Let $c$ be odd.  Since $\mA{mc^2}\in\SLTwo$, we may apply Proposition \ref{PropositionSTransformationInitial}
along with (\ref{EqZTheorem1.11P1})
to obtain
\begin{align*}
&\widetilde{S}(r,c;mc^2A\tau)
=
	\widetilde{S}(r,c;\mA{mc^2}(mc^2\tau))
\\
&=
	-2i\exp\left(-\frac{\pi im\alpha\beta(c-2r)^2}{4}\right)
	\nu(\mA{mc^2})^{-3} \sqrt{\gamma\tau+\delta}
	q^{-\frac{\alpha^2m(c-2r)^2}{8}}
	\\&\quad\times
	\tmu\left(
		\frac{r(\alpha mc^2\tau+mc^2\beta)}{c},
		\frac{\alpha mc^2\tau+mc^2\beta}{2}; mc^2\tau
	\right)
\\
&=
	-2i(-1)^{\frac{m\beta}{2} + \frac{r(\alpha-1)}{c}+\frac{(\alpha-1)}{2}}
	\exp\left(-\frac{\pi im\alpha\beta}{4}\right)
	\nu(\mA{mc^2})^{-3} \sqrt{\gamma\tau+\delta}
	q^{-\frac{\alpha^2m(c-2r)^2}{8}}
	\\&\quad\times
	\exp\left( \pi imc^2\tau\left( \frac{r(\alpha-1)}{c}-\frac{(\alpha-1)}{2} \right)^2  \right)	
	\exp\left( 2\pi i\left( \frac{r(\alpha-1)}{c}-\frac{(\alpha-1)}{2} \right)\left(\frac{rmc^2\tau}{c}-\frac{mc^2\tau}{2}\right)  \right)		
	\\&\quad\times
	\tmu\left( \frac{rmc^2\tau}{c}, \frac{mc^2\tau}{2} ; mc^2\tau \right)
\\
&=
	(-1)^{\frac{m\beta}{2} +\frac{(\alpha-1)}{2}}
	\exp\left(-\frac{\pi im\alpha\beta}{4}\right)
	\nu(\mA{mc^2})^{-3} \sqrt{\gamma\tau+\delta}
	\widetilde{S}(r,c;mc^2\tau)
.
\end{align*}
Again, the proof for $c$ even is similar.
\end{proof}

We now establish that $\widetilde{S}(r,c;\tau)$ is a harmonic Maass form and determine the order at cusps of the homomorphic part.

\begin{proposition}\label{PropSMassFormAndInvariantOrder}
Suppose $c$ and $m$ are integers with $m$ even.
Then $\widetilde{S}(r,c;mc^2\tau)$ is annihilated by $\Delta_{\frac{1}{2}}$
and has at most linear exponential growth at the cusps. In particular,
the invariant order of the holomorphic part of $\widetilde{S}(r,c;mc^2\tau)$ 
at the cusp 
$\frac{\alpha}{\gamma}$ is at least
\[
\frac{g^2}{mc^2}\left( -\frac{\ell^2(c-2r)^2}{8c^2}+ \widetilde{\nu}\left(\frac{r\ell}{c},\frac{\ell}{2}\right) \right)
\]
if $c$ is odd, and 
\[
\frac{g^2}{mc^2}\left( -\frac{\ell^2(c-2r)^2}{8c^2} -\frac{\ell u(2r-c)}{4c}-\frac{u^2}{8}
+\widetilde{\nu}\left(\frac{r\ell}{c}+\frac{u}{2},\frac{\ell}{2}\right) \right)
\]
if $c$ is even, where $g=\gcd(mc^2,\gamma)$, $\ell=\frac{mc^2\alpha}{g}$, and $u=\frac{\gamma}{g}$.
\end{proposition}
\begin{proof}
\sloppy First we verify that $\widetilde{S}(r,c;mc^2\tau)$ is annihilated by 
$\Delta_{\frac{1}{2}}$. For this we need only verify that
$\sqrt{\mbox{Im}(\tau)}\frac{\partial}{\partial\overline{\tau}}
q^{-\frac{m(c-2r)^2}{8}}R\left( rmc\tau+\star-\frac{mc^2\tau}{2} \right)$,
where $\star=0$ when $c$ is odd and $\star=\frac{1}{2}$ is even, is holomorphic
in $\overline{\tau}$. However, both cases follow immediately from Lemma 1.8 of
\cite{Zwegers}.

\fussy Next we check the growth condition at the cusps.  Suppose that $\alpha$ and $\gamma$ are integers with $\gcd(\alpha,\gamma)=1$, then take $A=\begin{psmallmatrix}\alpha&\beta\\\gamma&\delta\end{psmallmatrix}\in \SLTwo$.   To obtain the growth of $\widetilde{S}(r,c;mc^2\tau)$ at $\tau=\frac{\alpha}{\gamma}$, we study the growth of $\widetilde{S}(r,c;mc^2A\tau)$ at infinity.  We set $g=\gcd(mc^2,\gamma)$, $\ell=\frac{mc^2\alpha}{g}$,
and $u=\frac{\gamma}{g}$. Since $\gcd(\alpha,\gamma)=\gcd(\ell,u)=1$, we can
take
$L=\begin{psmallmatrix}\ell&n\\ u&v\end{psmallmatrix}\in\SLTwo$
and set
\begin{align*}
	B
	&=
		L^{-1}\TwoTwoMatrix{mc^2\alpha}{mc^2\beta}{\gamma}{\delta}
	=
		\TwoTwoMatrix{g}{mc^2\beta v- \delta n}{0}{mc^2/g}
.
\end{align*}
As the transformation formulas for $\widetilde{S}(r,c;\tau)$ depend on the 
parity of $c$, we now consider cases. 

When $c$ is odd, by Proposition 
\ref{PropositionSTransformationInitial} we have that
\begin{align*}
(\gamma\tau+\delta)^{-\frac{1}{2}}\widetilde{S}(r,c;mc^2A\tau)
&=
(\gamma\tau+\delta)^{-\frac{1}{2}}\widetilde{S}(r,c;LB\tau)
\\
&=
\epsilon_1 \exp\left(-2\pi iB\tau\frac{\ell^2(c-2r)^2}{8c^2} \right)
\tmu\left(\frac{r(\ell B\tau+n)}{c}, \frac{(\ell B\tau+n)}{2} ; B\tau \right),
\end{align*}
where $\epsilon_1$ is some nonzero constant. However, we see that 
$\exp(2\pi iB\tau)=\epsilon_2q^{\frac{g^2}{mc^2}}$, where $|\epsilon_2|=1$.
From the definition of $\tmu$ we deduce that $\widetilde{S}(r,c;mc^2\tau)$
has at most linear exponential growth at $\frac{\alpha}{\gamma}$
and by Corollary \ref{PropTMuOrders} we find that the invariant order
of the holomorphic part is at least
\begin{align*}
\frac{g^2}{mc^2}\left( -\frac{\ell^2(c-2r)^2}{8c^2} + \widetilde{\nu}\left(\frac{r\ell}{c},\frac{\ell}{2}\right) \right).
\end{align*}

Similarly, when $c$ is even, we have that
\begin{align*}
&(\gamma\tau+\delta)^{-\frac{1}{2}}\widetilde{S}(r,c;mc^2A\tau)
\\
&=
\epsilon_1 \exp\left(-2\pi iB\tau\left( \frac{\ell^2(c-2r)^2}{8c^2}+\frac{\ell u(2r-c)}{4c}+\frac{u^2}{8}\right) \right)
\tmu\left(\frac{r(\ell B\tau+n)}{c}+\frac{(uB\tau+v)}{2}, \frac{(\ell B\tau+n)}{2} ; B\tau \right),
\end{align*}
where $\epsilon_1$ is some nonzero constant. 
Again we see that $\widetilde{S}(r,c;mc^2\tau)$
has at most linear exponential growth at $\frac{\alpha}{\gamma}$,
but now the invariant order of the holomorphic part is at least
\begin{align*}
\frac{g^2}{mc^2}\left( -\frac{\ell^2(c-2r)^2}{8c^2} -\frac{\ell u(2r-c)}{4c} - \frac{u^2}{8}
+\widetilde{\nu}\left(\frac{r\ell}{c}+\frac{u}{2},\frac{\ell}{2}\right) \right).
\end{align*}
\end{proof}

\section{Proofs of Main Theorems }

We are now able to prove our main theorems, Theorems 
\ref{TheoremOverpartitionDRankModularity} and \ref{thm:devenMain}, which 
correspond to the cases when $d$ odd and $d$ even, respectively. 

\begin{proof}[Proof of Theorem \ref{TheoremOverpartitionDRankModularity}]
First, we note that Proposition \ref{prop:FullOTransf} together with 
\eqref{eqn:OviaM} and Proposition \ref{PropDOddMdkMaassForm} imply that 
$\widetilde{\mathcal{O}}_{d}(a,b,c;\tau)$ is a harmonic Maass form of weight 
$\frac12$ on a congruence subgroup of $\SLTwo$.  The rest of part 
(\ref{TheoremOverpartitionDRankModularityPartDOddMockModular}) of Theorem 
\ref{TheoremOverpartitionDRankModularity} follows from Proposition 
\ref{PropositionOddDMdRankModularAndNonModular}. 

It is important to note that while $\widetilde{M}_{d,k}(a,b,c;2d^2\tau)$ and 
$\widetilde{\mathcal{O}}_{d}(a,b,c;2d^2\tau)$ are harmonic Maass forms, the functions
$\frac{\eta(2\tau)f_{2d^2,d^2+2dk}}{\eta(\tau)^2}\widetilde{M}_{d,k}(a,b,c;2d^2\tau)$ 
are not harmonic Maass forms as they need not be annihilated by 
the hyperbolic Laplacian.

To prove part (\ref{TheoremOverpartitionDRankModularityPartDOddDissection}), by Propositions \ref{prop:FullOTransf} and \ref{PropositionSTransformationFinal} the functions
$\frac{\eta(\tau)^2}{\eta(2\tau)}\widetilde{\mathcal{O}}_d(a,0,c;\tau)$ and
$\frac{\eta(\tau)^2}{\eta(2\tau)}\widetilde{S}(r,c;2c^2d^2\tau)$ all 
transform like a weight $1$ modular forms on $\Gamma$,
where 
\begin{align*}
	\Gamma &=
	\left\{\begin{array}{cl}
		\Gamma_0(2c^2d^2)\cap\Gamma_1(\normalfont{\mbox{lcm}}(c,d))
		&
		\mbox{ if } c\equiv 1 \pmod{2}
		,\\
		\Gamma_0(4c^2d^2)\cap\Gamma_1(\normalfont{\mbox{lcm}}(c,d))
		&
		\mbox{ if } c\equiv 0 \pmod{2}
		.		
	\end{array}
	\right.
\end{align*}	
If the equality
\begin{align}\label{eqn:equality_odd}
	\frac{(1+\z{c}{a})}{(1-\z{c}{a})}\mathcal{O}_d(\z{c}{a};q)
	-
	\sum_{r=1}^{c-1}(-1)^r\z{c}{-2ar}S(r,c;2c^2d^2\tau)
	&=		
	\widetilde{\mathcal{O}}_d(a,0,c;\tau)
	-
	\sum_{r=1}^{c-1}(-1)^r\z{c}{-2ar}\widetilde{S}(r,c;2c^2d^2\tau)
\end{align}
holds, then the function
\begin{align*}
	\frac{(1+\z{c}{a})}{(1-\z{c}{a})}\mathcal{O}_d(\z{c}{a};q)
	-
	\sum_{r=1}^{c-1}(-1)^r\z{c}{-2ar}S(r,c;2c^2d^2\tau)
\end{align*}
is a holomorphic harmonic Maass form, and as such is a modular form and part (\ref{TheoremOverpartitionDRankModularityPartDOddDissection}) follows.

To prove \eqref{eqn:equality_odd} in the case when $c$ is odd, we use Proposition 
\ref{prop:Rdissection} to find that
\begin{align*}
	\frac{(1+\z{c}{a})}{(1-\z{c}{a})}\mathcal{O}_d(\z{c}{a};q)
	&=
		\widetilde{\mathcal{O}}_d(a,0,c;\tau)
		-
		\left(
			\z{c}{a}q^{-\frac{d^2}{4}}
			R\left(\tfrac{2a}{c}-d^2\tau	;2d^2\tau\right)		
			-
			1
		\right)
	\\
	&=
		\widetilde{\mathcal{O}}_d(a,0,c;\tau) 
		+ 
		1
		-
		\sum_{r=0}^{c-1}
		(-1)^r \z{c}{-2ar} q^{ - \tfrac{d^2(c-2r)^2}{4}  }
		R\left( (2cd^2r-c^2d^2)\tau ; 2c^2d^2\tau \right)
.
\end{align*}
Using \eqref{eqn:prop1.9Z}, we deduce that
$R\left( -\tfrac{\tau}{2}; \tau \right) = q^{\frac{1}{8}}$.
We then have
\begin{align*}
	&\frac{(1+\z{c}{a})}{(1-\z{c}{a})}\mathcal{O}_d(\z{c}{a};q)
	\\
	&=
		\widetilde{\mathcal{O}}_d(a,0,c;\tau) 
		- 
		\sum_{r=1}^{c-1}
		(-1)^r \z{c}{-2ar} q^{ - \tfrac{d^2(c-2r)^2}{4}  }
		R\left( (2cd^2r-c^2d^2)\tau ; 2c^2d^2\tau \right)
	\\	
	&=
		\widetilde{\mathcal{O}}_d(a,0,c;\tau) 
		- 
		\sum_{r=1}^{c-1}
		(-1)^r \z{c}{-2ar} q^{ - \tfrac{d^2(c-2r)^2}{4}  }
		\left(		
			2i\mu\left(2cd^2r\tau, c^2d^2\tau; 2c^2d^2 \right)
			-
			2i\tmu\left(2cd^2r\tau, c^2d^2\tau; 2c^2d^2\tau\right)		
		\right)
	\\	
	&=
		\widetilde{\mathcal{O}}_d(a,0,c;\tau) 
		+ 
		\sum_{r=1}^{c-1}
		(-1)^r \z{c}{-2ar} 
		\left(		
			S(r,c;2c^2d^2\tau)
			-
			\widetilde{S}(r,c;2c^2d^2\tau)
		\right)
	,
\end{align*}
so that
\begin{align*}
	\frac{(1+\z{c}{a})}{(1-\z{c}{a})}\mathcal{O}_d(\z{c}{a};q)
	-
	\sum_{r=1}^{c-1}(-1)^r\z{c}{-2ar}S(r,c;2c^2d^2\tau)
	&=		
	\widetilde{\mathcal{O}}_d(a,0,c;\tau)
	-
	\sum_{r=1}^{c-1}(-1)^r\z{c}{-2ar}\widetilde{S}(r,c;2c^2d^2\tau)
.
\end{align*}
To prove \eqref{eqn:equality_odd} in the case when $c$ is even, we again use Proposition 
\ref{prop:Rdissection} to find that
\begin{align*}
	\frac{(1+\z{c}{a})}{(1-\z{c}{a})}\mathcal{O}_d(\z{c}{a};q)
	&=
		\widetilde{\mathcal{O}}_d(a,0,c;\tau)
		-
		\left(
			\z{c}{a}q^{-\frac{d^2}{4}}
			R\left(\tfrac{2a}{c}-d^2\tau	;2d^2\tau\right)		
			-
			1
		\right)
	\\
	&=
		\widetilde{\mathcal{O}}_d(a,0,c;\tau) 
		+ 
		1
		-
		i\sum_{r=0}^{c-1}
		(-1)^r \z{c}{-2ar} q^{ - \tfrac{d^2(c-2r)^2}{4}  }
		R\left( (2cd^2r-c^2d^2)\tau + \tfrac{1}{2}; 2c^2d^2\tau \right)
.
\end{align*}
Using \eqref{eqn:prop1.9Z}, we deduce that
$R\left( -\tfrac{\tau}{2} + \tfrac{1}{2}; \tau \right) = -iq^{\frac{1}{8}}$.
We then have
\begin{align*}
	&\frac{(1+\z{c}{a})}{(1-\z{c}{a})}\mathcal{O}_d(\z{c}{a};q)
	\\
	&=
		\widetilde{\mathcal{O}}_d(a,0,c;\tau) 
		- 
		i\sum_{r=1}^{c-1}
		(-1)^r \z{c}{-2ar} q^{ - \tfrac{d^2(c-2r)^2}{4}  }
		R\left( (2cd^2r-c^2d^2)\tau + \tfrac{1}{2}; 2c^2d^2\tau \right)
	\\	
	&=
		\widetilde{\mathcal{O}}_d(a,0,c;\tau) 
		\\&\quad
		- 
		i\sum_{r=1}^{c-1}
		(-1)^r \z{c}{-2ar} q^{ - \tfrac{d^2(c-2r)^2}{4}  }
		\left(		
			2i\mu\left(2cd^2r\tau + \tfrac{1}{2}, c^2d^2\tau; 2c^2d^2 \right)
			-
			2i\tmu\left(2cd^2r\tau + \tfrac{1}{2}, c^2d^2\tau; 2c^2d^2\tau\right)		
		\right)
	\\	
	&=
		\widetilde{\mathcal{O}}_d(a,0,c;\tau) 
		+ 
		\sum_{r=1}^{c-1}
		(-1)^r \z{c}{-2ar} 
		\left(		
			S(r,c;2c^2d^2\tau)
			-
			\widetilde{S}(r,c;2c^2d^2\tau)
		\right)
	,
\end{align*}
so that
\begin{align*}
	\frac{(1+\z{c}{a})}{(1-\z{c}{a})}\mathcal{O}_d(\z{c}{a};q)
	-
	\sum_{r=1}^{c-1}(-1)^r\z{c}{-2ar}S(r,c;2c^2d^2\tau)
	&=		
	\widetilde{\mathcal{O}}_d(a,0,c;\tau)
	-
	\sum_{r=1}^{c-1}(-1)^r\z{c}{-2ar}\widetilde{S}(r,c;2c^2d^2\tau)
.
\end{align*}
\end{proof}

Next we prove Theorem \ref{thm:devenMain}.  

\begin{proof}[Proof of Theorem \ref{thm:devenMain}]
First, we note that Proposition \ref{prop:FullOTransf} together with \eqref{eqn:OviaM} and Proposition \ref{prop:devenMaass} imply that $\widetilde{\mathcal{O}}_{d}(a,b,c;\tau)$ is a harmonic Maass form of weight $\frac12$ on a congruence subgroup of $\SLTwo$.  The rest of part (1) of Theorem \ref{thm:devenMain} follows from Propositions \ref{prop:Oerr2} and \ref{prop:Oerr4}. 

To prove part (2), we observe that the functions $\frac{\eta(\tau)^2}{\eta(2\tau)} \widetilde{O}_d(a,0,c;\tau)$,
$\frac{\eta(\tau)^2}{\eta(2\tau)} \widetilde{S}(r,c;\tfrac{c^2d^2}{2}\tau)$ for $d\equiv 2\pmod{4}$,
and $\frac{\eta(\tau)^2}{\eta(2\tau)} \widetilde{S}(2r,2c;\tfrac{c^2d^2}{2}\tau)$ for $d\equiv 0\pmod{4}$
all transform like weight $1$ modular forms on $\G_{c,d}$, where
\[
\G_{c,d} :=
\begin{cases}
\G_0(c^2d^2) \cap \G_1({\rm lcm}(c,2d)) & \text{ if } d\equiv 2 \pmod 4,\\
\G_0( {\rm lcm}(2,c)^2 d^2) \cap \G_1({\rm lcm}(2c,2d)) & \text{ if } d\equiv 0\pmod4. 
\end{cases}
\]
This follows from letting $b=0$ in the groups appearing in Proposition \ref{prop:FullOTransf}, and considering the groups in Proposition \ref{PropositionSTransformationFinal}, case by case (when $d\equiv 2\pmod{4}$ we use $m=\frac{d^2}{2}$ and when
$d\equiv 0\pmod{4}$ we use $r\rightarrow 2r$, $c\rightarrow 2c$, and $m=\frac{d^2}{8}$).  

If when $d\equiv 2\pmod 4$,
\begin{equation}\label{eqn:equality_2}
\frac{(1+\zeta_c^a)}{(1-\zeta_c^a)}\mathcal{O}_d(\zeta_c^a;q) - \sum_{r=1}^{c-1}(-1)^r\zeta_c^{-ar} S(r,c;\tfrac{c^2d^2}{2}\tau)
= \widetilde{\mathcal{O}}_d(a,0,c;\tau) - \sum_{r=1}^{c-1}(-1)^r\zeta_c^{-ar} \widetilde{S}(r,c;\tfrac{c^2d^2}{2}\tau),
\end{equation}
and when $d\equiv 0 \pmod 4$,
\begin{equation}\label{eqn:equality_0}
\frac{(1+\zeta_c^a)}{(1-\zeta_c^a)}\mathcal{O}_d(\zeta_c^a;q) - \sum_{r=1}^{c-1}\zeta_c^{-ar} S(2r,2c;\tfrac{c^2d^2}{2}\tau)
= \widetilde{\mathcal{O}}_d(a,0,c;\tau) - \sum_{r=1}^{c-1}\zeta_c^{-ar} \widetilde{S}(2r,2c;\tfrac{c^2d^2}{2}\tau),
\end{equation}
then the functions
\begin{align*}
	\frac{(1+\z{c}{a})}{(1-\z{c}{a})}\mathcal{O}_d(\z{c}{a};q)
	- 
	\sum_{r=1}^{c-1}(-1)^r\zeta_c^{-ar} S(r,c;\tfrac{c^2d^2}{2}\tau)
\end{align*}
when $d\equiv 2\pmod 4$, and 
\begin{align*}
	\frac{(1+\z{c}{a})}{(1-\z{c}{a})}\mathcal{O}_d(\z{c}{a};q)
	- 
	\sum_{r=1}^{c-1}\zeta_c^{-ar} S(2r,2c;\tfrac{c^2d^2}{2}\tau)
\end{align*}
when $d\equiv 0 \pmod 4$, are holomorphic harmonic Maass forms, and as such are modular forms and part (2) follows.

By definition, we see that
\[
\frac{(1+\zeta_c^a)}{(1-\zeta_c^a)}\mathcal{O}_d(\zeta_c^a;q) - \widetilde{\mathcal{O}}_d(a,0,c;\tau) 
= 
\begin{cases}
	1 + \zeta_{2c}^aq^{\frac{-d^2}{16}} R(\tfrac{-d^2}{4}\tau + (\tfrac{a}{c}-1); \tfrac{d^2}{2}\tau)
	& \text{ if } d\equiv 2 \!\!\! \pmod 4, \\
	1 + i\zeta_{2c}^aq^{\frac{-d^2}{16}} R(\tfrac{-d^2}{4}\tau + (\tfrac{a}{c}-\tfrac12); \tfrac{d^2}{2}\tau)
	& \text{ if } d\equiv 0 \!\!\! \pmod 4.
\end{cases}
\]
We now apply the dissections in Proposition \ref{prop:Rdissection} with $s=c$ and observe that in each case the $r=0$ term cancels with the summand $1$ above.  We obtain that
\begin{multline*}
\frac{(1+\zeta_c^a)}{(1-\zeta_c^a)}\mathcal{O}_d(\zeta_c^a;q) - \widetilde{\mathcal{O}}_d(a,0,c;\tau) = 
\\
\begin{cases}
	(-1)^a \sum_{r=1}^{c-1} (-1)^r \zeta_c^{-ar} q^{\frac{-d^2}{16}(2r-c)^2} 
	R((\tfrac{cd^2r}{2} - \tfrac{c^2d^2}{4})\tau + (a - c); \tfrac{c^2d^2}{2}\tau)
	& \text{ if } d\equiv 2 \pmod 4, c \text{ odd},
	\\
	-i(-1)^a \sum_{r=1}^{c-1} (-1)^r \zeta_c^{-ar} q^{\frac{-d^2}{16}(2r-c)^2} 
	R((\tfrac{cd^2r}{2} - \tfrac{c^2d^2}{4})\tau + (a - c + \frac12); \tfrac{c^2d^2}{2}\tau)
	& \text{ if }d\equiv 2 \pmod 4, c \text{ even}, 
	\\
	i(-1)^{a+\frac{c-1}{2}} \sum_{r=1}^{c-1}  \zeta_c^{-ar} q^{\frac{-d^2}{16}(2r-c)^2} 
	R((\tfrac{cd^2r}{2} - \tfrac{c^2d^2}{4})\tau + (a - \tfrac{c}2); \tfrac{c^2d^2}{2}\tau) 
	& \text{ if }d\equiv 0 \pmod 4, c \text{ odd}, 
	\\
	-i(-1)^{a+\frac{c}{2}} \sum_{r=1}^{c-1}  \zeta_c^{-ar} q^{\frac{-d^2}{16}(2r-c)^2} 
	R((\tfrac{cd^2r}{2} - \tfrac{c^2d^2}{4})\tau + (a - \tfrac{c}2 + \tfrac12); \tfrac{c^2d^2}{2}\tau) 
	& \text{ if }d\equiv 0 \pmod 4, c \text{ even}.
\end{cases}
\end{multline*}
Now we rewrite the $R$ terms that appear using that $R(u-v;\tau) = 2i(\mu(u,v;\tau) - \widetilde{\mu}(u,v;\tau))$.  We obtain that
\begin{multline*}
\frac{(1+\zeta_c^a)}{(1-\zeta_c^a)}\mathcal{O}_d(\zeta_c^a;q) - \widetilde{\mathcal{O}}_d(a,0,c;\tau) 
\\
=
\begin{cases}
	-2i\sum_{r=1}^{c-1} (-1)^r \zeta_c^{-ar} q^{\frac{-d^2}{16}(2r-c)^2} 
		&\\\quad\times
		\left( \mu(\tfrac{cd^2r\tau}{2}, \tfrac{c^2d^2\tau}{4} ; \tfrac{c^2d^2\tau}{2}) 
		- 
		\widetilde{\mu}(\tfrac{cd^2r\tau}{2}, \tfrac{c^2d^2\tau}{4} ; \tfrac{c^2d^2\tau}{2}) \right)
		& 
		\text{if } d\equiv 2 \pmod{4}, c \text{ odd},
	\\
	2\sum_{r=1}^{c-1} (-1)^r \zeta_c^{-ar} q^{\frac{-d^2}{16}(2r-c)^2}
		&\\\quad\times
		\left( \mu(\tfrac{cd^2r\tau}{2} +\tfrac12 , \tfrac{c^2d^2\tau}{4} ; \tfrac{c^2d^2\tau}{2}) 
		- \widetilde{\mu}(\tfrac{cd^2r\tau}{2} +\tfrac12 , \tfrac{c^2d^2\tau}{4} ; \tfrac{c^2d^2\tau}{2}) \right)
		& \text{if }d\equiv 2 \pmod 4, c \text{ even}, 
	\\
	2\sum_{r=1}^{c-1}  \zeta_c^{-ar} q^{\frac{-d^2}{16}(2r-c)^2}
		&\\\quad\times
		\left( \mu(\tfrac{cd^2r\tau}{2} +\tfrac12 , \tfrac{c^2d^2\tau}{4} ; \tfrac{c^2d^2\tau}{2}) 
		- \widetilde{\mu}(\tfrac{cd^2r\tau}{2} +\tfrac12 , \tfrac{c^2d^2\tau}{4} ; \tfrac{c^2d^2\tau}{2}) \right)
		& \text{if }d\equiv 0 \pmod 4.
\end{cases}
\end{multline*}
Thus using our definitions in \eqref{EqDefinitionOfS} and \eqref{EqDefinitionOfSTilde}, we obtain \eqref{eqn:equality_2} and \eqref{eqn:equality_0} as desired.
\end{proof}

\section{The 3-dissection of $\mathcal{O}_3(\zeta_3;q)$}
To begin we determine a certain class of generalized eta quotients that transform
in the same fashion as $\widetilde{\mathcal{O}}_d(a,0,c;\tau)$. These are
functions of the type that appear on the right hand side of our dissection formulas in Theorem \ref{TheoremOverpartition3Rank3Dissection}.

\begin{theorem}\label{TheoremDOddCOddGetaQuotients}
For positive odd integers $c$ and $d$, define
\begin{align*}
	f(\tau)
	&=
		\eta(2c^2d^2\tau)^{r_0}\eta(2c^2\tau)^{s_0}
		\prod_{k=1}^{cd} f_{2c^2d^2,dck}(\tau)^{r_k}
		\prod_{k=1}^{c} f_{2c^2,ck}(\tau)^{s_k}
	,
\end{align*}
and suppose 
$A=\begin{pmatrix}\alpha & \beta \\ \gamma & \delta \end{pmatrix}\in\Gamma_0(2c^2d^2)\cap\Gamma_1(dc)$.
Then
\begin{align*}
	f(A\tau)
	&=
		(-1)^{(\beta+\frac{(\alpha-1)}{2})T} i^{\alpha\beta T}
		\nu(\mA{2c^2d^2})^{r_0+3R}	\nu(\mA{2c^2})^{s_0+3S}
		(\gamma\tau+\delta)^{\frac{r_0+s_0+R+S}{2}}
		f(\tau)
	,
\end{align*}
where
\begin{align*}
R &= \sum_{k=1}^{cd} r_k, &
S &= \sum_{k=1}^c s_k, &
T &=  \sum_{\substack{k=1 \\ k\text{ odd} }}^{cd} r_k + \sum_{\substack{k=1 \\ k\text{ odd} }}^{c} s_k.
\end{align*}	
In particular, if we have that $r_0\equiv s_0\equiv 0\pmod{3}$ and $T\equiv -1\pmod{4}$, and also that $r_0+s_0+R+S=1$ and $1+2R+2S\equiv -3\pmod{24}$, then
\begin{align*}
	f(A\tau)
	&=
		(-1)^{\beta+\frac{(\alpha-1)}{2}} i^{-\alpha\beta}
		\nu(\mA{2})^{-3}\sqrt{\gamma\tau+\delta}\,
		f(\tau)
.
\end{align*}
\end{theorem}
\begin{proof}
We note that $2dck\alpha \equiv 2dck \pmod{2c^2d^2}$ and $ck\alpha \equiv cd \pmod{2c^2}$, so that by \eqref{EqBiagioliTransform} we have that
\begin{align*}
	f_{2c^2d^2,cdk}(A\tau)
	&=
		(-1)^{\beta k + \Floor{\frac{\alpha k}{2dc}}}
		\exp\left(\frac{\pi i\alpha\beta k^2}{2}\right)
		\nu(\mA{2c^2d^2})^3 \sqrt{\gamma\tau+\delta}
		f_{2c^2d^2,cdk}(\tau)
	,\\
	f_{2c^2,ck}(A\tau)
	&=
		(-1)^{\beta k + \Floor{\frac{\alpha k}{2c}}}
		\exp\left(\frac{\pi i\alpha\beta k^2}{2}\right)
		\nu(\mA{2c^2})^3 \sqrt{\gamma\tau+\delta}
		f_{2c^2,ck}(\tau)
	.
\end{align*}
Next we note that
\begin{align*}
	\Floor{\frac{\alpha k}{2dc}}
	&=
		\Floor{
			\frac{k}{2dc}
			+
			\frac{(\alpha-1)k}{2dc}			
		}
		=
		\frac{(\alpha-1)k}{2dc}
		\equiv
		\frac{(\alpha-1)k}{2} \pmod{2}
		,\\		
	\Floor{\frac{\alpha k}{2c}}
	&=
		\Floor{
			\frac{k}{2c}
			+
			\frac{(\alpha-1)k}{2c}			
		}
		=
		\frac{(\alpha-1)k}{2c}
		\equiv
		\frac{(\alpha-1)k}{2}	\pmod{2}
	.
\end{align*}
Thus we have that
\begin{align*}
	f(A\tau)
	&=
		(-1)^{(\beta + \frac{(\alpha-1)}{2})T}
		i^{\alpha\beta T}
		\nu(\mA{2c^2d^2})^{r_0+3R}
		\nu(\mA{2c^2})^{s_0+3S}
		(\gamma\tau+\delta)^{\frac{r_0+s_0+R+S}{2}}
		f(\tau).
\end{align*}

If we assume the additional conditions on $r_0$, $s_0$, $R$, $S$, and $T$ then
we find that
\begin{align*}
	f(A\tau)
	&=
		(-1)^{\beta + \frac{(\alpha-1)}{2}}
		i^{-\alpha\beta }
		\nu(\mA{2})^{-3}
		\sqrt{\gamma\tau+\delta}
		f(\tau),
\end{align*}
using the fact that $\nu(\mA{2c^2d^2})^{r_0+3R}\nu(\mA{2c^2})^{s_0+3S}
= \nu(\mA{2})^{r_0+s_0+3R+3S} = \nu(\mA{2})^{-3}$.
\end{proof}


We now give the proof of Theorem \ref{TheoremOverpartition3Rank3Dissection}.
\begin{proof}[proof of Theorem \ref{TheoremOverpartition3Rank3Dissection}]
Upon careful inspection, we find that Theorem \ref{TheoremOverpartition3Rank3Dissection}
is equivalent to
\begin{align}
	&\mathcal{O}_3(\z{3}{};q)-3iS_3(1,3;162\tau)
	\nonumber \\	
	&=	
		\frac{\eta(18\tau)^3}{f_{18;3,3,6,9}}\left(
			-\frac{3f_{162;27,36,63,81}(\tau)}{f_{162;9,72}(\tau)}
			-\frac{6f_{162;27,36,45,81}(\tau)}{f_{162;9,72}(\tau)}
			+\frac{4f_{162;18,36,54,72,81}(\tau)}{f_{162;9,45,63}(\tau)}
			+\frac{3f_{162;18,27,45,81}(\tau)}{f_{162;9,72}(\tau)}
			\right. \nonumber \\&\quad
			-\frac{2f_{162;18,27,36,54,72}(\tau)}{f_{162;9,45,63}(\tau)}
			+\frac{12f_{162;27,45,54}(\tau)}{f_{162;18}(\tau)}
			-\frac{6f_{162;27,36,81}(\tau)}{f_{162;18}(\tau)}
			+\frac{12f_{162;45,63,72}(\tau)}{f_{162;54}(\tau)}
			-\frac{6f_{162;27,72,81}(\tau)}{f_{162;36}(\tau)}
			\nonumber \\&\quad			
			+\frac{12f_{162;27,45,63,72}(\tau)}{f_{162;54,81}(\tau)}		
			+\frac{6f_{162;18,45,63,81}(\tau)}{f_{162;9,72}(\tau)}
			-\frac{4f_{162;18,36,54,72}(\tau)}{f_{162;9,45}(\tau)}
			-\frac{2f_{162;18,36,54,72}(\tau)}{f_{162;9,63}(\tau)}
			+\frac{6f_{162;27,36,63}(\tau)}{f_{162;18}(\tau)}
			\nonumber \\&\quad\left.			
			+\frac{2f_{162;45,54,63}(\tau)}{f_{162;72}(\tau)}
			+\frac{4f_{162;27,45,72}(\tau)}{f_{162;36}(\tau)}
			-\frac{2f_{162;27,36,63}(\tau)}{f_{162;72}(\tau)}
		\right)
		+
		\frac{\eta(18\tau)^3}{f_{18;3,6,9,9}}\left(
			\frac{12f_{162;18,45,63,81}(\tau)}{f_{162;9,72}(\tau)}
			\right. \nonumber \\&\quad			
			-\frac{8f_{162;18,36,54,72}(\tau)}{f_{162;9,45}(\tau)}
			-\frac{4f_{162;18,36,54,72}(\tau)}{f_{162;9,63}(\tau)}
			+\frac{12f_{162;27,36,63}(\tau)}{f_{162;18}(\tau)}
			+\frac{4f_{162;45,54,63}(\tau)}{f_{162;72}(\tau)}
			+\frac{8f_{162;27,45,72}(\tau)}{f_{162;36}(\tau)}
			\nonumber \\&\quad\left.			
			-\frac{4f_{162;27,36,63}(\tau)}{f_{162;72}(\tau)}
		\right) \label{eqn:big_thing}
,
\end{align}
where 
\begin{align*}
	f_{N;\rho_1,\dots,\rho_k}(\tau) 
	&= 
		f_{N,\rho_1}(\tau)\dots f_{N,\rho_k}(\tau)
	.
\end{align*}
However, since $S_3(2,3;162\tau)=S_3(1,3;162\tau)$, we find that the lefthand side of \eqref{eqn:big_thing} is a modular form by Theorem \ref{TheoremOverpartitionDRankModularity}
part (\ref{TheoremOverpartitionDRankModularityPartDOddDissection}).
By Theorem \ref{TheoremDOddCOddGetaQuotients} we find the righthand side
to be a modular form as well. In particular both are weight $\frac{1}{2}$ and have the
same multiplier on $\Gamma_0(162)\cap\Gamma_1(9)$. We then divide both sides
by $\frac{\eta(18\tau)^3f_{162;27,36,63,81}(\tau)}{f_{18;3,3,6,9}f_{162;9,72}(\tau)}$
to obtain an identity between two modular forms of weight zero on 
$\Gamma_0(162)\cap\Gamma_1(9)$. We let $LHS$ denote the lefthand side and
$RHS$ the righthand side of this resulting identity.
We verify that $LHS=RHS$ according to the valence formula \eqref{eqn:valence}.

A complete set of inequivalent cusps, along with their widths, 
for $\Gamma_0(162)\cap\Gamma_1(9)$ is
$$
\renewcommand{\arraystretch}{1.1}	
\setlength\arraycolsep{3pt}
	\begin{array}{l|cccccccccccccccccccc}
		\mbox{cusp}&
			0,& \tfrac{1}{18},& \tfrac{2}{33},& \tfrac{1}{16},& \tfrac{2}{31},& 
			\tfrac{1}{15},& \tfrac{2}{29},& \tfrac{5}{72},& \tfrac{1}{14},& \tfrac{1}{12},& 
			\tfrac{2}{21},& \tfrac{7}{72},& \tfrac{1}{10},& \tfrac{1}{9},& \tfrac{7}{54},& 
			\tfrac{2}{15},& \tfrac{5}{36},& \tfrac{1}{6},& \tfrac{8}{45},& \tfrac{17}{90}, 
		\\		
		\mbox{width}&				
			162,& 1,& 18,& 81,& 162,& 
			18,& 162,& 1,& 81,& 9,& 
			18,& 1,& 81,& 2,& 1,& 
			18,& 1,& 9,& 2,& 1, 
		\\			
		\hline
		\mbox{cusp}&
			\tfrac{7}{36},& \tfrac{2}{9},& \tfrac{7}{30},& \tfrac{13}{54},& \tfrac{11}{45},& 
			\tfrac{7}{27}&, \tfrac{43}{162}&, \tfrac{17}{63},& \tfrac{5}{18},& \tfrac{13}{45},& 
			\tfrac{7}{24},& \tfrac{8}{27},& \tfrac{11}{36},& \tfrac{17}{54},& \tfrac{1}{3},& 
			\tfrac{13}{36},& \tfrac{23}{63},& \tfrac{10}{27},& \tfrac{31}{81},& \tfrac{7}{18},
		\\
		\mbox{width}&
			1,& 2,& 9,& 1,& 2,& 
			2,& 1,& 2,& 1,& 2,& 
			9,& 2,& 1,& 1,& 18,& 
			1,& 2,& 2,& 2,& 1, 
		\\
		\hline			
		\mbox{cusp}&
			\tfrac{5}{12},& \tfrac{19}{45},& \tfrac{4}{9},& \tfrac{37}{81},& \tfrac{29}{63},& 
			\tfrac{38}{81},& \tfrac{17}{36},& \tfrac{77}{162},& \tfrac{13}{27},& \tfrac{14}{27},& 
			\tfrac{19}{36},& \tfrac{34}{63},& \tfrac{5}{9},& \tfrac{37}{63},& \tfrac{11}{18},& 
			\tfrac{40}{63},& \tfrac{2}{3},& \tfrac{37}{54},& \tfrac{32}{45},&\tfrac{13}{18}, 			
		\\
		\mbox{width}&
			9,& 2,& 2,& 2,& 2,& 
			2,& 1,& 1,& 2,& 2,& 
			1,& 2,& 2,& 2,& 1,& 
			2,& 18,& 1,& 2,& 1, 
		\\
		\hline			
		\mbox{cusp}&
			\tfrac{20}{27},& \tfrac{34}{45},& \tfrac{41}{54},& \tfrac{7}{9},& \tfrac{73}{90},& 
			\tfrac{5}{6},& \tfrac{47}{54},& \tfrac{8}{9},& \tfrac{65}{72},& \tfrac{67}{72},& 
			\tfrac{17}{18},& \infty&
			&&&&&&&
		\\
		\mbox{width}&
			2,& 2,& 1,& 2,& 1,& 
			9,&	1,& 2,& 1,& 1,& 
			1,& 1&
			&&&&&&&
	\end{array}
$$
We let $\mathcal{D}$ denote these cusps along with a fundamental region of the 
action of $\Gamma$.

We note $LHS-RHS$ has no poles in $\mathcal{H}$, but it may have zeros in $\mathcal{H}$. 
We take a lower bound on the
orders at the non-infinite cusps by taking the minimum order of each of the 
individual summands, which we compute with
Propositions 
\ref{PropDOddMTildeInvariantOrder},
\ref{PropBiagioliInvariantOrder},
\ref{PropSMassFormAndInvariantOrder}, and
\ref{PropDOddPInvariantOrder}.

This lower bound yields
\begin{align*}
	\sum_{\zeta\in\mathcal{D}} ORD_{\Gamma}\left( LHS-RHS ; \zeta \right)
	\ge	
	ORD_{\Gamma}\left( LHS-RHS   ,\infty\right)
	-119.
\end{align*}
However, we can expand $LHS-RHS$ as a series in $q$
and find the coefficients of $LHS-RHS$ are zero to at least 
$q^{120}$. Thus
\begin{align*}
	\sum_{\zeta\in\mathcal{D}} ORD_{\Gamma}\left(LHS-RHS;\zeta\right)
	\ge	
	1,
\end{align*}
and so $LHS-RHS$ must be identically zero by the valence formula \eqref{eqn:valence}.
This establishes Theorem \ref{TheoremOverpartition3Rank3Dissection}

\end{proof}

\section{Additional results}

We give two additional results of interest about $\mathcal{O}_d(z;q)$.
The next theorem shows that as a function of $q$, the odd part of 
$\mathcal{O}_4(z;q)$ is a simple infinite product for all complex $z$.

\begin{theorem}
We have that
\begin{align*}	
	\sum_{n=0}^\infty \sum_{m=-\infty}^\infty
		M_4(m,2n+1)z^mq^{2n+1}	
	&=
	\frac{2q \aqprod{-zq^8,-q^8/z}{q^8}{\infty} \aqprod{q^4}{q^4}{\infty}^4\aqprod{q^8}{q^8}{\infty} }
	{\aqprod{zq^4,q^4/z}{q^4}{\infty} \aqprod{q^2}{q^2}{\infty}^4}
,
\end{align*}
where $M_d(m,n)$ is the coefficient of $z^mq^n$ in $\mathcal{O}_d(z;q)$, as in \eqref{EqDefinitionOfRank}.
\end{theorem}
\begin{proof}
The proof will follow from a generalized Lambert series identity of Chan \cite[Theorem 2.3]{Chan}
along with the fact that the $2$-dissection of $\frac{\aqprod{-q}{q}{\infty}}{\aqprod{q}{q}{\infty}}$
is trivial to deduce.
To begin we let 
\begin{align*}
	\phi(q)
	&=
		\sum_{n=-\infty}^\infty q^{n^2}
	.
\end{align*}
By the Jacobi triple product identity \cite[Theorem 2.8]{Andrews1} we know that
$
	\phi(-q)
	=
	\frac{\aqprod{q}{q}{\infty}}{\aqprod{-q}{q}{\infty}}
$.
By (\ref{EqDefinitionOfRank}) we then see that
\begin{align*}
	\sum_{n=0}^\infty \sum_{m=-\infty}^\infty
		M_4(m,2n+1)z^mq^{2n+1}	
	&=
	\left(\frac{1}{2\phi(-q)}-\frac{1}{2\phi(q)}\right)	
	\left(
		1
		+
		2\sum_{n=1}^\infty
		\frac{(1-z)(1-z^{-1})q^{4n^2+8n}}
		{(1-zq^{8n})(1-z^{-1}q^{8n})}
	\right)	
	\\&\quad+
	\left(\frac{1}{2\phi(-q)}+\frac{1}{2\phi(q)}\right)	
	\left(
		-2\sum_{n=0}^\infty
		\frac{(1-z)(1-z^{-1})q^{4n^2+12n+5}}
		{(1-zq^{8n+4})(1-z^{-1}q^{8n+4})}
	\right)	
	.
\end{align*}
Also by the Jacobi triple product identity we have that
that
\begin{align*}
	\phi(q)
	&=
		\sum_{n=-\infty}^\infty q^{4n^2}
		+
		\sum_{n=-\infty}^\infty q^{4n^2+4n+1}
	=
		\phi(q^4)
		+
		2q\aqprod{-q^8,-q^8,q^8}{q^8}{\infty}	
	.
\end{align*}
Thus
\begin{align*}
	\frac{\phi(q)-\phi(-q)}{2}
	&=
		2q\aqprod{-q^8,-q^8,q^8}{q^8}{\infty}	
	,&
	\frac{\phi(q)+\phi(-q)}{2}
	&=
		\phi(q^4)
	,
\end{align*}
and so
\begin{align*}
	\frac{1}{\phi(-q)}-\frac{1}{\phi(q)}
	&=
		\frac{\phi(q)-\phi(-q)}{\phi(-q)\phi(q)}		
	=
	\frac{ 4q \aqprod{-q^8,-q^8,q^8}{q^8}{\infty} }
		{\phi(-q)\phi(q)}
	,\\
	\frac{1}{\phi(-q)}+\frac{1}{\phi(q)}
	&=
		\frac{\phi(q)+\phi(-q)}{\phi(-q)\phi(q)}		
	=
	\frac{ 2\phi(q^4) }
		{\phi(-q)\phi(q)}
	.
\end{align*}
By \cite[Entry 25(iii)]{Berndt} we have that
$
	\phi(-q)\phi(q)
	=
	\phi(-q^2)^2
	=
		\frac{\aqprod{q^2}{q^2}{\infty}^2}{\aqprod{-q^2}{q^2}{\infty}^2}
$.
We let
\begin{align*}
	F_0(q)
	&=
	\frac{\phi(q^4)}{\phi(q)\phi(-q)}
	=
	\frac{\aqprod{q^8}{q^8}{\infty}^5}{\aqprod{q^2}{q^2}{\infty}^4 \aqprod{q^{16}}{q^{16}}{\infty}^2}
	,\\
	F_1(q)
	&=
	\frac{2q\aqprod{-q^8,-q^8,q^8}{q^8}{\infty}}{\phi(q)\phi(-q)}
	=
	\frac{2q\aqprod{q^4}{q^4}{\infty}^2 \aqprod{q^{16}}{q^{16}}{\infty}^2}
		{\aqprod{q^2}{q^2}{\infty}^4 \aqprod{q^8}{q^8}{\infty}}	
	.
\end{align*}
Thus
\begin{align}\label{eqn:newM_4}
	\sum_{n=0}^\infty \sum_{m=-\infty}^\infty
		M_4(m,2n+1)z^mq^{2n+1}	
	&=
		F_1(q)	
		\sum_{n=-\infty}^\infty
		\frac{(1-z)(1-z^{-1})q^{4n^2+8n}}
		{(1-zq^{8n})(1-z^{-1}q^{8n})}	
	\nonumber \\&\quad
	-
	qF_0(q)	
	\sum_{n=-\infty}^\infty
	\frac{(1-z)(1-z^{-1})q^{4n^2+12n+4}}
	{(1-zq^{8n+4})(1-z^{-1}q^{8n+4})}	
	.
\end{align}

In Theorem 2.3 of \cite{Chan}, we take $r=2$, $s=3$,
$q\mapsto q^8$, $b_1=z$, $b_2=zq^4$, $b_3=bz^2$, $a_1=z^2$
and $a_2=-bz$ to arrive at
\begin{align}
	\label{Eq4RankDissectionChanIdentity1}
	\frac{\aqprod{z^2q^8,z^{-2}q^8,q^8,q^8}{q^8}{\infty} \jacprod{-bz}{q^8}  } 
		{\jacprod{z,zq^{4}, bz^2}{q^8}}
	&=
		-
		\frac{z^{-1}  \jacprod{z,-b}{q^8} }{ \jacprod{q^4,bz}{q^8} }	
		\sum_{n=-\infty}^\infty
		\frac{q^{4n^2+8n} }
		{(1-zq^{8n})(1-z^{-1}q^{8n})}	
		\nonumber\\&\quad		
		-
		\frac{z^{-1}  \jacprod{zq^{-4},-bq^{-4}}{q^8} }{ \jacprod{q^{-4},bzq^{-4}}{q^8} }	
		\sum_{n=-\infty}^\infty
		\frac{q^{4n^2+12n+4} }
		{(1-zq^{8n+4})(1-z^{-1}q^{8n+4})}	
		\nonumber\\&\quad
		- 
		\frac{b \jacprod{b^{-1},-z^{-1}}{q^8} }{ \jacprod{b^{-1}z^{-1},b^{-1}z^{-1}q^{4}}{q^8} }	
		\sum_{n=-\infty}^\infty
		\frac{q^{4n^2+8n} (zb)^n }
		{(1-bz^2q^{8n})(1-bq^{8n})}	
	.
\end{align}
Letting $b\rightarrow 1$ in (\ref{Eq4RankDissectionChanIdentity1}) gives that
\begin{align*}
	\frac{\jacprod{-z}{q^8} \aqprod{q^8}{q^8}{\infty}^2 }
		{\jacprod{z,zq^4}{q^8}(1-z^2)}
	&=
		-
		\frac{z^{-1}  \jacprod{-1}{q^8} }{ \jacprod{q^4}{q^8} }	
		\sum_{n=-\infty}^\infty
		\frac{q^{4n^2+8n} }
		{(1-zq^{8n})(1-z^{-1}q^{8n})}	
		\nonumber\\&\quad		
		-
		\frac{z^{-1}  \jacprod{-q^{-4}}{q^8} }{ \jacprod{q^{-4}}{q^8} }	
		\sum_{n=-\infty}^\infty
		\frac{q^{4n^2+12n+4} }
		{(1-zq^{8n+4})(1-z^{-1}q^{8n+4})}	
		-
		\frac{\jacprod{-z}{q^8} \aqprod{q^8}{q^8}{\infty}^2 }
		{\jacprod{z,zq^4}{q^8}(1-z^2)}
,
\end{align*}
which we simplify to
\begin{align}\label{eqn:brackets}
	\frac{2\jacprod{-z}{q^8} \aqprod{q^8}{q^8}{\infty}^2 }
		{\jacprod{z}{q^4}(1-z^2)}
	&=
		-
		\frac{z^{-1}  \jacprod{-1}{q^8} }{ \jacprod{q^4}{q^8} }	
		\sum_{n=-\infty}^\infty
		\frac{q^{4n^2+8n} }
		{(1-zq^{8n})(1-z^{-1}q^{8n})}	
		\nonumber\\&\quad		
		+
		\frac{z^{-1}  \jacprod{-q^{4}}{q^8} }{ \jacprod{q^{4}}{q^8} }	
		\sum_{n=-\infty}^\infty
		\frac{q^{4n^2+12n+4} }
		{(1-zq^{8n+4})(1-z^{-1}q^{8n+4})}	
	.
\end{align}
By elementary rearrangements we have that
\begin{align}\label{eqn:products}
	-\frac{z^{-1}\jacprod{-1}{q^8}}{\jacprod{q^4}{q^8}}
	\cdot
	\frac{\jacprod{q^4}{q^8}}{z^{-1}\jacprod{-q^4}{q^8}}
	&=
	\frac{-F_1(q)}{qF_0(q)}
	.
\end{align}
By \eqref{eqn:newM_4}, \eqref{eqn:brackets}, and \eqref{eqn:products}, we have that
\begin{align*}
	&\sum_{n=0}^\infty \sum_{m=-\infty}^\infty
		M_4(m,2n+1)z^mq^{2n+1}	
	\\
	&=
		-qF_0(q)(1-z)(1-z^{-1})
		\left(
			-\frac{F_1(q)}{qF_0(q)}
			\sum_{n=-\infty}^\infty
			\frac{q^{4n^2+8n}}
			{(1-zq^{8n})(1-z^{-1}q^{8n})}	
			+
			\sum_{n=-\infty}^\infty
			\frac{q^{4n^2+12n+4}}
			{(1-zq^{8n+4})(1-z^{-1}q^{8n+4})}
		\right)
	\\
	&=
		-
		\frac{ 2qz F_0(q) (1-z)(1-z^{-1})
			\jacprod{-z,q^4}{q^8} \aqprod{q^8}{q^8}{\infty}^2 }
		{\jacprod{z}{q^4}\jacprod{-q^4}{q^8}(1-z^2)}
	\\
	&=	
		\frac{2q \aqprod{-zq^8,-q^8/z}{q^8}{\infty} \aqprod{q^4}{q^4}{\infty}^4\aqprod{q^8}{q^8}{\infty} }
		{\aqprod{zq^4,q^4/z}{q^4}{\infty} \aqprod{q^2}{q^2}{\infty}^4}
	.
\end{align*}

\end{proof}

Lastly, we note that certain linear combinations of $\mathcal{O}_d(z;q)$
will always be modular, rather than mock modular.
\begin{theorem}
The following functions are modular forms of weight 
$\frac{1}{2}$ on some congruence subgroup of $\SLTwo$:
\begin{enumerate}
\item
$\frac{(1+\z{c}{a}q^{\frac{b}{c}})}{(1-\z{c}{a}q^{\frac{b}{c}})} q^{-\frac{b^2}{c^2d^2}} 
\mathcal{O}_d(\z{c}{a}q^{\frac{b}{c}};q)
-
\frac{(1+\z{c}{2a}q^{\frac{2b}{c}})}{(1-\z{c}{2a}q^{\frac{2b}{c}})}q^{-\frac{b^2}{c^2d^2}} 
\mathcal{O}_{2d}(\z{c}{2a}q^{\frac{2b}{c}};q)$
when $d$ is odd,
\item
$\frac{(1+\z{c}{a}q^{\frac{b}{c}})}{(1-\z{c}{a}q^{\frac{b}{c}})} q^{-\frac{b^2}{c^2d^2}} 
\mathcal{O}_d(\z{c}{a}q^{\frac{b}{c}};q)
-
\frac{(1-\z{c}{a}q^{\frac{b}{c}})}{(1+\z{c}{a}q^{\frac{b}{c}})}q^{-\frac{b^2}{c^2d^2}} 
\mathcal{O}_{2d}\left(-\z{c}{a}q^{\frac{b}{c}}; q^{\frac{1}{4}}\right)$
when $d\equiv 2\pmod{4}$, 
\item
$\frac{(1+\z{c}{a}q^{\frac{b}{c}})}{(1-\z{c}{a}q^{\frac{b}{c}})} q^{-\frac{b^2}{c^2d^2}} 
\mathcal{O}_d(\z{c}{a}q^{\frac{b}{c}};q)
-
\frac{(1+\z{c}{a}q^{\frac{b}{c}})}{(1-\z{c}{a}q^{\frac{b}{c}})}q^{-\frac{b^2}{c^2d^2}} 
\mathcal{O}_{2d}\left(\z{c}{a}q^{\frac{b}{c}};q^{\frac{1}{4}}\right)$
when $d\equiv 0\pmod{4}$.
\end{enumerate}
\end{theorem}
\begin{proof}
\sloppy
We note that the only contribution to the non-holomorphic part of
$\widetilde{\mathcal{O}}_d(a,b,c;\tau)$ is from
$\z{c}{a}q^{-\frac{d^2}{4}-\frac{b^2}{c^2d^2}+\frac{b}{c}}
R\left(\frac{2a}{c}+\left(\frac{2b}{c}-d^2\right)\tau;2d^2\tau  \right)$
when $d$ is odd,
$-\z{2c}{a}q^{-\frac{d^2}{16}-\frac{b^2}{c^2d^2}+\frac{b}{2c}}
R\left(\left(\frac{a}{c}-1\right)+\left(\frac{b}{c}-\frac{d^2}{4}\right)\tau;\frac{d^2\tau}{2}  \right)$
when $d\equiv 2\pmod{4}$, and
$-i\z{2c}{a}q^{-\frac{d^2}{16}-\frac{b^2}{c^2d^2}+\frac{b}{2c}}
R\left(\left(\frac{a}{c}-\frac{1}{2}\right)+\left(\frac{b}{c}-\frac{d^2}{4}\right)\tau;\frac{d^2\tau}{2}  \right)$
when $d\equiv 0\pmod{4}$.

\fussy
The $R(u;\tau)$ terms cancel trivially to give
\begin{align*}
\widetilde{\mathcal{O}}_d(a,b,c;\tau)-\widetilde{\mathcal{O}}_{2d}(2a,2b,c;\tau)
&=
\frac{(1+\z{c}{a}q^{\frac{b}{c}})}{(1-\z{c}{a}q^{\frac{b}{c}})} q^{-\frac{b^2}{c^2d^2}} 
\mathcal{O}_d(\z{c}{a}q^{\frac{b}{c}};q)
-
\frac{(1+\z{c}{2a}q^{\frac{2b}{c}})}{(1-\z{c}{2a}q^{\frac{2b}{c}})}q^{-\frac{b^2}{c^2d^2}} 
\mathcal{O}_{2d}(\z{c}{2a}q^{\frac{2b}{c}};q)
\end{align*}
when $d$ is odd,
\begin{align*}
\widetilde{\mathcal{O}}_d(a,b,c;\tau)-\widetilde{\mathcal{O}}_{2d}\left(2a+c,8b,2c;\frac{\tau}{4}\right)
&=
\frac{(1+\z{c}{a}q^{\frac{b}{c}})}{(1-\z{c}{a}q^{\frac{b}{c}})} q^{-\frac{b^2}{c^2d^2}} 
\mathcal{O}_d(\z{c}{a}q^{\frac{b}{c}};q)
-
\frac{(1-\z{c}{a}q^{\frac{b}{c}})}{(1+\z{c}{a}q^{\frac{b}{c}})}q^{-\frac{b^2}{c^2d^2}} 
\mathcal{O}_{2d}\left(-\z{c}{a}q^{\frac{b}{c}};q^{\frac{1}{4}}\right)
\end{align*}
when $d\equiv 2\pmod{4}$, and
\begin{align*}
\widetilde{\mathcal{O}}_d(a,b,c;\tau)-\widetilde{\mathcal{O}}_{2d}\left(2a,8b,2c;\frac{\tau}{4}\right)
&=
\frac{(1+\z{c}{a}q^{\frac{b}{c}})}{(1-\z{c}{a}q^{\frac{b}{c}})} q^{-\frac{b^2}{c^2d^2}} 
\mathcal{O}_d(\z{c}{a}q^{\frac{b}{c}};q)
-
\frac{(1+\z{c}{a}q^{\frac{b}{c}})}{(1-\z{c}{a}q^{\frac{b}{c}})}q^{-\frac{b^2}{c^2d^2}} 
\mathcal{O}_{2d}\left(\z{c}{a}q^{\frac{b}{c}};q^{\frac{1}{4}}\right)
\end{align*}
when $d\equiv 0\pmod{4}$.

As such, the three functions in the statement of the theorem are
holomorphic harmonic Maass forms of weight $\frac{1}{2}$.  As such they are modular forms.
\end{proof}

\section*{Acknowledgements}

We would like to thank the Institute for Computational and Experimental Research in Mathematics (ICERM) for the special semester program on Computational Aspects of the Langlands Program where this work was initiated.

\end{document}